\documentclass[10pt]{article}
\usepackage[english]{babel}
\usepackage{url}
\usepackage{inputenc}
\usepackage{amsfonts}
\usepackage{amsmath}
\usepackage{empheq}
\usepackage{graphicx}
\graphicspath{{images/}}
\usepackage{parskip}
\usepackage{fancyhdr}
\usepackage{bbold}
\usepackage{amssymb}
\usepackage{vmargin}
\usepackage{array}
\usepackage{amsthm}
\usepackage[ruled,vlined]{algorithm2e}
\usepackage{hyperref}
\usepackage{caption}
\usepackage{float}
\usepackage{fourier-orns}
\usepackage{listings}
\usepackage{epsfig}
\usepackage{amsmath} 
\usepackage[scr=boondoxo,scrscaled=1.05]{mathalfa}

\usepackage[shortlabels]{enumitem}
\mathtoolsset{showonlyrefs=true}

\usepackage[bbgreekl]{mathbbol}

\newtheorem{assumption}{Assumption}
\newtheorem{theorem}{Theorem}[section]
\newtheorem{corollary}{Corollary}[section]
\newtheorem{lemma}[theorem]{Lemma}
\newtheorem{remark}{Remark}[section]
\theoremstyle{definition}

\newtheorem{prop}{Proposition}[section]
\usepackage{listings}

\newcommand{\cha}{\mathbb{1}_{\Omega_0}}

\newcommand{\R}{\mathbb{R}}
\newcommand{\om}{{\Omega_{0}}}

\newcommand{\bydef}{\stackrel{\mbox{\tiny\textnormal{\raisebox{0ex}[0ex][0ex]{def}}}}{=}}

\newcommand{\bU}{\mathbf{U}}
\newcommand{\bgam}{\text{\boldmath{$\gamma$}}}
\newcommand{\bGam}{\text{\boldmath{$\Gamma$}}}

\newcommand{\bpi}{\text{\boldmath{$\pi$}}}

\usepackage[dvipsnames]{xcolor}

\makeatletter
\newcommand\subsubsubsection{\@startsection{paragraph}{4}{\z@}{-2.5ex\@plus -1ex \@minus -.25ex}{1.25ex \@plus .25ex}{\normalfont\normalsize\bfseries}}
\newcommand\subsubsubsubsection{\@startsection{subparagraph}{5}{\z@}{-2.5ex\@plus -1ex \@minus -.25ex}{1.25ex \@plus .25ex}{\normalfont\normalsize\bfseries}}
\title{The 2D Gray-Scott system of equations: constructive proofs of existence of localized stationary patterns
}
\author{
Matthieu Cadiot
\footnote{McGill University, Department of Mathematics and Statistics, 805 Sherbrooke Street West, Montreal, QC, H3A 0B9, Canada. {\tt matthieu.cadiot@mail.mcgill.ca}}
\and
Dominic Blanco \footnote{McGill University, Department of Mathematics and Statistics, 805 Sherbrooke Street West, Montreal, QC, H3A 0B9, Canada. {\tt dominic.blanco@mail.mcgill.ca}}}

\begin{document}

\maketitle
\begin{abstract}
In this article, we present a comprehensive framework for constructing smooth, localized solutions in systems of semi-linear partial differential equations, with a particular emphasis to the Gray-Scott model.
 Specifically, we construct a natural Hilbert space $\mathcal{H}$ for the study of systems of autonomous semi-linear PDEs, on which products and differential operators are well-defined.  Then, given an approximate solution $\mathbf{u}_0$, we derive a Newton-Kantorovich approach based on the construction of an approximate inverse of the linearization around $\mathbf{u}_0$. In particular, we derive a condition under which we prove the existence of a unique solution in a neighborhood of $\mathbf{u}_0$. Such a condition can be verified thanks to the explicit computation of different upper bounds, for which analytical details are presented. Furthermore, we provide an extra condition under which localized patterns are proven to be the limit of an unbounded branch of (spatially) periodic solutions as the period tends to infinity. We then demonstrate our approach by proving (constructively) the existence of four different localized patterns in the 2D Gray-Scott model. In addition, these solutions are proven to satisfy the $D_4$-symmetry. That is, the symmetry of the square. The algorithmic details to perform the computer-assisted proofs are available on GitHub at \cite{julia_cadiot_blanco_GS}.
\end{abstract}

\begin{center}
\end{center}

\section{Introduction}
In this paper, we investigate the existence of stationary localized solutions to autonomous systems of semi-linear partial differential equations (PDEs). This includes reaction-diffusion models of the form
\begin{equation}\label{eq : stationary reaction diffusion}
   \partial_t \mathbf{u} + \mathbf{D}_1 \Delta \mathbf{u} + \mathbf{D}_2 \mathbf{u}  + \mathbb{G}(\mathbf{u}) =0
\end{equation} where  $\mathbf{D}_1$ is a diagonal matrix of diffusion coefficients, $\mathbf{D}_2$ is a matrix and $\mathbb{G}$ is nonlinear and accounts for local reactions. We examine \eqref{eq : stationary reaction diffusion} on the whole space $\mathbb{R}^m$ and look for stationary  solutions $\mathbf{u} = \mathbf{u}(x)$. The precise class of equations under investigation will be specified  by Assumptions \ref{ass:A(1)} and \ref{ass : LinvG in L1} in Section \ref{sec : general set up system of PDEs}. In particular, this class of PDEs includes
the planar Gray-Scott system of equations :
\begin{equation}\label{eq : gray_scott}
\begin{split}
     u_t &= \delta_u \Delta u - uv^2 + \theta_{F}(1-u) \\
     v_t &= \delta_v \Delta v + uv^2 - (\theta_{F}+\theta_{k})v
     \end{split} ~~ \text{ with } ~~ u = u(x,t),~~ v = v(x,t),~~ x \in \mathbb{R}^2,~ \theta_F, \theta_k > 0
\end{equation}
where $\theta_F$ is the ``feed rate" of $u$ and $\theta_k$ is the ``kill rate" of $v$ (see \cite{sims_gs}). More specifically, $\theta_F$ is the rate at which $u$ is replenished in the chemical reaction. On the other hand, $\theta_k$ is the rate at which $v$ is removed from the reaction. First introduced by Peter Gray and S.K. Scott in \cite{gs_original} as a chemical equation, the Gray-Scott model is a prototypical reaction-diffusion equation of the activator-inhibitor type. It models a pair of reactions featuring a cubic autocatalysis. The equations presented in \eqref{eq : gray_scott} are the resulting reaction-diffusion equations corresponding to the chemical reaction proposed in \cite{gs_original}.  More generally, reaction-diffusion systems are known to exhibit complex dynamics and are of particular interest in the field of pattern formation since they provide a rich variety of patterns (see for instance the following review papers \cite{jason_review_paper},\cite{knobloch_patterns}, and \cite{ward_review}). Specifically to the Gray-Scott model, the pioneered work of  John Pearson in \cite{pearson_gs} allowed to discover and classify a wide variety of patterns, including Turing patterns.  Many works have followed his classification on both Turing and localized patterns (moving and stationary). These include \cite{mcgough_riley_gs}, \cite{munafo_gs}, \cite{autosolitons_2D}, \cite{pearson_spots}, \cite{turing_3d}, and \cite{ueyama_numerics} all of which found and classified additional numerical solutions to \eqref{eq : gray_scott}.

In this paper, we focus on stationary localized patterns. By localized pattern, we mean spatially confined structures that approach a constant solution at infinity. Such solutions are ubiquitous in reaction-diffusion systems and represent a fundamental subject of study in such equations (cf. \cite{gly_1d,jason_review_paper,  knobloch_patterns, sandstede_paper_three}). Note that in general, the change of variable $u \to u - \alpha$ allows to trivially switch from a solution approaching $\alpha$ at infinity to a solution vanishing at infinity. Without loss of generality, we consider solutions going to zero at infinity. In particular, we introduce the following change of variables.
Define $\lambda_2 = \frac{\theta_F}{(\theta_F+\theta_k)^2}, \lambda_1 = (\theta_F+\theta_k)\frac{\delta_v}{\delta_u}$, $v = (\theta_F+\theta_k) u_1\left(\frac{\theta_F+\theta_k}{\sqrt{\delta_u}}x\right)$, and $u = u_2\left(\frac{\theta_F+\theta_k}{\sqrt{\delta_u}}x\right) - \lambda_1 u_1\left(\frac{\theta_F+\theta_k}{\sqrt{\delta_u}}x\right) + 1$ and obtain
\begin{equation}\label{eq : gray_scott cov}
\begin{split}
     \lambda_1 \Delta u_1 + (u_2 + 1 - \lambda_1 u_1)u_1^2 -u_1 &= 0 \\
     \Delta u_2 - \lambda_2 u_2 + (\lambda_2 \lambda_1 - 1)u_1 &= 0,
\end{split}
\end{equation}
where $\lambda_1, \lambda_2 >0$. We look for solutions $(u_1,u_2)$ of \eqref{eq : gray_scott cov} where $u_1(x), u_2(x) \to 0$ as $|x| \to \infty$.
Equivalently, denoting $\mathbf{u} \bydef (u_1,u_2)$ and defining $\mathbb{F}(\mathbf{u}) \bydef \mathbb{L}\mathbf{u} + \mathbb{G}(\mathbf{u})$, 
where
\begin{equation}\label{definition_of_L}
    \mathbb{L} \bydef \begin{bmatrix}
        \mathbb{L}_{11} & 0 \\
        \mathbb{L}_{21} & \mathbb{L}_{22}
    \end{bmatrix} \bydef \begin{bmatrix}
        \lambda_1 \Delta - I & 0 \\
        (\lambda_2 \lambda_1 - 1)I & \Delta - \lambda_2 I
    \end{bmatrix}, 
\end{equation}
\begin{equation}\label{definition_of_G}
    \mathbb{G}(\mathbf{u}) \bydef \begin{bmatrix}
         (u_2 + 1 - \lambda_1 u_1)u_1^2 \\ 0
    \end{bmatrix},
\end{equation} 
we look for zeros of $\mathbb{F}$ vanishing at infinity.
Note that \eqref{eq : gray_scott cov} has the particular advantage of possessing a linear equation. This point will simplify parts of the analysis, both theoretically and numerically. Before presenting our methodology, we provide a brief summary of the existing literature review on localized patterns in the Gray-Scott model.

\par In the one dimensional equivalent of \eqref{eq : gray_scott}, diverse results have been obtained in terms of stationary localized solutions. In \cite{gs_alsaadi}, the authors introduce a homotopy parameter to a more general class of reaction diffusion models. This homotopy parameter allows them to transform the Glycolysis model studied in \cite{gly_1d} into Gray-Scott. By performing numerical continuation in this parameter, the authors observe the more simple and understood dynamics of the Glycolysis model transition to the more complex ones of Gray-Scott. In \cite{doleman_pulse}, in the region where a parameter $0 < \delta^2 \ll 1$, the authors used asymptotic scalings on the other parameters to prove the existence of single pulse solutions. More specifically, they consider a scaled system in which the variables and parameters are $O(1)$. Then, by using geometric arguments, their scalings enable them to deduce correct asymptotic scalings for the variables and parameters in order for the patterns to exist. In particular, each of these scalings depend on the small parameter $\delta$. Their main results are then proofs of existence of single-pulse and multi-pulse stationary states on the infinite line. They also construct these solutions for the rescaled parameter values.  Without relying on asymptotic scalings, the authors of \cite{exact_sol1D_Hale} found an exact solution to \eqref{eq : gray_scott cov} in the parameter regime $\lambda_2 \lambda_1 =1 $. Here, \eqref{eq : gray_scott cov} reduces to a scalar ODE which can be solved analytically, providing an exact solution on the infinite line. Additionally, the authors provided numerical evidence for the existence of pulse patterns away from $\lambda_2\lambda_1 = 1$ by performing a continuation method. Using a computer-assisted approach, the authors in \cite{jay_jp_gs} verified such an observation. Indeed, they used a rigorous continuation method, starting at $\lambda_1\lambda_2=1$, and proved the existence of a branch of spike solutions to \eqref{eq : gray_scott}.

In the planar case, the known results on localized patterns are restricted to radially symmetric solutions.  For instance, existence of such solutions has been shown analytically through the use of asymptotics by Wei in \cite{wei_spike_2d_unbounded}. Here, the author considers a specific parameter regime. More specifically, their regime provides that one of the component of the solution is almost constant, simplifying the analysis. In particular, Wei was able to reduce \eqref{eq : gray_scott} to a single equation and proved the existence of a radially symmetric solution. In the same regime, an alternative methodology was derived in \cite{wei_ring_2d_small_epsilon_unbounded} for proving the existence of a ring-like solution. Their approach is based on a Liapunov-Schmidt reduction method combined with asymptotic analysis. The authors begin by solving \eqref{eq : gray_scott} on a disk of radius $R$ with Neumann boundary conditions and show that ring solutions always exist with radius $r_{R} < R$, under some conditions on $R$. They are then able to extend these results to $\mathbb{R}^2$ by deriving a further condition on the parameters under which a ring solution exists. 

 As far as 2D non-radial solutions are concerned, there exists a general numerical method for constructing localized patterns in two component reaction-diffusion systems near a Turing instability. By considering a Galerkin truncation of the system in polar coordinates, the authors of \cite{jason_ring_paper} and \cite{jason_spot_paper} observe the emergence of "spot" and "ring" localized dihedral patterns. By then performing numerical continuation away from the Turing instability, they obtain more localized patterns in the desired system. The method is demonstrated on the Swift Hohenberg scalar PDE, but is applicable to a general class of two component reaction diffusion systems. A similar method can be applied to three component reaction diffusion systems as demonstrated by the authors of \cite{sandstede_paper_three}. As far as Gray-Scott, a rich variety of patterns have been studied on bounded domains. In particular, in \cite{wei_multi_spike_2d} and \cite{wei_asym_spike_2d}, Wei and Winter proved the existence of multi-spot patterns on a smooth bounded domain $\Omega$ in $\mathbb{R}^2$ with Neumann boundary conditions. In both of these studies, the same asymptotic parameter regime described in the previous paragraph is considered. This allows one to reduce the problem to a single equation. Then, using a Green's function, the authors were able to  construct multi-spot solutions to \eqref{eq : gray_scott}. In \cite{wei_multi_spike_2d}, symmetric spot solutions are constructed whereas \cite{wei_asym_spike_2d} focuses on asymmetric spot solutions. In \cite{thesis_localized_but_moving}, the authors proved the existence of various patterns by considering a perturbation argument along with asymptotic approximations, including some multi-spot patterns in 2D. In particular, quasi-equilibrium solutions, which are slowly moving in $O(\epsilon^2)$, are proven to exist. Their solutions evolve slowly and drift with asymptotically slow speed, providing a proof of a so-called moving pattern. Finally, by using a computer-assisted approach, Castelli in \cite{castelli_gs} proved various solutions to \eqref{eq : gray_scott cov} to exist. The method of \cite{castelli_gs} is applicable on a rectangular domain with Neumann or periodic boundary conditions. Furthermore, the author proved that if $\lambda_2 \leq 4$, then there are no non-constant solutions to \eqref{eq : gray_scott cov} for $\lambda_1 > \frac{1}{4}$.

In this manuscript, we propose a general methodology for constructively proving the existence of planar localized patterns in \eqref{eq : gray_scott cov}, without assuming radial symmetry, and away from any specific asymptotic regime. In particular, we prove the existence of a non-radial localized pattern (cf. Theorem \ref{th : proof_of_leaf}). These results are, to the best of our knowledge, new and novel in the pattern formation field.
\par  Recently, \cite{unbounded_domain_cadiot} provided a method for constructively proving the existence of strong solutions to scalar autonomous semi-linear PDEs on $\mathbb{R}^m$. An application to the 1D Kawahara equation on $\mathbb{R}$ was provided in the same paper. Then, the authors applied the theory of \cite{unbounded_domain_cadiot} to prove non-radial localized patterns in the 2D Swift-Hohenberg equation in  \cite{sh_cadiot}  demonstrating the method on a two-dimensional example. One of the objectives of this paper is to extend the analysis done in \cite{unbounded_domain_cadiot} to systems of autonomous semi-linear PDEs.  As we consider systems rather than scalar PDEs, this naturally introduces additional complications, such as the well-definedness of a product, which have to be taken into account.  More specifically, we  derive appropriate assumptions (cf. Assumptions \ref{ass:A(1)} and \ref{ass : LinvG in L1}) for a specific class of system of autonomous semi-linear PDEs, encompassing a large set of reaction-diffusion equations, including the Gray-Scott system. Under such assumptions, we demonstrate how the ideas of \cite{unbounded_domain_cadiot} can be generalized to cover systems of PDEs. Moreover, we had to address some subsidiary problems arising specifically from the Gray-Scott model, leading to the development of technical tools. We now provide a summary of our approach in the following paragraph.

 To begin with, our analysis relies on the ability to (accurately) compute an approximate solution, denoted $\mathbf{u}_0$, with support on a square $\om = (-d,d)^2$. Specifically, we construct a sequence of Fourier coefficients ${\mathbf{U}}_0$, providing a representation of $\mathbf{u}_0$ on $\om$. In order to obtain approximate solutions with different shapes and symmetries, we numerically follow branches using continuation. Indeed, in this paper we focus our attention on solutions possessing the symmetries of a square, which we denote as $D_4$-symmetric. In order to do so, we consider a specific Hilbert space of interest which has these properties. We choose to perform our analysis in the space denoted $\mathcal{H}_{D_4}$, which possesses functions in $H^2(\mathbb{R}^2) \times H^2(\mathbb{R}^2)$ with $D_4$-symmetry. A formal definition is provided in \eqref{H_l_D_4_definition}. In particular, we ensure that $\mathbf{u}_0 \in \mathcal{H}_{D_4}$ by construction.
 Enforcing such a symmetry allows us to remove the natural translation and rotation invariances in the set of solutions as well as proving the $D_4$-symmetry by construction. Moreover, restricting to such a symmetry implies that we can consider a reduced set of Fourier indices, allowing us to save memory when performing the computer-assisted parts of the proof. In order to enforce $D_4$-symmetry, we had to implement a custom sequence structure compatible with the Julia package RadiiPolynomial.jl \cite{julia_olivier}. The algorithmic details to use the $D_4$-sequence structure are available at \cite{dominic_D_4_julia}. In particular, this sequence structure allows us to construct an approximate solution $\mathbf{u}_0$ in $\mathcal{H}_{D_4}.$
 At this step, we focus on identifying a solution $\tilde{\mathbf{u}} \in \mathcal{H}_{D_4}$ to \eqref{eq : gray_scott cov} in a vicinity of $\mathbf{u}_0$ using a Newton-Kantorovich approach. This requires the construction of an approximate inverse $\mathbb{A}$ to $D\mathbb{F}(\mathbf{u}_0)$, for which we follow the approach of \cite{unbounded_domain_cadiot}. The details on this construction can be found in Section \ref{sec : the operator A}. With $\mathbb{A}$ built, we define a fixed point operator $\mathbb{T}(\mathbf{u}) \bydef \mathbf{u} - \mathbb{A}\mathbb{F}(\mathbf{u})$. Then, we employ a Newton-Kantorovich approach in Theorem \ref{th: radii polynomial} to prove that $\mathbb{T}: \overline{B_r(\mathbf{u}_0)} \to \overline{B_r(\mathbf{u}_0)}$ is a contraction mapping on $\overline{B_r(\mathbf{u}_0)}$, the closed ball of radius $r$ around $\mathbf{u}_0$. The Banach fixed point theorem  then provides the existence of a unique solution in $\overline{B_r(\mathbf{u}_0)}$. Illustrations of approximate solutions $\mathbf{u}_0$, for which a proof was successful (cf. Theorems \ref{th : proof_of_leaf}, \ref{th : proof_of_ring} and \ref{th : proof_of_spike_away}), are displayed in the next figure.   
\begin{figure}[H]
\centering
 \begin{minipage}{.33\linewidth}
  \centering\epsfig{figure=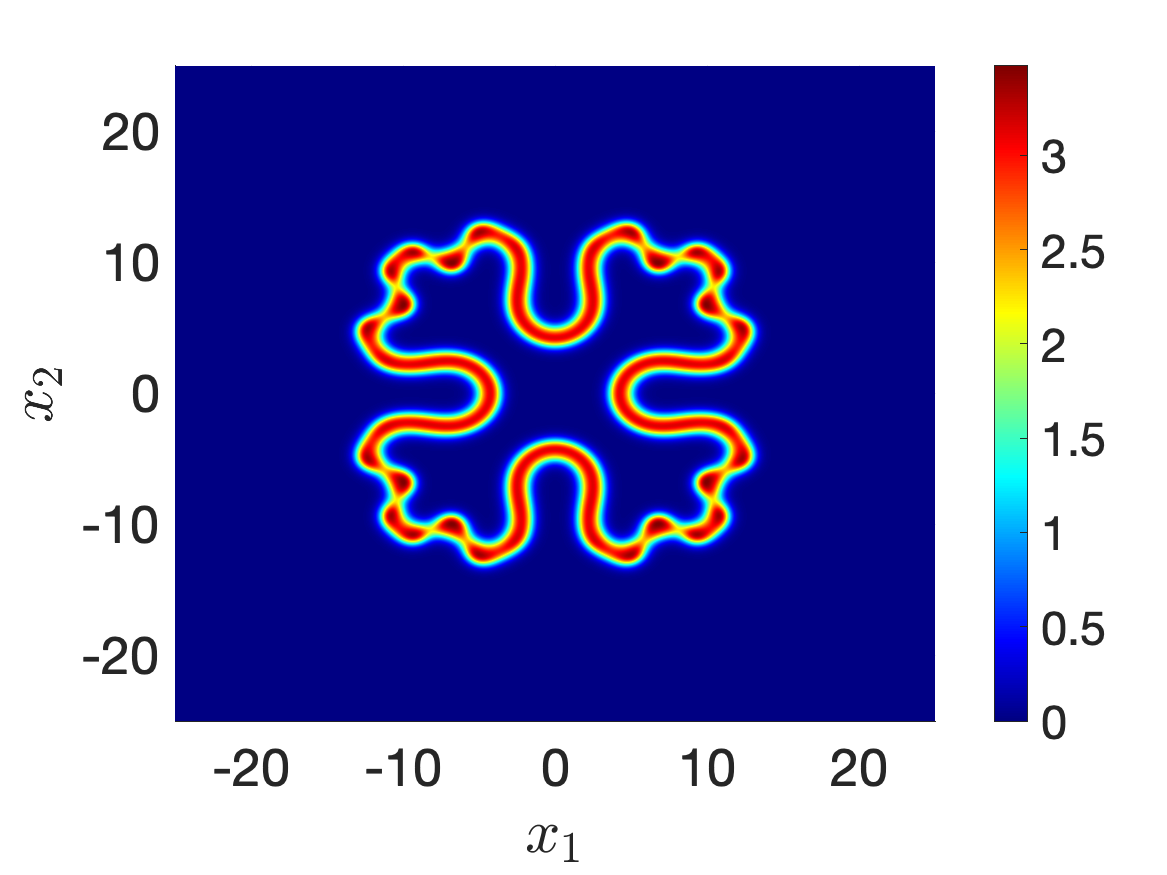,width=\linewidth}
  \end{minipage}%
 \begin{minipage}{.33\linewidth}
  \centering\epsfig{figure=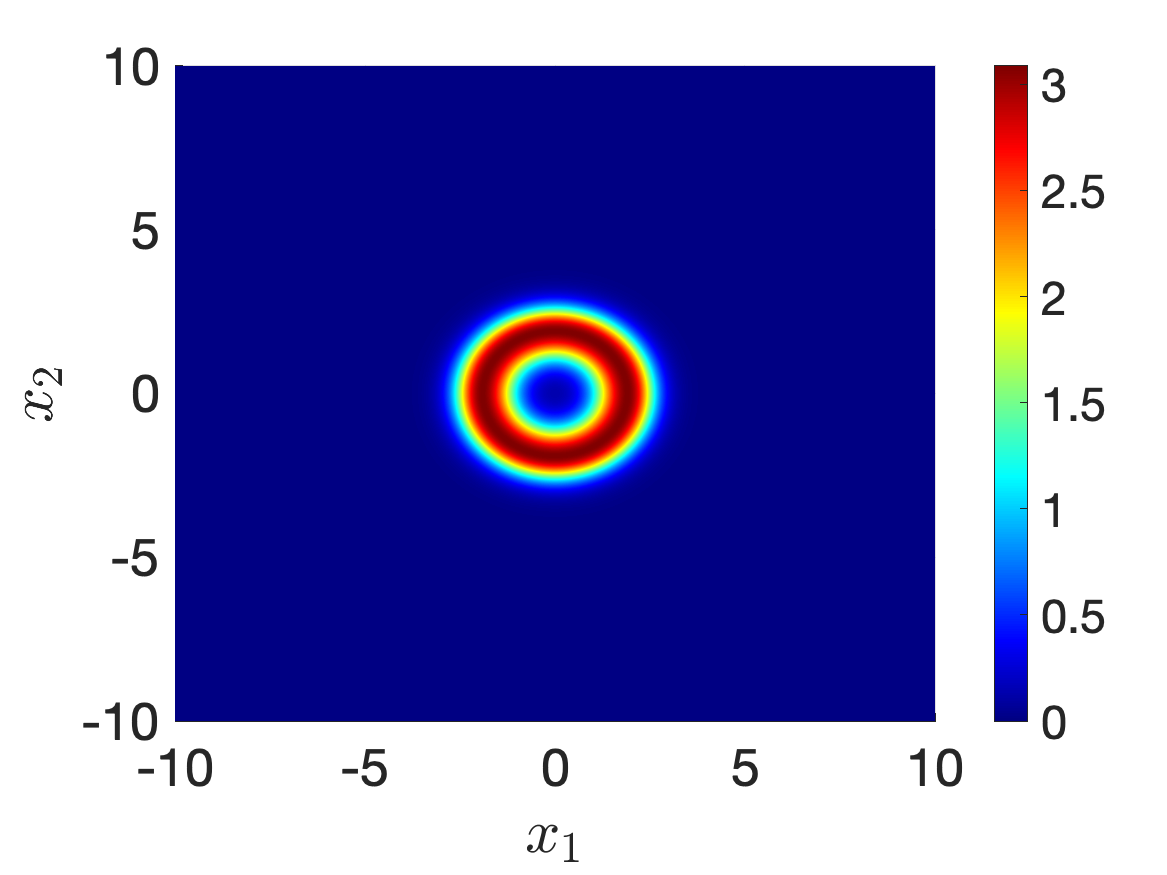,width=\linewidth}
 \end{minipage} %
 \begin{minipage}{.33\linewidth}
  \centering\epsfig{figure=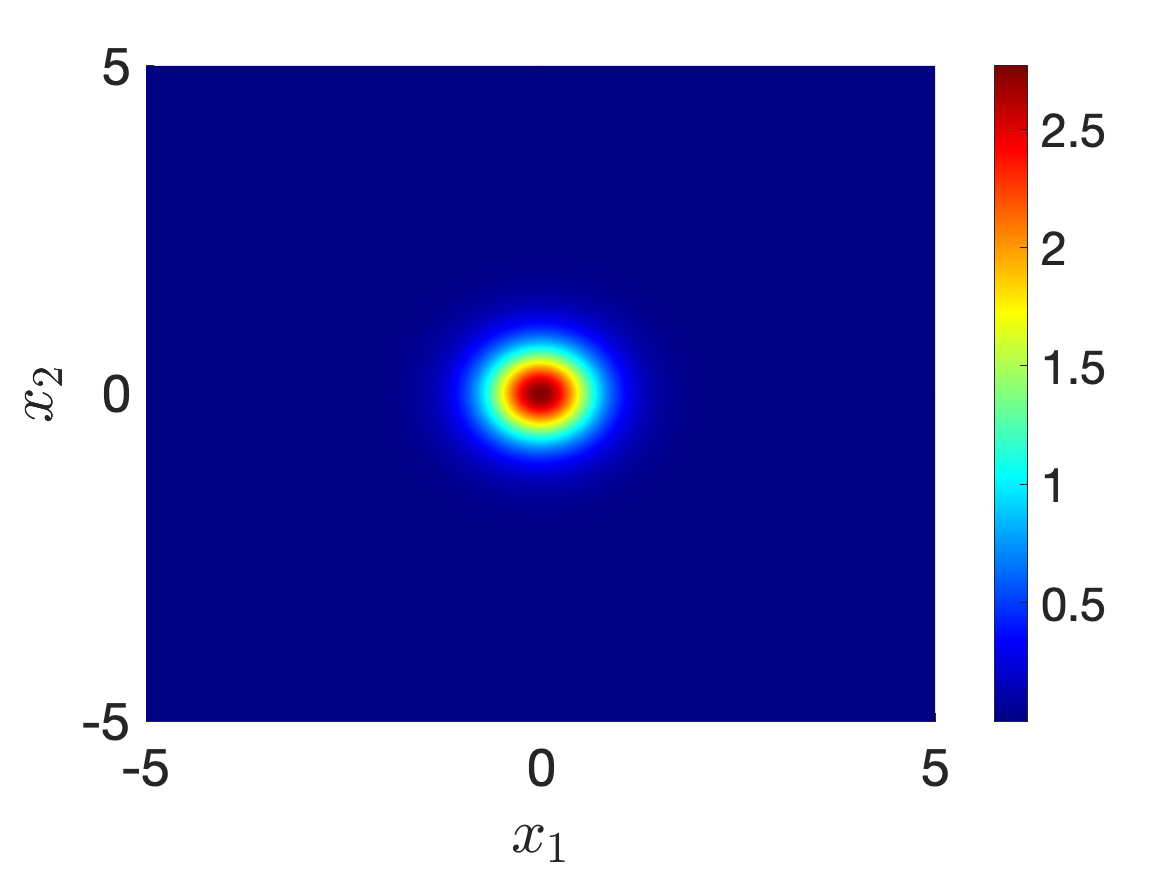,width=\linewidth}
 \end{minipage} 
 \caption{Plot of $u_1$ of an approximation of a ``leaf pattern" (L), a ``ring pattern" (C) and a ``spike pattern" (R) in the Gray-Scott System of Equations.}
 \end{figure}\label{eq : fig_intro}%
As mentioned earlier, the development of the Newton-Kantorovich approach in Theorem \ref{th: radii polynomial} involved subsidiary problems of interest we had to address. To begin with, we  establish a general set-up under which nonlinear operators are well-defined as bounded operators on $\mathcal{H}$. Using the pointwise product (cf. \eqref{def : product for systems}), we provide, under Assumption \eqref{ass : LinvG in L1}, a class of PDEs for which polynomial, autonomous, nonlinear operators satisfy the above requirement. In particular, we provide in Lemma \ref{Banach algebra} an explicit upper bound for the operator norm of differential multilinear operators. Furthermore, Theorem \ref{th: radii polynomial} depends on the explicit computation of multiple bounds, which are presented in Section \ref{sec : bounds}. One of the bounds specific to the unboundedness of the domain is the bound $\mathcal{Z}_u$. Defined in Lemma \ref{lem : Z_full_1}, the bound $\mathcal{Z}_u$ can be viewed as an estimate of how close our Fourier series approximations, computed on a bounded domain $\Omega_0 = (-d,d)^2$, are to the true PDE quantities on $\R^2$. This bound is unique to the case when we are attempting to prove a solution on an unbounded domain, and hence the tools needed for its computation need to be developed. Specifically, we provide an improvement for the computation of such a bound, enabling the construction of tighter estimates. Finally,  Theorem 3.7 in \cite{sh_cadiot} allows one to prove the existence of a branch of periodic solutions limiting a planar localized pattern as the period approaches infinity. In this paper, we derive a similar result in the case of the Gray-Scott model. Assuming that the existence of the localized pattern is obtained thanks to Theorem \ref{th: radii polynomial},  Theorem \ref{th : radii periodic} provides a condition under which an unbounded branch of (spatially) periodic solutions, converging to the localized pattern, can be obtained. We present applications of such a result in Theorems \ref{th : proof_of_spike_R1}, \ref{th : proof_of_spike_away}, \ref{th : proof_of_ring}, and \ref{th : proof_of_leaf} along with the existence proofs of the respective localized patterns. This is, to the best of our knowledge, a new result for pattern formation in systems of PDEs.

 Before delving into the specifics of our approach, it is worth mentioning that computer-assisted techniques have provided innovative results for PDEs on unbounded domains. For ODEs, the parameterization method allows one to formulate a projected boundary value problem which one can treat using Chebyshev series or splines. Details on the method can be found in \cite{MR1976079}, \cite{MR1976080}, and \cite{MR2177465}. In \cite{plum_numerical_verif}, Plum et al. presented a method for proving weak solutions to second and fourth-order PDEs. Their approach relies on the rigorous control of the spectrum of the linearization around an approximate solution, using a homotopy argument and the Temple-Lehmann-Goerisch method (see Section 10.2 in \cite{plum_numerical_verif}). In particular, using this method, Wunderlich in \cite{plum_thesis_navierstokes} proved the existence of a weak solution to the Navier-Stokes equations defined on an infinite strip with an obstacle. As far as strong solutions are concerned, the method presented in \cite{olivier_radial} allows to study the existence of radially symmetric solutions on $\R^d$. By studying the PDE in polar coordinates, the problem in question transforms into an ordinary differential equation. Then, a homoclinic connection at zero is computed via a rigorous enclosure of the center stable manifold, thanks the use of a Lyapunov-Perron operator. In this paper, we opt for the framework of \cite{unbounded_domain_cadiot} under which we are able to prove constructively the existence of strong, non-radial, localized patterns in the Gray-Scott model.

This paper is organized as follows: in Section \ref{sec : formulation of the problem}, we introduce the necessary theory and notation to extend the method of \cite{unbounded_domain_cadiot} to systems of PDEs. In particular, we derive assumptions under which our analysis will be applicable. In Section \ref{sec : settingup}, we present the Newton-Kantorovich approach we apply in this paper as well as the construction of both the approximate solution $\mathbf{u}_0$ and the approximate inverse $\mathbb{A}$. In Section \ref{sec : bounds}, we provide the required analysis of the bounds invoked in Theorem \ref{th: radii polynomial}. In Section \ref{sec : bounds_reduced}, we address the specific case $\lambda_2 \lambda_1 = 1$, under which \eqref{eq : gray_scott} reduces to a scalar PDE. In Section \ref{sec : CAPs}, we provide four computer-assisted proofs of existence of localized patterns in the Gray-Scott model, as well as their numerical description.  All computer-assisted proofs, including the requisite codes, are accessible on GitHub at \cite{julia_cadiot_blanco_GS}.
\section{Set-up of the problem}\label{sec : formulation of the problem}
Before treating the 2D Gray-Scott model, we develop a general framework to study the existence (and local uniqueness) of localized stationary solutions of systems of autonomous semi-linear PDEs. In particular, we wish to extend the method presented in \cite{unbounded_domain_cadiot} to fit this setup. To begin, we need to introduce assumptions  characterizing the class of system of PDEs which will fall under the analysis presented in this manuscript.
\subsection{Localized solutions in a system of PDEs}\label{sec : general set up system of PDEs}

\par We first introduce some notations. Let $m\in \mathbb{N}$ and define the Lebesgue spaces $L^2(\mathbb{R}^m)$ and $L^2(\om)$ on $\R^m$ and on a bounded domain $\om \subset \mathbb{R}^m$ respectively. Then, given $k \in \mathbb{N}$, we define the Lebesgue product space $L^2 \bydef \left(L^2(\mathbb{R}^m)\right)^k$. $L^2$ is associated to its natural inner product $(\cdot, \cdot)_2$. More generally, $L^p \bydef \left(L^p(\R^m)\right)^k$ denotes the usual $p$ Lebesgue product space on $\mathbb{R}^m$ associated to its norm $\| \cdot \|_{p}$. Moreover, given $s \geq 0$, denote by $H^s \bydef (H^s(\mathbb{R}^m))^k$ the usual Sobolev product space on $\mathbb{R}^m$.  For a bounded linear operator $\mathbb{K} : L^2 \to L^2$, denote by $\mathbb{K}^*$ the adjoint of $\mathbb{K}$ in $L^2$. For $\alpha \in \mathbb{N}^m$, let $\partial^{\alpha} \bydef \partial^{\alpha_1}_{x_1} \dots \partial^{\alpha_m}_{x_m}$ be the partial derivative operator  and denote by $\mathcal{S} \bydef \left\{ u \in C^{\infty}(\mathbb{R}^m) : \sup_{x\in \mathbb{R}^m}|x^\beta \partial^\alpha u(x)| < \infty \text{ for all } \beta, \alpha \in \mathbb{N}^m \right\}$ the {\em Schwartz space} on $\mathbb{R}^m$.
\par Now, let $\mathbf{u} = (u_1,\dots,u_k) \in L^2$ and denote $\hat{\mathbf{u}} \bydef \mathcal{F}(\mathbf{u}) \bydef (\hat{u}_1,\dots,\hat{u}_k) \bydef (\mathcal{F}(u_1),\dots,\mathcal{F}(u_k))$ the Fourier transform of $\mathbf{u}$. More specifically,  $\displaystyle \hat{u}_j(\xi)  \bydef  \int_{\mathbb{R}^m}u_j(x)e^{-2\pi i x \cdot \xi}dx$ for all $\xi \in \mathbb{R}^m.$ Next, given $u_1,u_2 \in L^2(\mathbb{R}^m)$, we denote  $u_1*u_2$ the convolution of $u_1$ and $u_2$.
Finally, for given $\mathbf{u},\mathbf{v} \in L^2$, we define the pointwise multiplication $\mathbf{u}\mathbf{v}$ as 
\begin{align}\label{def : product for systems}
    \mathbf{u}\mathbf{v} \bydef \begin{bmatrix}
        u_1v_1\\
        u_2 v_2 \\
        u_3 v_3 \\
        \vdots \\
        u_k v_k
    \end{bmatrix}.
\end{align}

Now, given $\psi \in L^2$, we consider the following class of system of autonomous semi-linear PDEs
\begin{equation}\label{eq : f(u)=0 on S}
    \mathbb{L}\mathbf{u} + \mathbb{G}(\mathbf{u}) = \psi
    \vspace{-.2cm}
\end{equation}
 where $\mathbf{u} : \R^m \to \R^k$, $\mathbb{L}$ is a linear differential operator with constant coefficients and $\mathbb{G}$ is a purely nonlinear ($\mathbb{G}$ does not contain linear operators) polynomial differential operator of order $N_{\mathbb{G}} \in \mathbb{N}$  where $N_{\mathbb{G}} \geq 2$, that is we can decompose it as a finite sum 
 \vspace{-.3cm}
\begin{equation}\label{def: G and j}
     \mathbb{G}(\mathbf{u}) \bydef \displaystyle\sum_{i = 2}^{N_{\mathbb{G}}}\mathbb{G}_i(\mathbf{u}) 
     \vspace{-.2cm}
\end{equation}
where $\mathbb{G}_i$ is a sum of monomials (where the product is defined thanks to \eqref{def : product for systems}) of degree $i$ in the components of $\mathbf{u}$ and some of its partial derivatives.
In particular, for $i \in \{2,\dots,N_{\mathbb{G}}\}$, $\mathbb{G}_i$ can be decomposed as follows
\vspace{-.2cm}
\[
\mathbb{G}_i(\mathbf{u}) \bydef  \mathbb{G}_{i,1}\mathbf{u} \dots \mathbb{G}_{i,i}\mathbf{u}
\vspace{-.2cm}
\]
 where $\mathbb{G}_{i,j}$ is a linear differential operator with constant coefficients for all  $j \in \{1, \dots, i\}$. $\mathbb{G}(\mathbf{u})$ depends on the derivatives of $\mathbf{u}$ but we only write the dependency in $\mathbf{u}$ for simplification.   {This setup includes a wide variety of systems of PDEs, including polynomial stationary reaction-diffusion equations of the form \eqref{eq : stationary reaction diffusion}.

 Our goal is now to derive assumptions on $\mathbb{L}$ and $\mathbb{G}$, under which we can recover similarities with the analysis derived in \cite{unbounded_domain_cadiot}. In particular, let us first consider the following assumption.
\begin{assumption}\label{ass:A(1)}
Assume that the Fourier transform of the linear operator $\mathbb{L}$ is given by 
\begin{equation} \label{eq:assumption1_on_L}
\mathcal{F}\big(\mathbb{L}\mathbf{u}\big)(\xi) = l(\xi) \hat{\mathbf{u}}(\xi), \quad \text{for all } \mathbf{u}\in \mathcal{S}^k,
\end{equation}
where $l$ is a $k$ by $k$ matrix of polynomials in $\xi$. Moreover, assume that there exists $\sigma_0 >0$ such that
\begin{equation} \label{eq:l>0}
\left|\mathrm{det}\left(l(\xi)\right)\right|  \geq \sigma_0, \qquad \text{for all } \xi \in \mathbb{R}^m.
\end{equation}
That is, the determinant of $l(\xi)$ is uniformly bounded away from $0$.
\end{assumption}

\begin{remark}
    
Observe that Assumption \ref{ass:A(1)} is the equivalent of Assumption 2.1 from \cite{unbounded_domain_cadiot} for systems of PDEs. In fact, as illustrated in Remark 2.7 in \cite{unbounded_domain_cadiot}, such an assumption is essential for controlling the essential spectrum of $\mathbb{L}$ away from zero. As such, it is necessary for constructing a contracting Newton-like fixed point operator, which is our main objective in Section \ref{sec : settingup}. Note, however, that  a large class of system of PDEs falls into the scope of Assumption \ref{ass:A(1)}. Specifically for reaction-diffusion equations of the form \eqref{eq : stationary reaction diffusion}, Assumption \ref{ass:A(1)} imposes that $ \mathbf{D}_1 \Delta  + \mathbf{D}_2 I_d : (H^2(\R^m))^d \to (L^2(\R^m))^d$ is invertible and has a bounded inverse. This is the case for instance for well-chosen parameter regimes in the FitzHugh–Nagumo equation or in the Gray-Scott model, as presented Section \ref{ssec : illustration gray scott}.
\end{remark}
   More specifically, $l$ is the symbol (Fourier transform) of the differential operator $\mathbb{L}$. Essentially, Assumption \ref{ass:A(1)} is equivalent to having $l(\xi)$ invertible for all $\xi \in \mathbb{R}^m$. We now define the following norm and inner product \begin{align}\label{def : definition of the norm and inner product Hl}
    \|\mathbf{u}\|_{\mathcal{H}} \bydef \|\mathbb{L}\mathbf{u}\|_{2} ~~ \text{ and } ~~(\mathbf{u},\mathbf{v})_{\mathcal{H}} \bydef (\mathbb{L}\mathbf{u},\mathbb{L}\mathbf{v})_2 
\end{align}
for all $\mathbf{u}, \mathbf{v} \in \mathcal{H}$ where $\mathcal{H}$ is the Hilbert space
\begin{align}\label{def : Hilbert space Hl}
    \mathcal{H} \bydef \{ \mathbf{u} \in L^2, \|\mathbf{u}\|_{\mathcal{H}} < \infty \}.
\end{align} 
 Using Plancherel's identity, we have 
\begin{align*}
(\mathbf{u}, \mathbf{v})_{\mathcal{H}} = (l\hat{\mathbf{u}},l\hat{\mathbf{v}})_2 
\end{align*}
for all $\mathbf{u}, \mathbf{v} \in \mathcal{H}.$ In particular, note that $\mathbb{L} : \mathcal{H} \to L^2$ is a well-defined bounded linear operator, which is actually an isometric isomorphism. Now, it remains to add an assumption to ensure that the non-linear part $\mathbb{G} : \mathcal{H} \to L^2$ is well-defined as well (under the pointwise product \eqref{def : product for systems}). Before stating the necessary assumption for this, we first introduce some notations. We begin with two norms
\\ \noindent\begin{minipage}{.4\linewidth}
\begin{equation}\label{M_2_norm_def}
 \|M\|_{\mathcal{M}_1} \bydef \sup_{\xi \in \R^m} \sup\limits_{\substack{x \in \R^k\\ |x|_2=1}}  |M(\xi)x|_2\end{equation}\end{minipage}%
\begin{minipage}{.6\linewidth}
\begin{equation}\label{M_21_norm_def}
\vspace{0.15cm}\|M\|_{\mathcal{M}_2} \bydef \max_{i \in \{1, \dots, k\}} \left\{\left(\sum_{j = 1}^k \|M_{i,j}\|_{L^2(\R^m)}^2\right)^{\frac{1}{2}} \right\}\end{equation}
\end{minipage}
\\
where $|x|_2 \bydef \left(\sum_{i = 1}^k |x_i|^2 \right)^\frac{1}{2}$ is the usual Euclidean 2 norm in $\mathbb{C}^k$. These norms then allow us to define the spaces $\mathcal{M}_1$ and $\mathcal{M}_2$ as
\begin{align}
 &\mathcal{M}_1 \bydef \left\{ M = \left(M_{i,j}\right)_{i,j \in \{ 1,\dots,k \} }, \text{ where } M_{i,j} \in L^{\infty}(\mathbb{R}^m) \text{ and } ~ \|M\|_{\mathcal{M}_1} < \infty \right\}  \\
&\mathcal{M}_2 \bydef \left\{ M = \left(M_{i,j}\right)_{i,j \in \{ 1,\dots,k \} }, \text{ where } M_{i,j} \in L^{2}(\mathbb{R}^m) \text{ and } ~ \|M\|_{\mathcal{M}_2} < \infty \right\}.
\end{align}
The spaces $\mathcal{M}_1$ and $\mathcal{M}_2$, we determine the class of system of PDEs under which our analysis is applicable.
\begin{assumption}\label{ass : LinvG in L1}
For all $2 \leq i \leq N_{\mathbb{G}}$ and $p \in \{1, \dots, i\}$ assume that there exists a $k$ by $k$ matrix of polynomials $g_{i,p}$ such that
\[
\mathcal{F}\left(\mathbb{G}_{i,p} \mathbf{u}\right)(\xi) = g_{i,p}(\xi)\hat{\mathbf{u}}(\xi),
\quad \text{for all } \xi \in \mathbb{R}^m. 
\]  Moreover,  assume that  $$g_{i,p}l^{-1} \in \mathcal{M}_1 \cap \mathcal{M}_2.$$
\end{assumption}
Assumption \ref{ass : LinvG in L1} is the equivalent of Assumption 2.3 for scalar PDEs in \cite{unbounded_domain_cadiot}.  Notice that under Assumptions \ref{ass:A(1)} and \ref{ass : LinvG in L1}, \eqref{eq : f(u)=0 on S} is a semi-linear autonoumous system of PDEs. Moreover, using Assumption \ref{ass : LinvG in L1}, one can prove that each $\mathbb{G}_i: \mathcal{H} \to L^2$ is a well-defined bounded multi-linear operator. Specifically, an explicit upper bound for the operator norm of each $\mathbb{G}_i$ can be computed.  Furthermore, we prove that $\mathbb{G}$ actually maps $\mathcal{H}$ to $\left(H^{1}(\mathbb{R}^m)\right)^k$. This point is fundamental as it allows to gain regularity for the solution a posteriori (cf. Proposition \ref{prop : regularity of the solution general}). We present the proof of these results in the following lemma.
\begin{lemma}\label{Banach algebra}
Suppose that Assumptions \ref{ass:A(1)}  and \ref{ass : LinvG in L1} are satisfied. Then,
\begin{align}\label{eq : semi linear regularity}
    \mathbb{G}(\mathbf{u}) \in \left(H^{1}(\mathbb{R}^m)\right)^k, \quad \text{for all } \mathbf{u} \in \mathcal{H}.
\end{align}
For each $i \in \{2,\dots,N_{\mathbb{G}}\}$, let $\kappa_i$ be a non-negative constant satisfying
\begin{equation}\label{def : definition of kappai banach algebra}
    \kappa_i \geq  \min_{p \in \{1, \dots, i\}}\left\|{g_{i,1}}l^{-1}\right\|_{\mathcal{M}_2} \cdots \left\|{g_{i,p-1}}l^{-1}  \right\|_{\mathcal{M}_2} \left\|{g_{i,p}}l^{-1}\right\|_{\mathcal{M}_1}   \left\|{g_{i,p+1}}l^{-1}\right\|_{\mathcal{M}_2}  \cdots \left\|{g_{i,i}}l^{-1}\right\|_{\mathcal{M}_2}.
\end{equation}
Then, for all $\mathbf{u} \in \mathcal{H}$,
\vspace{-.1cm}
\begin{equation}\label{eq : multilinearity of G}
    \|\mathbb{G}_i(\mathbf{u})\|_2 \leq  \kappa_i \|\mathbf{u}\|_{\mathcal{H}}^i.
    \vspace{-.1cm}
\end{equation}
\end{lemma}

\begin{proof}
    Let $\mathbf{u} \in \mathcal{H}$ and let $\mathbf{v} \bydef \mathbb{L}\mathbf{u} \in L^2$, then
    \begin{align*}
        \|\mathbb{G}_i(\mathbf{u})\|_2 &=  \|\mathbb{G}_{i,1}\mathbf{u} \cdots \mathbb{G}_{i,i}\mathbf{u}\|_2\\
        &= \|\mathbb{G}_{i,1}\mathbb{L}^{-1}\mathbf{v} \cdots \mathbb{G}_{i,i}\mathbb{L}^{-1}\mathbf{v} \|_2\\
        &\hspace{-0.5cm}\leq \min_{p \in \{1, \dots, i\}}\|\mathbb{G}_{i,1}\mathbb{L}^{-1}\mathbf{v}\|_\infty  \cdots \|\mathbb{G}_{i,p-1}\mathbb{L}^{-1}\mathbf{v}\|_\infty \|\mathbb{G}_{i,p}\mathbb{L}^{-1}\mathbf{v}\|_2 \|\mathbb{G}_{i,p+1}\mathbb{L}^{-1}\mathbf{v}\|_\infty \dots  \|\mathbb{G}_{i,i}\mathbb{L}^{-1}\mathbf{v} \|_\infty.
    \end{align*}
    Now, given $i, p \in \{1, \dots, k\}$ and using Parseval's identity, we have
    \begin{align}
        \|\mathbb{G}_{i,p}\mathbb{L}^{-1}\mathbf{v}\|_2  = \|g_{i,p}l^{-1}\hat{\mathbf{v}}\|_2 = \left(\int_{\R^m} |g_{i,p}l^{-1}\hat{\mathbf{v}}|_2^2\right)^{\frac{1}{2}}.
    \end{align}
     Then, observing that
    \begin{align*}
        |(g_{i,p}l^{-1}\hat{\mathbf{v}})(\xi)|_2^2 \leq |\hat{\mathbf{v}}(\xi)|_2^2\sup_{\xi \in \R^m} \sup\limits_{\substack{x \in \R^k\\ |x|_2=1}}|g_{i,p}(\xi)l^{-1}(\xi)x|_2^2
    \end{align*}
    for all $\xi \in \mathbb{R}^m$, we get
    \begin{align*}
        \|g_{i,p}l^{-1}\hat{\mathbf{v}}\|_2  \leq  \|\hat{\mathbf{v}}\|_2 \|g_{i,p}l^{-1}\|_{\mathcal{M}_2} = \|{\mathbf{v}}\|_2 \|g_{i,p}l^{-1}\|_{\mathcal{M}_2}.
    \end{align*}
    Moreover, given $i,p \in \{1, \dots,k\}$,  let us define the $k$ by $k$ matrix of polynomials $M$  given by $M = (M_{j,s})_{j,s \in \{1,2,\dots,k\}} \bydef g_{i,p}l^{-1}$. Then notice that 
    \begin{align*}
         \|\mathbb{G}_{i,1}\mathbb{L}^{-1}\mathbf{v}\|_\infty \leq  \max_{j \in \{1, \dots, k\}} \left\|\sum_{s=1}^k M_{j,s}\hat{v}_s\right\|_1 
    \end{align*}
    where we used that $\|u\|_ \infty \leq \|\hat{u}\|_1$ for all $u : \R^m \to \R$ such that $\hat{u} \in L^1(\R^m)$. Moreover, for fixed $j \in \{1, \dots,k\}$, we have 
    \begin{align*}
        \left\|\sum_{s=1}^k M_{j,s}\hat{v}_s \right\|_1 \leq \sum_{s =1}^k \left\| M_{j,s}\hat{v}_s\right\|_1 \leq \sum_{s = 1}^k \|M_{j,s}\|_2 \|\hat{v}_s\|_2 
        \leq \|\hat{\mathbf{v}}\|_2 \left(\sum_{s = 1}^k \|M_{j,s}\|_2^2\right)^{\frac{1}{2}}
    \end{align*}
    where we used Holder's inequality in $L^2(\R^m)$ and in $\R^m$ in the last two steps. Using that $\|\hat{\mathbf{v}}\|_2 = \|{\mathbf{v}}\|_2$ by Plancherel's theorem, we obtain
\begin{align*}
     \|\mathbb{G}_{i,1}\mathbb{L}^{-1}\mathbf{v}\|_\infty &\leq \|{\mathbf{v}}\|_2 \max_{j \in \{1, \dots,k\}}\left(\sum_{s = 1}^k \|M_{j,s}\|_2^2 \right)^{\frac{1}{2}}= \|g_{i,p}l^{-1}\|_{\mathcal{M}_1}\|\mathbf{v}\|_2.
\end{align*}
Replacing $\mathbf{v}$ by $\mathbb{L}\mathbf{u}$ and recalling that $\|\mathbf{u}\|_{\mathcal{H}} = \|\mathbb{L}\mathbf{u}\|_2$,
this concludes the proof of \eqref{eq : multilinearity of G}. In particular, it implies that $\mathbb{G}(\mathbf{u}) \in L^2$ for all $\mathbf{u} \in \mathcal{H}$. \par To prove that $\mathbb{G}(\mathbf{u}) \in \left(H^{1}(\mathbb{R}^m)\right)^k$, let $i \in \{2, \dots, N_{\mathbb{G}}\}$ and let $p_1, p_2 \in \{1, \dots, i\}$. Moreover, let $\mathbf{u} \in \mathcal{H}$ and let $j \in \{1, \dots, m\}$. Denote $\partial_j$ the derivative with respect to the $j$-th variable, then
\begin{align*}
    \left\|\left(\partial_j\mathbb{G}_{i,p_1} \mathbf{u}\right)\left( \mathbb{G}_{i,p_2} \mathbf{u}\right)\right\|_2 \leq \|\partial_j\mathbb{G}_{i,p_1} \mathbf{u}\|_2 \|\mathbb{G}_{i,p_2} \mathbf{u}\|_\infty.
\end{align*}
Now, denoting $\mathbf{v} \bydef \mathbb{L}\mathbf{u} \in L^2$, we get
\begin{align*}
    \left\|\left(\partial_j\mathbb{G}_{i,p_1} \mathbf{u} \right)\left(\mathbb{G}_{i,p_2} \mathbf{u}\right)\right\|_2 \leq \|\partial_j\mathbb{G}_{i,p_1} \mathbb{L}^{-1}\mathbf{v}\|_2 \|\mathbb{G}_{i,p_2}\mathbb{L}^{-1} \mathbf{v}\|_\infty.
\end{align*}
Using the above computations, we readily have 
\begin{align*}
    \|\mathbb{G}_{i,p_2}\mathbb{L}^{-1} \mathbf{v}\|_\infty \leq \|g_{i,p_2}l^{-1}\|_{\mathcal{M}_1} \|\mathbf{v}\|_2.
\end{align*}
Moreover, we get
\begin{align*}
    \|\partial_j\mathbb{G}_{i,p_1} \mathbb{L}^{-1}\mathbf{v}\|_2 \leq \|\mathbf{v}\|_2 \left\|h g_{i,p_1}  l^{-1}\right\|_{\mathcal{M}_2},
\end{align*}
where $h$ is a $k$ by $k$ matrix of functions given by $h = (h_{i,j})_{i,j \in \{1,\dots,k\}}$ and $h_{i,j}(\xi) = \delta_{i,j} |\xi|_2$ for all $i,j \in \{1,\dots,k\}$, where $\delta_{i,j}$ is the usual Kronecker symbol. Now,  since $g_{i,p_1}l^{-1} \in \mathcal{M}_2$ by assumption, it implies that 
\begin{align*}
    \left(g_{i,p_1}l^{-1}\right)_{s,n} \in L^2(\R^m)
\end{align*}
for all $s,n \in \{1, \dots, k\}$. Moreover, under Assumptions \ref{ass:A(1)} and \ref{ass : LinvG in L1}, $\left(g_{i,p_1}l^{-1}\right)_{s,n}$ is a rational function. This implies that  
\begin{align*}
    \left(g_{i,p_1}l^{-1}\right)_{s,n}(\xi) = \mathcal{O}\left(\frac{1}{1+|\xi|_2}\right).
\end{align*}
Consequently, this implies that 
\begin{align*}
    \sup_{\xi \in \R^m} |\xi|_2 \left|\left(g_{i,p_1}l^{-1}\right)_{s,n}(\xi)\right| < \infty
\end{align*}
for all $s,n \in \{1, \dots, k\}$ and therefore, $\left\|h g_{i,p_1}  l^{-1}\right\|_{\mathcal{M}_2} < \infty.$  By equivalence of norms this implies that $\mathbb{G}_{i,p_1}\mathbf{u}\mathbb{G}_{i,p_2}\mathbf{u} \in \left(H^{1}(\mathbb{R}^m)\right)^k $ for all $\mathbf{u} \in \mathcal{H}$,  all $i \in \{2, \dots, N_\mathbb{G}\}$ and  all $p_1, p_2 \in \{1, \dots, i\}$.
\end{proof}
Using Lemma \ref{Banach algebra}, we obtain that the operator $\mathbb{G}$ is smooth from $\mathcal{H}$ to $L^2$. This implies that the zero finding problem $\mathbb{F}(\mathbf{u}) =0$ is well defined on $\mathcal{H}.$ 
Moreover, similarly to \cite{unbounded_domain_cadiot}, one can prove, under Assumptions \ref{ass:A(1)} and \ref{ass : LinvG in L1}, that zeros to $\mathbb{F}$ in $\mathcal{H}$ are actually smooth solutions to \eqref{eq : f(u)=0 on S} (assuming that $\psi$ is also smooth). Consequently, we prove in the next proposition that one can study \eqref{eq : f(u)=0 on S} on $\mathcal{H}$ and equivalently obtain classical solutions a posteriori. 
\begin{prop}\label{prop : regularity of the solution general}
Let $\mathbf{u} \in \mathcal{H}$ such that $\mathbf{u}$ solves \eqref{eq : f(u)=0 on S} and suppose $\psi \in (H^\infty(\mathbb{R}^m))^k  $. Then $\mathbf{u} \in (H^\infty(\mathbb{R}^m))^k $, specifically $\mathbf{u}$ is infinitely differentiable.
\end{prop}
\begin{proof}
The proof can easily be obtained combining Lemma \ref{Banach algebra} and  the bootstrapping argument presented in Proposition 2.5 in \cite{unbounded_domain_cadiot}.
\end{proof}

Notice that Assumptions \ref{ass:A(1)} and \ref{ass : LinvG in L1} provide a natural set-up to study localized solutions in systems of PDEs defined on $\R^m$. In particular, they provide the fundamental results which are required to extend the analysis derived in \cite{unbounded_domain_cadiot}. We now illustrate this set-up by studying the Gray-Scott system \eqref{eq : gray_scott cov}. In particular, in the rest of this paper we fix $k=2$ and $m=2$. Moreover, we show that \eqref{eq : gray_scott cov} satisfies Assumptions \ref{ass:A(1)} and \ref{ass : LinvG in L1}, and we provide explicit computations of the constants $\kappa_{i}$ introduced in Lemma \ref{Banach algebra}.

\subsection{The Gray-Scott System of Equations}\label{ssec : illustration gray scott}
 Studying localized patterns in the Gray-Scott system \eqref{eq : gray_scott cov}, and using the notations derived above, we look for $\mathbf{u} = (u_1,u_2)$ such that
\begin{align}\label{eq : zero finding problem original}
    \mathbb{F}(\mathbf{u}) \bydef \mathbb{L}\mathbf{u} + \mathbb{G}(\mathbf{u}) = 0
\end{align}
where $\mathbb{L}$ and $\mathbb{G}$ are defined in \eqref{definition_of_L} and \eqref{definition_of_G} respectively. We further define
\begin{align*}
    \mathbb{G}_2(\mathbf{u}) \bydef  \begin{bmatrix}
         u_1^2 \\ 0
    \end{bmatrix}
    \text{ and } ~
    \mathbb{G}_3(\mathbf{u}) \bydef \begin{bmatrix}
         (u_2 -\lambda_1 u_1)u_1^2 \\ 0
    \end{bmatrix}
\end{align*}
and observe that $\mathbb{G}(\mathbf{u}) = \mathbb{G}_2(\mathbf{u}) + \mathbb{G}_3(\mathbf{u})$.
In particular, defining $\mathbb{G}_{2,1} =  \mathbb{G}_{2,2} = \mathbb{G}_{3,2} = \mathbb{G}_{3,3} = \begin{bmatrix}
         1 & 0 \\
         0 & 0
    \end{bmatrix}$
    and $\mathbb{G}_{3,1} = \begin{bmatrix}
         -\lambda_1 & 1 \\
         0 & 0
    \end{bmatrix}$
we obtain that $\mathbb{G}_2(\mathbf{u}) =\mathbb{G}_{2,1}\mathbf{u}  \mathbb{G}_{2,2}\mathbf{u}  $ and that $\mathbb{G}_3(\mathbf{u}) =\mathbb{G}_{3,1}\mathbf{u}  \mathbb{G}_{3,2}\mathbf{u} \mathbb{G}_{3,3}\mathbf{u}$. Now that we introduced the notations of Section \ref{sec : general set up system of PDEs} to describe \eqref{eq : gray_scott cov}, we verify that Assumptions \ref{ass:A(1)} and \ref{ass : LinvG in L1} are satisfied.

To begin with, notice that $\mathcal{F}\left(\mathbb{L}\mathbf{u}\right) = l \hat{\mathbf{u}}$, where 
\begin{equation}\label{definition_of_l}
    l(\xi) \bydef \begin{bmatrix}
        l_{11}(\xi) &  0 \\
        l_{21}(\xi) & l_{22}(\xi)
    \end{bmatrix} \bydef \begin{bmatrix}
        -\lambda_1 |2\pi\xi|^2 - 1 & 0 \\
        \lambda_2 \lambda_1 - 1 & -|2\pi\xi|^2 - \lambda_2
    \end{bmatrix}
\end{equation}
for all $\xi \in \mathbb{R}^2$. In particular, $l(\xi)$ is invertible for all $\xi \in \R^2$ and
\begin{equation}\label{l_inverse}
    l^{-1}(\xi) = \begin{bmatrix}
        \frac{1}{l_{11}(\xi)} & 0 \\
        -\frac{l_{21}(\xi)}{l_{11}(\xi)l_{22}(\xi)} & \frac{1}{l_{22}(\xi)}
    \end{bmatrix}.
\end{equation}
Moreover, we readily have $|l^{-1}(\xi)| = \frac{1}{l_{11}(\xi)l_{22}(\xi)} \geq \frac{1}{\lambda_2} >0$ for all $\xi \in \R^2.$ This implies that Assumption \ref{ass:A(1)} is satisfied.

 Now, observe that for $(i,j) \in \{(2,1),(2,2),(3,2),(3,3)\}$, we have
\begin{align}\label{eq : gik linv for all ik}
    g_{i,j}l^{-1} &= \begin{bmatrix}
        \frac{1}{l_{11}} & 0 \\
        0 & 0
    \end{bmatrix}.
\end{align}
In particular, since $\frac{1}{l_{11}} \in L^2(\R^2)\bigcap L^\infty(\R^2)$, we obtain that $ g_{i,j}l^{-1} \in \mathcal{M}_1 \cap \mathcal{M}_2$ for all $(i,j) \in \{(2,1),(2,2),(3,2),(3,3)\}$.
Moreover, we have
\begin{align}\label{eq : g13 equation}
    g_{1,3}l^{-1} &= \begin{bmatrix}
        -\frac{\lambda_1}{l_{11}} - \frac{l_{21}}{l_{11}l_{22}} & \frac{1}{l_{22}} \\
        0 & 0
    \end{bmatrix}.
\end{align}
Similarly as above, since $\frac{1}{l_{11}}, \frac{1}{l_{22}} \in L^2(\R^2)\bigcap L^\infty(\R^2)$ and $l_{21} \in L^\infty(\R^2)$, we obtain that $ g_{1,3}l^{-1} \in \mathcal{M}_1 \cap \mathcal{M}_2$. This implies that Assumption \ref{ass : LinvG in L1} is satisfied.

Since both Assumption \ref{ass:A(1)} and \ref{ass : LinvG in L1} are satisfied, we denote $\mathcal{H}$ the Hilbert space introduced  in \eqref{def : Hilbert space Hl} for the specific operator $\mathbb{L}$ given in \eqref{definition_of_L}. Moreover, Assumption \ref{ass : LinvG in L1} provides that $\mathbb{G}$ is well-defined and smooth on $\mathcal{H}$, which provides the well-definedness of the zero finding problem $\mathbb{F}(\mathbf{u})=0$ in $\mathcal{H}$. In addition, Lemma \ref{Banach algebra} is applicable and we can compute explicit upper bounds for the operator norms of $\mathbb{G}_2$ and $\mathbb{G}_3$. This is achieved in the next lemma.
\begin{corollary}\label{cor : banach algebra1}
Let $\kappa_2, \kappa_3$ be non-negative constants defined as
\begin{align}\label{def : definition kappa2 kappa3}
    \kappa_2 \bydef  \frac{1}{2\sqrt{\lambda_1 \pi}} ~~ \text{ and } ~~
    \kappa_3 \bydef \frac{\sqrt{2}}{4\pi}\min\left\{\frac{1}{ \lambda_2\lambda_1}, ~ \frac{1}{\sqrt{\lambda_2 \lambda_1}} \right\} 
\end{align}
then
    \begin{align}\label{eq : inequality given by the kappas}
        \|\mathbb{G}_2(\mathbf{u})\|_2 \leq  \kappa_2 \|\mathbf{u}\|_{\mathcal{H}}^2 ~~\text{ and }~~
         \|\mathbb{G}_3(\mathbf{u})\|_2 \leq \kappa_3 \|\mathbf{u}\|_{\mathcal{H}}^3, ~~\text{ for all } \mathbf{u} \in \mathcal{H}.
    \end{align}
\end{corollary}
\begin{proof}
Using \eqref{eq : gik linv for all ik}, straightforward computations lead to 
\begin{align}
    \|g_{i,k}l^{-1}\|_{\mathcal{M}_1} = \sup_{\xi \in \R^2} \frac{1}{|l_{11}(\xi)|} = 1 \text{ and } \|g_{i,k}l^{-1}\|_{\mathcal{M}_2} = \left\|\frac{1}{l_{11}}\right\|_2
\end{align}
for  all $(i,k)  \in \left\{(2,1) ; ~(2,2); ~(3,1) ; ~(3,3) \right\}$. Now, notice that 
\begin{align}\label{eq : equality of the lik}
    \frac{\lambda_1}{l_{11}(\xi)} + \frac{l_{21}(\xi)}{l_{11}(\xi)l_{22}(\xi)} = \frac{1}{l_{22}(\xi)},
\end{align} for all $\xi \in \R^2$. Consequently, using \eqref{eq : g13 equation}, we get
\begin{align}
    \|g_{1,3}l^{-1}\|_{\mathcal{M}_1} &= \sup_{\xi \in \R^2}\sup\limits_{\substack{x \in \R^2\\ |x|_2=1}} \left| -\frac{\lambda_1}{l_{11}(\xi)}x_1 - \frac{l_{21}(\xi)}{l_{11}(\xi)l_{22}(\xi)}x_1 + \frac{x_2}{l_{22}(\xi)}\right|_2 \\
 &= \sup_{\xi \in \R^2}  \frac{1}{l_{22}(\xi)} \sup\limits_{\substack{x \in \R^2\\ |x|_2=1}} \left| x_2 - x_1 \right|_2 =  \frac{\sqrt{2}}{\lambda_2}.
\end{align}
Moreover, using \eqref{eq : equality of the lik}, we get
\begin{align*}
    \|g_{1,3}l^{-1}\|_{\mathcal{M}_2} = \sqrt{2}\left\|\frac{1}{l_{22}}\right\|_2.
\end{align*}
Consequently, using Lemma \ref{Banach algebra}, we obtain
\begin{align}
    \kappa_2 = \left\|\frac{1}{l_{11}}\right\|_{2} ~~ \text{ and } ~~ \kappa_3 = \sqrt{2}\min\left\{ \frac{1}{\lambda_2}\left\|\frac{1}{l_{11}}\right\|_{2}^2 , ~ \left\|\frac{1}{l_{11}}\right\|_{2} \left\|\frac{1}{l_{22}}\right\|_{2}\right\}.\label{kappas_with_norms}
\end{align}
In order to compute explicitly $\kappa_2$ and $\kappa_3$, it remains to compute $\left\|\frac{1}{l_{11}}\right\|_2$ and $\left\|\frac{1}{l_{22}}\right\|_2$. 
Observe first that
\begin{align}
    \left\|\frac{1}{l_{11}}\right\|_{2} &= \left\|\frac{1}{\lambda_1 |2\pi \xi|^2 + 1}\right\|_{2} = \frac{1}{\lambda_1} \left\|\frac{1}{|2\pi \xi|^2 + \frac{1}{\lambda_1}}\right\|_{2}.\label{rewrite_1_over_l_22}
\end{align}
Now, both $\frac{1}{l_{11}}$ and $\frac{1}{l_{22}}$ can be computed if we compute
\begin{align}
\left\|\frac{1}{|2\pi\xi|^2 + \eta}\right\|_{2}^2 &= \int_{\mathbb{R}^2} \frac{1}{(|2\pi \xi|^2 + \eta)^2} d\xi
\end{align}
for $\eta = \lambda_2, \frac{1}{\lambda_1}$.
Using polar coordinates and some standard integration techniques (see \cite{gradshteyn2014table} for example) we obtain
\begin{align}
\int_{\mathbb{R}^2} \frac{1}{(|2\pi \xi|^2 + \eta)^2} d\xi= \frac{1}{(2\pi)^3} \int_{0}^{\infty} \frac{r}{(r^2 + \frac{\eta}{(2\pi)^2})^2} dr = \frac{1}{4\eta \pi}.
\end{align}
Applying the above, we conclude the proof using that
\begin{align}\label{eq : norm of l11 and l22}
     \left\|\frac{1}{l_{11}}\right\|_{2} = \frac{1}{2\sqrt{\pi\lambda_1}} ~~ \text{ and } ~~ \left\|\frac{1}{l_{22}}\right\|_{2} = \frac{1}{2\sqrt{\pi\lambda_2}}.
\end{align}
\end{proof}
 The previous corollary is fundamental in our computer-assisted approach since it allows to use the mean value inequality on Banach spaces (as illustrated in Lemma \ref{lem : Bound Z_2}) and obtain explicit bounds while doing so. Before delving into the computer-assisted analysis, we introduce a specific set-up to enforce symmetries for the solutions and quotient out the natural invariants of the problem.

\par Supposing that $\mathbf{u} = (u_1,u_2)$ is a solution of \eqref{eq : gray_scott cov}, then any translation and rotation of the components $u_1$ and $u_2$ provides a new solution. Therefore, in order to isolate a pattern in the set of solutions, we choose to look for specific ones that are invariant under the symmetries of the square. More specifically, let $D_4$ be the dihedral group of order $8$ (i.e. the symmetry group of the square).
Then, for all $g \in D_4$, we enforce that $\mathbf{u}(gx) = \mathbf{u}(x)$ for all $x \in \R^2$. Not only does this quotient out the translation and rotation invariance in the set of solutions, but it also allows to enforce the $D_4$-symmetry.  With this in mind, we introduce the following Hilbert subspace $\mathcal{H}_{D_4} \subset \mathcal{H}$ which takes into account these symmetries \begin{equation}\label{H_l_D_4_definition}
         \mathcal{H}_{D_4} \bydef \{ \mathbf{u} \in \mathcal{H} \ | \ \mathbf{u}(x_1,x_2) = \mathbf{u}(-x_1,x_2) = \mathbf{u}(x_1,-x_2) = \mathbf{u}(x_2,x_1)\}.
    \end{equation}
Similarly, denote $L_{D_4}^2$ the Hilbert subspace of $L^2$ satisfying the $D_4$-symmetry.
In particular we notice that if $\mathbf{u} \in \mathcal{H}_{D_4}$, then $\mathbb{L}\mathbf{u} \in L^2_{D_4}$ and $\mathbb{G}(\mathbf{u}) \in L^2_{D_4}$ ; hence, it makes sense to define $\mathbb{L}$ and $\mathbb{G}$ as operators from $\mathcal{H}_{D_4}$ to $L^2_{D_4}$. 
 Finally, we look for solutions of the following problem
\begin{equation}\label{eq : f(u)=0 on H^l_D_4}
    \mathbb{F}(\mathbf{u}) = 0 ~~ \text{ and } ~~ \mathbf{u} \in \mathcal{H}_{D_4}.
\end{equation}
In particular, using Proposition \ref{prop : regularity of the solution general}, we obtain that solutions to \eqref{eq : f(u)=0 on H^l_D_4} are actually smooth (infinitely differentiable) solutions to \eqref{eq : gray_scott cov}. 

 Notice that what we have developed in this section is applicable for any choice of $\lambda_2,\lambda_1 > 0$. Now, under the assumption  $\lambda_2 \lambda_1 = 1$, the second equation of \eqref{eq : gray_scott cov} can be solved explicitly. Specifically, we obtain
\begin{align}\label{theta_sol_derive}
    \Delta u_2 - \lambda_2 u_2 = 0.
\end{align}
Now,  \eqref{theta_sol_derive} has a unique smooth solution vanishing at infinity which is the zero solution.
 Hence, for the special case $\lambda_2\lambda_1=1$, we can reduce \eqref{eq : gray_scott cov} to a scalar PDE, namely
\begin{equation}\label{gray_scott_reduced}
    0 = \lambda_1 \Delta u_1 + u_1^2 - \lambda_1 u_1^3 - u_1.
\end{equation}
This simplified form of \eqref{eq : gray_scott cov} has attracted a wide variety of studies (see \cite{castelli_gs}, \cite{exact_sol1D_Hale}, \cite{jay_jp_gs}). In particular, the authors of \cite{exact_sol1D_Hale} found multiple exact solutions in 1D. To the best of our knowledge, this is the only parameter regime under which an exact solution is known in \eqref{eq : gray_scott cov}. Note that the analysis we are about to present can be simplified in this case and we will discuss its specifics in Section \ref{sec : bounds_reduced}.
\begin{remark}\label{rem : not D4 symmetric}
By construction, the space $\mathcal{H}_{D_4}$ allows eliminating the translation and rotation invariance of the solutions. If one is interested in proving solutions that are not necessarily $D_4$-symmetric, one may restrict to a different symmetry and/or use the set-up of Section 5 in \cite{unbounded_domain_cadiot}. In this case, extra equations can be added with the same number of unfolding parameters to isolate the solution (see \cite{Lessard2021conserved,Lessard2021conserved_nbody} for illustrations).
\end{remark}
\subsection{Periodic spaces}\label{sec : periodic_spaces}

In this section, we recall some notations introduced in Section 2.4 of \cite{unbounded_domain_cadiot}. In particular, the objects defined in the above sections have a corresponding representation in Fourier series when restricted to a bounded domain. Specifically, we define our bounded domain of interest $$\Omega_0 \bydef (-d,d)^2$$  where $1 \leq d<\infty$.
Moreover, denote $$\tilde{n} = (\tilde{n}_1,\tilde{n}_2)\bydef\left(\frac{n_1}{2d},\frac{n_2}{2d}\right) \in \mathbb{R}^2$$ for all $(n_1,n_2) \in \mathbb{Z}^2$. As for the continuous case, we want to restrict to Fourier series representing $D_4$-symmetric functions. Let $\mathbf{u}_{n_1,n_2} = ((u_1)_{n_1,n_2},(u_2)_{n_1,n_2})$ be the $(n_1,n_2)$ Fourier coefficient of $\mathbf{u}$. In terms of Fourier coefficients, this restriction reads
\begin{align*}
    \mathbf{u}_{n_1,n_2} = \mathbf{u}_{-n_1,n_2} = \mathbf{u}_{n_1,-n_2} = \mathbf{u}_{n_2,n_1} \text{ for all } (n_1,n_2) \in \mathbb{Z}^2.
\end{align*}
In particular, when enforcing the $D_4$-symmetry, we can restrict the indices of Fourier coefficients from $\mathbb{Z}^2$ to the following reduced set
\begin{equation}\label{def : definition of J_redD4}
 J_{\mathrm{red}}(D_4) \bydef \left\{n \in \mathbb{Z}^2 \ | \ 0 \leq n_2 \leq n_1 \right\}.
\end{equation}
Furthermore, each index $n \in J_{\mathrm{red}}(D_4)$ has an associated orbit, denoted $\mathrm{orb}_{D_4}(n)$ (see \cite{gallianalgebra}), defined as 
\begin{align}
    \mathrm{orb}_{D_4}(n) = \{g \cdot n, \ \mathrm{for \ all} \ g \in D_4\}\label{def : orbit_D4}
\end{align}where $\cdot$ denotes the group action. When needed, the orbit can be used to construct the full series indexed on $\mathbb{Z}^2$. We chose this approach as the method presented by the authors of \cite{unbounded_domain_cadiot} is developed for Fourier series on a square and the method of \cite{olivier_radial} is only for radially symmetric solutions. Now, let $(\alpha_n)_{n \in J_{\mathrm{red}}(D_4)}$ be defined by 
\begin{equation}\label{def : alpha_n}
    \alpha_n \bydef \begin{cases}
        1 &\text{ if } n=(0,0)\\
       4 &\text{ if } n_1 = n_2 \text{ and } n_1 \neq 0\\
        4 &\text{ if } n_2 = 0 \text{ and } n_1 \neq 0\\
        8 &\text{ if } n_1\neq n_2  \text{ and } n_1\neq 0, n_2 \neq 0. 
    \end{cases}
\end{equation}
In particular, $\alpha_n$ is the size of $\mathrm{orb}_{D_4}(n)$. That is, $\alpha_n = |\mathrm{orb}_{D_4}(n)|$. Next, let $\ell^p(J_{\mathrm{red}}(D_4))$ denote the following Banach space
\begin{align}
    \ell^p({J_{\mathrm{red}}(D_4)}) \bydef \left\{U = (u_n)_{n \in J_{\mathrm{red}}(D_4)}: ~ \|U\|_p \bydef \left( \sum_{n \in J_{\mathrm{red}}(D_4)} \alpha_n|u_n|^p\right)^\frac{1}{p} < \infty \right\}.
\end{align}
In the above, note that $U$ is a sequence of scalars, and not a sequence of vectors.
In particular, $\ell^2(J_{\mathrm{red}}(D_4))$ is an Hilbert space associated to its natural inner product $(\cdot, \cdot)_2$ given by
\[
(U,V)_2 \bydef \sum_{n \in J_{\mathrm{red}}(D_4)} \alpha_n u_n \overline{v_n}
\]
for all $U = (u_n)_{n \in J_{\mathrm{red}}(D_4)}, V = (v_n)_{n \in J_{\mathrm{red}}(D_4)} \in \ell^2(J_{\mathrm{red}}(D_4)).$
Then, we define the Banach space $\ell^p_{D_4}$  as 
\begin{align}
   \ell^p_{D_4} \bydef \ell^p(J_{\mathrm{red}}(D_4)) \times \ell^p(J_{\mathrm{red}}(D_4)), ~ \text{ with norm } ~ \|\mathbf{U}\|_{p} = (\|U_1\|_{p}^p + \|U_2\|_{p}^p)^{\frac{1}{p}}.
\end{align}
 For a bounded operator $K : \ell^2_{D_4} \to \ell^2_{D_4}$ (resp. $K : \ell^2(J_{\mathrm{red}}(D_4)) \to \ell^2(J_{\mathrm{red}}(D_4))$), $K^*$ denotes the adjoint of $K$ in $\ell^2_{D_4}$ (resp. $\ell^2(J_{\mathrm{red}}(D_4))$). Now, similarly as what is achieved \cite{unbounded_domain_cadiot}, we  define $\gamma~:~L^2_{D_4}(\mathbb{R}^2) \to \ell^2(J_{\mathrm{red}}(D_4))$ as
\begin{align}
    \left(\gamma(u)\right)_n \bydef  \frac{1}{|\om|}\int_\om u(x) e^{-2\pi i \tilde{n}\cdot x}dx\label{single_gamma}
\end{align}
for all $n \in J_{\mathrm{red}}(D_4)$ and all $u \in L^2_{D_4}(\R^2)$. Similarly, we define $\gamma^\dagger : \ell^2(J_{\mathrm{red}}(D_4)) \to L^2_{D_4}(\mathbb{R}^2)$ as 
\begin{align}
    \gamma^\dagger\left(U\right)(x) \bydef \cha(x) \sum_{n \in J_{\mathrm{red}}(D_4)} u_n \sum_{m \in \mathrm{orb}_{D_4}(n)} e^{2\pi i \tilde{m} \cdot x}\label{single_gamma_dagger}
\end{align}
for all $x = (x_1,x_2) \in \mathbb{R}^2$ and all $U =\left(u_n\right)_{n \in J_{\mathrm{red}}(D_4)} \in \ell^2(J_{\mathrm{red}}(D_4))$, where $\cha$ is the characteristic function on $\om$. We now introduce a similar notation for functions in the product space $L^2_{D_4}$. In particular, define $\bgam : L^2_{D_4} \to \ell^2_{D_4}$ as
\begin{align}
\bgam(\mathbf{u}) = (\gamma(u_1),\gamma(u_2))\label{def : gamma}
\end{align}
and 
$\bgam^\dagger : \ell^2_{D_4} \to L^2_{D_4}$ as
\begin{align}
\bgam^\dagger(\mathbf{U}) = (\gamma^\dagger(U_1),\gamma^\dagger(U_2)).\label{def : gamma dagger}
\end{align}
More specifically, given $\mathbf{u} \in L^2_{D_4}$, $\bgam(\mathbf{u})$ represents the Fourier coefficients indexed on $J_{\mathrm{red}}(D_4)$ of the restriction of $\mathbf{u}$ on $\om$. Conversely, given a sequence $\mathbf{U}\in \ell^2_{D_4}$, $\bgam^\dagger\left(\mathbf{U}\right)$  is the function representation of $\mathbf{U}$ in $L^2_{D_4}.$ In particular, notice that $\bgam^\dagger\left(\mathbf{U}\right)(x) =0$ for all $x \notin \om.$  

Then, recall similar notations from \cite{unbounded_domain_cadiot}
\begin{align}
    L^2_{D_4,\om} \bydef \left\{\mathbf{u} \in L^2_{D_4} : \text{supp}(\mathbf{u}) \subset \overline{\om} \right\}~~ \text{ and } ~~
   \mathcal{H}_{D_4,\om} \bydef \left\{\mathbf{u} \in \mathcal{H}_{D_4} : \text{supp}(\mathbf{u}) \subset \overline{\om} \right\}.
\end{align}
Moreover, define $\mathcal{B}(L^2_{D_4})$ (respectively $\mathcal{B}(\ell^2_{D_4})$) as the space of bounded linear operators on $L^2_{D_4}$ (respectively $\ell^2_{D_4}$) and denote by $\mathcal{B}_\om(L^2_{D_4})$ the following subspace of $\mathcal{B}(L^2_{D_4})$
\begin{equation}\label{def : Bomega}
    \mathcal{B}_\om(L^2_{D_4}) \bydef \{\mathbb{K} \in \mathcal{B}(L^2_{D_4}) :  \mathbb{K} = \cha \mathbb{K} \cha\}.
\end{equation}
Finally, define $\bGam : \mathcal{B}(L^2_{D_4}) \to \mathcal{B}(\ell^2_{D_4})$ and $\bGam^\dagger : \mathcal{B}(\ell^2_{D_4}) \to \mathcal{B}(L^2_{D_4})$ as follows
\begin{equation}\label{def : Gamma and Gamma dagger}
    \bGam(\mathbb{K}) \bydef \bgam \mathbb{K} \bgam^\dagger ~~ \text{ and } ~~  \bGam^\dagger(K) \bydef \bgam^\dagger {K} \bgam 
\end{equation}
for all $\mathbb{K} \in \mathcal{B}(L^2_{D_4})$ and all $K \in \mathcal{B}(\ell^2_{D_4}).$

The maps defined above in \eqref{def : gamma}, \eqref{def : gamma dagger} and \eqref{def : Gamma and Gamma dagger} are fundamental in our analysis as they allow to pass from the problem on $\mathbb{R}^2$ to the one in $\ell^2_{D_4}$ and vice-versa. Furthermore, the following lemma, provides that this passage is an isometric isomorphism when restricted to the relevant spaces.

\begin{lemma}\label{lem : gamma and Gamma properties}
    The map $\sqrt{|\om|} \bgam : L^2_{D_4,\om} \to \ell^2_{D_4}$ (respectively $\bGam : \mathcal{B}_\om(L^2_{D_4}) \to \mathcal{B}(\ell^2_{D_4})$) is an isometric isomorphism whose inverse is given by $\frac{1}{\sqrt{|\om|}} \bgam^\dagger : \ell^2_{D_4} \to L^2_{D_4,\om}$ (respectively $\bGam^\dagger :   \mathcal{B}(\ell^2_{D_4}) \to \mathcal{B}_\om(L^2_{D_4})$). In particular,
    \begin{align}\label{eq : parseval's identity}
        \|\mathbf{u}\|_2 = \sqrt{\om}\|\mathbf{U}\|_2 \text{ and } \|\mathbb{K}\|_2 = \|K\|_2
    \end{align}
    for all $\mathbf{u} \in L^2_{D_4,\om}$ and $\mathbb{K} \in \mathcal{B}_\om(L^2_{D_4})$ where $\mathbf{U} \bydef \bgam(\mathbf{u})$ and $K \bydef \bGam(\mathbb{K})$.
\end{lemma}
\begin{proof}
The proof is obtained following similar steps as the ones of Lemma 3.2 in \cite{unbounded_domain_cadiot}.
\end{proof}
The above lemma not only provides a one-to-one correspondence between the elements in $L^2_{D_4,\om}$ (respectively $\mathcal{B}_\om(L^2_{D_4})$) and the ones in $\ell^2_{D_4}$ (respectively $\mathcal{B}(\ell^2_{D_4})$) but it also provides an identity on norms. This property is essential in our construction of an approximate inverse  in Section \ref{sec : the operator A}.

Now, we define the Hilbert space $\mathscr{h}$ as 
\begin{align*}
    \mathscr{h} \bydef \left\{ \mathbf{U}  \in \ell^2_{D_4} \text{ such that } \|\mathbf{U}\|_{\mathscr{h}} < \infty \right\}
\end{align*} associated to its inner product $(\cdot,\cdot)_{\mathscr{h}}$ and norm $\|\cdot\|_{\mathscr{h}}$ defined as 
\begin{align*}
    (\mathbf{U},\mathbf{V})_{\mathscr{h}} &\bydef \sum_{n \in J_{\mathrm{red}}(D_4)}  \alpha_n\left(l(\tilde{n})\mathbf{u}_n\right)\cdot \left(l(\tilde{n})\overline{\mathbf{v}_n}\right)\\
    \|\mathbf{U}\|_{\mathscr{h}} &\bydef \sqrt{(\mathbf{U},\mathbf{U})_{\mathscr{h}}}
\end{align*}
for all $\mathbf{U}=(\mathbf{u}_n)_{n \in J_{\mathrm{red}}(D_4)}, \mathbf{V}=(\mathbf{v}_n)_{n \in J_{\mathrm{red}}(D_4)} \in \mathscr{h}$ and $\alpha_n$ is defined as in \eqref{def : alpha_n}. Denote by $L : \mathscr{h} \to \ell^2_{D_4}$ and $G : \mathscr{h} \to \ell^2_{D_4}$ the Fourier coefficients representation of $\mathbb{L}$ and $\mathbb{G}$ respectively. More specifically,
\begin{align*}
    L \bydef \begin{bmatrix}
        L_{11} & 0\\
        L_{21} & L_{22}
    \end{bmatrix} \text{ and } G(\bU) \bydef \begin{bmatrix}
        (U_2 + 1 - \lambda_1 U_1)*U_1*U_1\\
        0
    \end{bmatrix}
\end{align*}
for all $\bU \in \mathscr{h}.$ For all $(i,j) \in \{(1,1),(2,1),(2,2)\}$, 
$L_{ij}$ is an infinite diagonal matrix with coefficients $\left(l_{ij}(\tilde{n})\right)_{n\in J_{\mathrm{red}}(D_4)}$ on its diagonal. Moreover $U*V \bydef \gamma(\gamma^\dagger(U)\gamma^\dagger(V))$ is defined as the discrete convolution (under $D_4$-symmetry). In particular, notice that Young's inequality for convolution is applicable 
\begin{align}
    \|U*V\|_{2} \leq \|U\|_{2} \|V\|_{1}\label{young_inequality}
\end{align}
for all $U \in \ell^2(J_{\mathrm{red}}(D_4)), V \in \ell^1(J_{\mathrm{red}}(D_4))$.
Furthermore, using the definition of $L$, notice that \[\|\mathbf{U}\|_{\mathscr{h}} = \|L\mathbf{U}\|_{2}.\] Finally, we define $F(\mathbf{U}) \bydef L\mathbf{U} + G(\mathbf{U})$ and introduce  
\begin{equation}\label{eq : F(U)=0 in X^l_D_4}
    F(\mathbf{U}) =0 ~~ \text{ and } ~~ \mathbf{U} \in \mathscr{h}
\end{equation}
as the periodic equivalent on $\om$ of \eqref{eq : f(u)=0 on H^l_D_4}.

\section{Setting up the computer-assisted approach}\label{sec : settingup}
The goal of this section is to set-up a computer-assisted approach based on a Newton-Kantorovich argument. We  state the necessary theorem  to perform such an analysis, and construct the related objects to apply it.
\subsection{Radii-Polynomial Theorem}\label{sec : radii_polynomial_theorem}
 Given $\mathbf{u}_0 \in \mathcal{H}_{D_4}$, an approximate solution to \eqref{eq : f(u)=0 on H^l_D_4}, and $\mathbb{A} : L^2_{D_4} \to \mathcal{H}_{D_4}$, an approximate inverse to $D\mathbb{F}(\mathbf{u}_0)$, we want to prove that there exists $r>0$ such that $\mathbb{T} : \overline{B_r(\mathbf{u}_0)} \to \overline{B_r(\mathbf{u}_0)}$ given by
\[
\mathbb{T}(\mathbf{u}) \bydef \mathbf{u} - \mathbb{A}\mathbb{F}(\mathbf{u})
\]
is well-defined and a contraction. In order to determine a possible value for $r>0$ that would provide the contraction, we wish to use a Radii-Polynomial theorem.  In particular, we  build $\mathbb{A} : L_{D_4}^2 \to \mathcal{H}_{D_4}$, $\mathcal{Y}_0, \mathcal{Z}_1 >0$ and  $\mathcal{Z}_2 : (0, \infty) \to [0,\infty)$ in such a way that the hypotheses of the following theorem are satisfied.
\begin{theorem}\label{th: radii polynomial}
Let $\mathbb{A} : L_{D_4}^2 \to \mathcal{H}_{D_4}$ be a bounded linear operator. Moreover, let $\mathcal{Y}_0, \mathcal{Z}_1$ be non-negative constants and let  $\mathcal{Z}_2 : (0, \infty) \to [0,\infty)$ be a non-negative function  such that for all $r>0$
  \begin{align}\label{eq: definition Y0 Z1 Z2}
    \|\mathbb{A}\mathbb{F}(\mathbf{u}_0)\|_{\mathcal{H}} \leq &\mathcal{Y}_0\\
    \|I_d - \mathbb{A}D\mathbb{F}(\mathbf{u}_0)\|_{\mathcal{H}} \leq &\mathcal{Z}_1\\
    \|\mathbb{A}\left({D}\mathbb{F}(\mathbf{u}_0 + \mathbf{h}) - D\mathbb{F}(\mathbf{u}_0)\right)\|_{\mathcal{H}} \leq &\mathcal{Z}_2(r)r, ~~ \text{for all } \mathbf{h} \in B_r(0)
\end{align}  
If there exists $r_0>0$ such that
\begin{equation}\label{condition radii polynomial}
    \frac{1}{2}\mathcal{Z}_2(r_0)r_0^2 - (1-\mathcal{Z}_1)r_0 + \mathcal{Y}_0 <0, \ and \ \mathcal{Z}_1 + \mathcal{Z}_2(r_0)r_0 < 1 
 \end{equation}
then there exists a unique $\tilde{\mathbf{u}} \in \overline{B_{r_0}(\mathbf{u}_0)} \subset \mathcal{H}_{D_4}$ such that $\mathbb{F}(\tilde{\mathbf{u}})=0$, where $B_{r_0}(\mathbf{u}_0)$ is the open ball of radius $r_0$ in $\mathcal{H}_{D_4}$ and centered at $\mathbf{u}_0$. 
\end{theorem}
\begin{proof}
The proof of the above theorem relies on the application of  the Banach fixed point theorem to $\mathbb{T}$. Indeed, we want to show that $\mathbb{T}$ maps $ \overline{B_{r_0}(\mathbf{u}_0)}$ to itself and is contracting, for some $r_0$ to determine. The technical details for the derivation of condition \eqref{condition radii polynomial}  and the statement of the theorem
can be found in \cite{unbounded_domain_cadiot, van2021spontaneous}. In particular, the aforementioned work provide that 
\begin{align*}
    \|\mathbb{T}(\mathbf{u}) - \mathbf{u}_0\|_{\mathcal{H}} \leq \frac{1}{1- \mathcal{Z}_1} \left( \frac{1}{2}\mathcal{Z}_2(r)r^2  + \mathcal{Y}_0 \right)  \text{ and } \|\mathbb{T}(\mathbf{u}) - \mathbb{T}(\mathbf{w})\|_{\mathcal{H}} \leq \frac{\mathcal{Z}_2(r)r}{1-\mathcal{Z}_1} \|\mathbf{u} - \mathbf{w}\|_{\mathcal{H}}
\end{align*}
for all $r>0$ and all $\mathbf{u}, \mathbf{w} \in  \overline{B_{r}(\mathbf{u}_0)}$. If \eqref{condition radii polynomial} is satisfied for some $r_0>0$, then we obtain that 
\begin{align*}
    \|\mathbb{T}(\mathbf{u}) - \mathbf{u}_0\|_{\mathcal{H}} \leq r_0  \text{ and } \frac{\mathcal{Z}_2(r_0)r_0}{1-\mathcal{Z}_1} < 1
\end{align*}
for all  $\mathbf{u}, \mathbf{w} \in  \overline{B_{r_0}(\mathbf{u}_0)}$, implying that $\mathbb{T} : \overline{B_r(\mathbf{u}_0)} \to \overline{B_r(\mathbf{u}_0)}$
is well-defined and a contraction.
\end{proof}
In order to apply Theorem \ref{th: radii polynomial}, we first need to construct both $\mathbf{u}_0$ and $\mathbb{A}$. Their construction is presented in the next sections.

\subsection{Construction of the approximate solution \texorpdfstring{$\mathbf{u}_{0}$}{u0}}\label{sec : numerical construction}

In this section, we discuss the  construction of $\mathbf{u}_0$, which is an approximate solution to \eqref{eq : f(u)=0 on H^l_D_4}. This is generally a challenging problem. There is a rich variety of work on computing numerical solutions to reaction diffusion models (cf. \cite{jason_review_paper},\cite{wei_ring_2d_small_epsilon_unbounded},\cite{wei_spike_2d_unbounded},\cite{wei_asym_spike_2d},\cite{wei_multi_spike_2d}). We choose an approach based on the exact solution found by the authors of \cite{exact_sol1D_Hale}.  $\mathbf{u}_0$ is constructed numerically thanks to its Fourier coefficients representation. Fix $N \in \mathbb{N}$ to be the size of our numerical approximation for linear operators (i.e. matrices) and $N_0 \in \mathbb{N}$ to be the one of our Fourier coefficients approximations (i.e. vectors). Now, given $\mathcal{N} \in \mathbb{N}$, let us introduce the following projection operators 
 \begin{align}
 \nonumber
    (\pi^{\mathcal{N}}(V))_n  =  \begin{cases}
          v_n,  & n \in I^{\mathcal{N}} \\
              0, &n \notin I^{\mathcal{N}}
    \end{cases} ~~ \text{ and } ~~
     (\pi_{\mathcal{N}}(V))_n  =  \begin{cases}
          0,  & n \in I^{\mathcal{N}} \\
              v_n, &n \notin I^{\mathcal{N}}
    \end{cases}\label{def : piN and pisubN}
 \end{align}
    where $I^{\mathcal{N}} \bydef \{n \in J_{\mathrm{red}}(D_4), ~ n_1 \leq \mathcal{N}\}$ for all $V = (v_n)_{n \in  J_{\mathrm{red}}(D_4)} \in \ell^2(J_{\mathrm{red}}(D_4)).$ Then, we define
\begin{align}
    (\bpi^{\mathcal{N}}\mathbf{U})_n \bydef ((\pi^{\mathcal{N}}U_1)_n, (\pi^{\mathcal{N}}U_2)_n) ~~\text{and}~~(\bpi_{\mathcal{N}}\mathbf{U})_n \bydef ((\pi_{\mathcal{N}}U_1)_n, (\pi_{\mathcal{N}}U_2)_n)
\end{align}
for all $\mathbf{U} = \mathbf{u}_{n \in J_{\mathrm{red}}(D_4)} \in \ell^2_{D_4}$.
 To obtain a numerical approximation of a pattern, we begin by examining the case $\lambda_2\lambda_1 = 1$ (cf. \eqref{gray_scott_reduced}) and fix $0 < \lambda_1< \frac{2}{9}$. This simplifies to studying a scalar PDE and we look for an approximate solution $\overline{u}_{0,1} : \R^2 \to \R$. In \cite{exact_sol1D_Hale}, the authors found an exact solution $u : \R \to \R$ for the ODE equivalent of \eqref{gray_scott_reduced} given by
\begin{equation}\label{initial_guess_for_newton}
    {u}(x) = \frac{3}{1 + Q\mathrm{cosh}\left( \sqrt{\frac{x}{\lambda_1}}\right)}
\end{equation}
for all $x \in \R$, where $Q = \sqrt{1- \frac{9\lambda_1}{2}}$. We use this 1D solution and replace all instances of $x$ with $\sqrt{x^2 + y^2}$ to construct a function on $\R^2$ that we denote $\overline{u}_{0,1}$. Specifically,  \[\overline{u}_{0,1}(x,y) \bydef \frac{3}{1 + Q\mathrm{cosh}\left( \sqrt{\frac{\sqrt{x^2+y^2}}{\lambda_1}}\right)}\]
for all $(x,y) \in \R^2$. 
From here, we compute a $D_4$-Fourier series approximation of $\overline{u}_{0,1}$, which Fourier coefficients are given by a sequence $\overline{U}_{0,1} \in \ell^2(J_{\mathrm{red}}(D_4))$. In particular, $\overline{U}_{0,1}$ satisfies $(\overline{U}_{0,1},0) = \bpi^{N_0} (\overline{U}_{0,1},0)$, that is $\overline{U}_{0,1}$ only possesses a finite number of non-zero coefficients and hence can be seen as a vector. We then use Newton's method to improve this approximation. Once Newton's method has reached a wanted tolerance, we obtain an improved approximation that we still denote by $\overline{U}_{0,1}$.We used a tolerance of $1 \times 10^{-14}$ before terminating our Newton method. Since our approach relies on having a good approximate solution, we choose a low tolerance to ensure we have a good enough approximate solution to attempt a proof.  We acknowledge that our method is not a general approach. Generally, Newton's method is not guaranteed to converge. For the purpose of the Gray-Scott model, this approach was sufficient in computing the approximate solution. If Newton does not converge, one could attempt to use radial coordinates and finite differences, such as the authors of \cite{jason_spot_paper}, \cite{jason_ring_paper}, and \cite{sh_cadiot}. More generally, the interested reader could refer to the numerical work on the Gray-Scott model exposed in the Introduction (cf. \cite{jason_review_paper} for instance). Moreover, we still denote $\overline{u}_{0,1} \bydef \gamma^\dagger\left(\overline{U}_{0,1}\right)$ the function representation of $\overline{U}_{0,1}$ in $L^2_{D_4}(\mathbb{R}^2).$
Now, note that $(\overline{U}_{0,1},0)$ is also an approximate solution for \eqref{eq : F(U)=0 in X^l_D_4}. Consequently, from this approximate solution we use continuation to construct other approximate solution in the full system \eqref{eq : F(U)=0 in X^l_D_4}. In general, this allows us to construct an approximate solution  $\overline{\mathbf{U}}_0 = (\overline{U}_{0,1},\overline{U}_{0,2})$ to \eqref{eq : F(U)=0 in X^l_D_4} such that $\overline{\mathbf{U}}_0 = \bpi^{N_0} \overline{\mathbf{U}}_0$ (which is numerically represented as a vector).

 At this point, we now have a vector which represents the $D_4$ Fourier coefficients of $\overline{\mathbf{u}}_0 = \bgam^\dagger\left(\overline{\mathbf{U}}_0\right) \in L^2_{D_4}$.  However, $\overline{\mathbf{u}}_0$ is not necessarily smooth, that is  $\overline{\mathbf{u}}_0 \notin \mathcal{H}_{D_4}$ in general. Let us denote $\overline{\mathbf{u}}_0|_{\om}$  the restriction of $\overline{\mathbf{u}}_0$ on $\om$. In particular, since $\overline{\mathbf{U}}_0$ has a finite-number of non-zero coefficients, $\overline{\mathbf{u}}_0|_{\om}$ is smooth on $\overline{\om}$. Now,  by equivalence of norms, observe that $\mathcal{H}_{D_4}$ possesses the same functions as $\left(H^2(\R^2)\right)^2$ restricted to $D_4$ functions. Consequently, in order to ensure  $\overline{\mathbf{u}}_0 \in \mathcal{H}_{D_4}$, we need $\overline{\mathbf{u}}_0|_{\om}$ and its first derivatives to vanish on  $\partial \om$ (see \cite{mclean2000strongly} for zero extensions of functions with null trace). Since $D_4$ functions are also even, we automatically obtain that the first derivatives of $\overline{\mathbf{u}}_0|_\om$ are null on $\partial \om$. Furthermore, leveraging $D_4$ symmetry and periodicity, we only need to enforce $\overline{\mathbf{u}}_0|_{\om}(d,x_2)=0$ for all $x_2 \in (-d,d)$. This condition has an equivalent in terms of Fourier coefficients, which we translate as the kernel of a finite-dimensional periodic trace operator on Fourier coefficients (cf. Section 4.1 in \cite{unbounded_domain_cadiot}). Specifically, we build the matrix $\mathcal{T}$ which is an $ 2(N_0+1)$ by $(N_0+1)(N_0+2)$ matrix defined as 
\begin{align}
    (\mathcal{T}(\mathbf{V}))_{n_2} = \sum_{n_1 = 0}^{n_2-1} \tilde{\alpha}_{n_2,n_1} \mathbf{v}_{n_2,n_1} (-1)^{n_1} + \sum_{n_1 = n_2}^{N_0} \tilde{\alpha}_{n_1,n_2} \mathbf{v}_{n_1,n_2} (-1)^{n_1}.\label{coeffs_T_M}
\end{align}
where
\begin{align}
    \tilde{\alpha}_n \bydef \begin{cases}
        1 & \mathrm{\ if \ } n = (0,0) \\
        4 & \mathrm{\ if \ } n_1 = n_2 \neq 0 \\
        2 & \mathrm{\ if \ } n_1 \neq 0, n_2 = 0 \\
        4 & \mathrm{\ else}
    \end{cases}
\end{align}
and $\mathbf{V} = (\mathbf{v}_n)_{n \in I^N_0} \in \R^{(N_0+1)(N_0+2)}$. One can verify that having $\mathcal{T} \overline{\mathbf{U}}_0 =0$ (where we abuse notation and consider $\overline{\mathbf{U}}_0$ as a vector in $\R^{(N_0+1)(N_0+2)})$ implies that $\overline{\mathbf{u}}_0|_{\om}$ has a null trace of order 2 on $\partial \om$. Consequently, we want to construct a projection $\mathbf{U}_0^{N_0}$ of $\overline{\mathbf{U}}_0$ in the kernel of $\mathcal{T}$. Then, we define $\mathbf{U}_0^{N_0}$ as follows:
\[
\mathbf{U}_0^{N_0} \bydef \overline{\mathbf{U}}_{0} - D \mathcal{T}^{T}(\mathcal{T} D\mathcal{T}^{T})^{-1}\mathcal{T}\overline{\mathbf{U}}_{0},
\]
where $D$ is an invertible $(N_0+1)(N_0+2)$ square matrix. Using Remark 4.2 in \cite{unbounded_domain_cadiot}, we choose $D$ to be given by
\[
\left(D\mathbf{V}\right)_n \bydef \begin{bmatrix}
    \frac{(V_1)_n}{l_{11}(\tilde{n})} \\
    -\frac{(V_1)_n l_{21}(\tilde{n})}{l_{11}(\tilde{n})l_{22}(\tilde{n})}  + \frac{(V_2)_n}{l_{22}(\tilde{n})} 
\end{bmatrix}
\]
for all $n \in I^{N_0}$.
 Moreover, $\mathbf{U}_0^{N_0}$ is a vector in $\R^{(N_0+1)(N_0+2)}$ by construction and we denote $\mathbf{U}_0 $ its extension by zero in $\ell^2_{D_4}$, that is
\begin{align*}
   (\mathbf{U}_0)_n \bydef \begin{cases}
    (\mathbf{U}_0^{N_0})_n &\text{ if } n \in I^{N_0} \\
    0 &\text{ otherwise}.
\end{cases} 
\end{align*}
Since $\mathbf{U}_0^{N_0}$ is in the kernel of $\mathcal{T}$ by construction, we define
\begin{align}\label{def : construction of u0}
\mathbf{u}_0 \bydef \bgam^\dagger(\mathbf{U}_0)
\end{align}
and obtain that $\mathbf{u}_0|_{\om}$ has a null trace of order $2$ on $\om$. Using Lemma 4.4 in \cite{unbounded_domain_cadiot}, this implies that $\mathbf{u}_0 \in H^2(\R^2) \times H^2(\R^2)$. In practice, applying the trace will result in a worse numerical zero of the map on the bounded domain $\om$. Hence, it is crucial that we take $d$ large enough so that applying the trace does not significantly change the approximate solution we obtained. Numerically, after obtaining an initial approximate solution, we check that the value of this solution on the boundary is close to $0$. If it is not, then we increase $d$ and restart Newton before applying the trace. Therefore, we have successfully built an approximate solution $\mathbf{u}_0 \in \mathcal{H}_{D_4}$ such that $\mathrm{supp}(\mathbf{u}_0) \subset \overline{\om}$, associated to its Fourier coefficients representation $\mathbf{U}_0 = \bpi^{N_0} \mathbf{U}_0 \in \mathscr{h}$.
\subsection{The approximate inverse \texorpdfstring{$\mathbb{A}$}{A}}\label{sec : the operator A}

In this section, we focus our attention on the construction of $\mathbb{A} : L^2_{D_4} \to \mathcal{H}_{D_4}$, which is an approximate inverse to $D\mathbb{F}(\mathbf{u}_0)$. Following the approach presented in \cite{unbounded_domain_cadiot}, we observe that $\mathbb{L} : \mathcal{H}_{D_4} \to L^2_{D_4}$ (cf. \eqref{definition_of_L}) is an isometric isomorphism. Therefore, building $\mathbb{A}$ is equivalent to first construct $\mathbb{B} : L^2_{D_4} \to L^2_{D_4}$ approximating the inverse of $D\mathbb{F}(\mathbf{u}_0) \mathbb{L}^{-1}$ and then define $\mathbb{A} \bydef \mathbb{L}^{-1} \mathbb{B}$. Leveraging the properties of $\bgam$ and $\bGam$ in Lemma \ref{lem : gamma and Gamma properties}, we are able to compute $\mathbb{B}$ thanks to a bounded linear operator on Fourier coefficients. We present the specificities of the construction in the rest of this section.

We begin by constructing $B^N$, a numerical approximate inverse to $\bpi^N DF(\mathbf{U}_0)L^{-1}\bpi^N$. In particular, $B^N = \bpi^N B^N \bpi^N$ has a numerical matrix representation.  In the case of \eqref{eq : gray_scott cov}, observe that
\begin{align}
    DF(\mathbf{U}_0)L^{-1} &= I_d + DG(\mathbf{U}_0)L^{-1} \\
    &= I_d + \begin{bmatrix}
    DG_{11}(\mathbf{U}_0) & DG_{12}(\mathbf{U}_0) \\
    0 & 0
    \end{bmatrix}\begin{bmatrix}
        L_{11}^{-1} & 0 \\
        -L_{22}^{-1} L_{21} L_{11}^{-1} & L_{22}^{-1}
    \end{bmatrix} \\
    &= \begin{bmatrix}
        I_d & 0 \\
        0 & I_d
    \end{bmatrix} + \begin{bmatrix}
        DG_{11}(\mathbf{U}_0)L_{11}^{-1} - DG_{12}(\mathbf{U}_0)L_{22}^{-1} L_{21} L_{11}^{-1} & DG_{12}(\mathbf{U}_0)L_{22}^{-1} \\
        0 & 0
    \end{bmatrix} \\ 
    &= \begin{bmatrix}
      I_d + DG_{11}(\mathbf{U}_0)L_{11}^{-1} - DG_{12}(\mathbf{U}_0)L_{22}^{-1} L_{21} L_{11}^{-1} & DG_{12}(\mathbf{U}_0)L_{22}^{-1} \\
        0 & I_d  
    \end{bmatrix}.\label{expanded_DFLinverse}
\end{align}
Since \eqref{expanded_DFLinverse} is upper triangular by block, $B^N$ is chosen to be upper triangular by block as well. Moreover, the presence of an identity operator in the bottom-right corner implies that we can choose $B^N$ as 
\begin{equation}\label{B_in_full_equation}
    B^N \bydef \begin{bmatrix}
        B_{11}^N & B_{12}^N \\
        0 & \pi^N
    \end{bmatrix}.
\end{equation}
Both $B_{11}^N$ and $B_{12}^N$ are constructed numerically. 
Then define $B$ as 
\begin{align}
    B \bydef \bpi_N + B^N \bydef \begin{bmatrix}
        B_{11} & B_{12} \\
        0 & I_d
    \end{bmatrix}.\label{def : B_finite_infinite}
\end{align}
Finally, we define $\mathbb{B} : L^2_{D_4} \to L^2_{D_4}$ and $\mathbb{A} : L^2_{D_4} \to \mathcal{H}_{D_4}$ as 
\begin{align}\label{def : the operator A}
    \mathbb{A} \bydef \mathbb{L}^{-1} \mathbb{B} ~~\text{ and } ~~ \mathbb{B} \bydef \begin{bmatrix}
\mathbb{1}_{\mathbb{R}^2 \setminus \om} & 0 \\
0 & \mathbb{1}_{\mathbb{R}^2 \setminus \om} \end{bmatrix}+ \bGam^\dagger\left(B\right).
\end{align}
In particular, we have 
\begin{align}\label{def : definition of B by blocks}
    \mathbb{B} = \begin{bmatrix}
        \mathbb{B}_{11} & \mathbb{B}_{12}\\
        0 & I_d
    \end{bmatrix}
\end{align}
where $\mathbb{B}_{11} = \Gamma^\dagger(B_{11}^N + \pi_N)$ and $\mathbb{B}_{12} =  \Gamma^\dagger(B_{12}^N + \pi_N)$.
Now that we have constructed $\mathbf{u}_0$ and $\mathbb{A}$, we provide explicit computations for the bounds $\mathcal{Y}_0, \mathcal{Z}_1$, and $\mathcal{Z}_2$ in order to apply Theorem \ref{th: radii polynomial}. This is the emphasis of the next section.
\section{Computing the bounds of Theorem \ref{th: radii polynomial}}\label{sec : bounds}

In this section, we aim to compute the bounds for a system of semilinear PDEs. To simplify the computational details, we choose to perform the computations specifically on the Gray-Scott model \eqref{eq : gray_scott}. We emphasize that this approach can be generalized; however, such a generalization is purely computational. Hence, we use Gray-Scott to demonstrate in a simpler, more specific case while still demonstrating key issues.
We start by providing notations and results which we will be useful along the section. First, given $u \in L^\infty(\mathbb{R}^2)$,  denote by 
\begin{align}\label{def : multiplication operator}
    \mathbb{u}  \colon L^2(\mathbb{R}^2) &\to L^2(\mathbb{R}^2)\\
    v &\mapsto uv
\end{align}
 the multiplication operator associated to $u$. Similarly, given $U= (u_n)_{n \in J_{\mathrm{red}}(D_4)} \in \ell^1(J_{\mathrm{red}}(D_4))$ (recalling the definition of $J_{\mathrm{red}}(D_4)$ from \eqref{def : definition of J_redD4}),  denote by
\begin{align}\label{def : discrete conv operator}
    \mathbb{U} : \ell^2(J_{\mathrm{red}}(D_4)) &\to \ell^2(J_{\mathrm{red}}(D_4)) \\
    V &\mapsto  U * V
\end{align}
 the discrete convolution operator associated to $U$. We now define the following map $\varphi: \mathbb{R}^4 \to \mathbb{R}$ as 
\begin{equation}\label{definition_of_phi}
    \varphi(x_1, x_2, x_3, x_4) \bydef \min\left\{\max \{x_1, x_4\} + \max \{x_2, x_3\},~ \sqrt{x_1^2 + x_2^2 + x_3^2 + x_4^2}\right\}
\end{equation}
for a given $(x_1,x_2,x_3,x_4) \in \mathbb{R}^4$.
Now, when performing computer-assisted analysis on systems of equations, one has to compute operator norms in a product space. Doing this computation directly can become numerically expensive. To remedy potential numerical issues, it is, in practice, useful to estimate the norm of an operator by block in terms of the norms of its individual blocks. Having this strategy in mind, we derive the following lemma.
\begin{lemma}\label{lem : full_matrix_estimate}
Let $K_1,K_2,K_3,K_4 \in \mathcal{B}(X)$ for $X \in \left\{\ell^2(J_{\mathrm{red}}(D_4)),L^2_{D_4}(\mathbb{R}^2),\ell^2_{D_4}\right\}$ where $\mathcal{B}(X)$ is the set of bounded linear operators on $X$. Then,
    \begin{align}  
    \nonumber &\left\| \begin{bmatrix}
        K_1 & K_2 \\ 
        K_3 & K_4 \\
    \end{bmatrix} \right\|_2 \leq \varphi(\left\|K_1\right\|_{2},\left\|K_2\right\|_{2},\left\|K_3\right\|_{2},\left\|K_4\right\|_{2}),
    \end{align}
    where $\varphi$ is defined as in \eqref{definition_of_phi}.
\end{lemma}
\begin{proof}
To begin, observe that we can estimate 
\begin{align}
    \left\| \begin{bmatrix}
        K_1 & K_2 \\
        K_3 & K_4
    \end{bmatrix}\right\|_{2} &= \left\| \begin{bmatrix}
        K_1 & 0 \\
        0 & K_4
    \end{bmatrix} + \begin{bmatrix}
        0 & K_2 \\
        K_3 & 0
    \end{bmatrix}\right\|_{2} \\
    &\leq \left\|\begin{bmatrix}
        K_1 & 0 \\
        0 & K_4
    \end{bmatrix}\right\|_{2} + \left\|\begin{bmatrix}
        0 & K_2 \\
        K_3 & 0
    \end{bmatrix}\right\|_{2}.
\end{align}
Then, using the definition of the $2$-norm, we obtain
\begin{align}
    \left\|\begin{bmatrix}
        K_1 & K_2 \\
        K_3 & K_4
    \end{bmatrix}\right\|_{2} \leq \max(\|K_1\|_{2},\|K_4\|_{2}) + \max(\|K_2\|_{2},\|K_3\|_{2}).
\end{align}
On the other hand, let $y = (y_1,y_2) \in X,~ \|y\|_{2} = 1$. Then, notice that
\begin{align}
    \left\| \begin{bmatrix}
        K_1 & K_2 \\
        K_3 & K_4
    \end{bmatrix}\begin{bmatrix}
        y_1 \\ y_2
    \end{bmatrix}\right\|_{2} 
    = \sqrt{\|K_1 y_1 + K_2 y_2\|_2^2 + \|K_3 y_1 + K_4 y_2\|_2^2}.
\end{align}
Finally, we conclude the proof using Cauchy-Schwarz inequality
\begin{align}
    &\left\| \begin{bmatrix}
        K_1 & K_2 \\
        K_3 & K_4
    \end{bmatrix}\begin{bmatrix}
        y_1 \\ y_2
    \end{bmatrix}\right\|_{2}
    \\
    &\leq \sqrt{[(\|K_1\|_{2}^2 + \|K_2\|_{2}^2)^{\frac{1}{2}}( \|y_1\|_{2}^2 + \|y_2\|_{2}^2)^{\frac{1}{2}}]^2 + [(\|K_3\|_{2}^2 + \|K_4\|_{2}^2)^{\frac{1}{2}} (\|y_1\|_{2}^2 + \|y_2\|_{2}^2)^{\frac{1}{2}}]^2} \\
    &= \sqrt{(\|K_1\|_{2}^2 + \|K_2\|_{2}^2) \|y\|_{2}^2 + (\|K_3\|_{2}^2 + \|K_4\|_{2}^2)\|y\|_{2}^2} \\
    &= \sqrt{\|K_1\|_{2}^2 + \|K_2\|_{2}^2 + \|K_3\|_{2}^2 + \|K_4\|_{2}^2}.
\end{align}
\end{proof}
Hence, whenever we wish to compute the operator norm of an operator defined by blocks, the function $\varphi$ defined in \eqref{definition_of_phi} provides an upper bound using the operator norm of each block.

Finally, we provide the next lemma which allows us to estimate sums by integrals using a Riemann's sum argument.
\begin{lemma}\label{lem : riemann sum}
     Let $g : \R^2 \to \R$ be a $D_4$-symmetric function decreasing  in $|\xi|$. Then,
    \begin{align*}
         \sup_{q \in [d,\infty)} \frac{1}{4q^2}\sum_{n \in \mathbb{Z}^2} g\left(\frac{n}{2q}\right)^2 \leq  \|g\|^2_2 + \frac{1}{4d^2}g(0)^2 + \frac{2}{d} \int_{0}^{\infty}g({x_1},0)^2dx_1.
    \end{align*}
\end{lemma}
\begin{proof}
Let $q \geq d$,  then since $g$ is $D_4$-symmetric and decreasing, we readily get
\begin{align}
     \nonumber \sum_{n \in \mathbb{Z}^2} g\left(\frac{n}{2q}\right)^2 &= 4\sum_{n \in \mathbb{N}^2} g\left(\frac{n}{2q}\right)^2 + g(0)^2 +  4\sum_{n_1=1}^\infty g\left((\frac{n_1}{2q},0)\right)^2\\
     &\leq 4\int_{0}^{\infty} \int_{0}^{\infty} g\left(\frac{x}{2q}\right)^2dx + g(0)^2 + 4 \int_{0}^{\infty}g\left(\frac{x_1}{2q},0\right)^2dx_1\\
     &= 4q^2 \|g\|^2_2 + g(0)^2 + 8q \int_{0}^{\infty}g({x_1},0)^2dx_1.
     \label{g_j_bound}
\end{align}
\end{proof}
Hence, whenever we wish to bound a sum, we can use Lemma \ref{lem : riemann sum} to compute an integral instead. We are now ready to begin the computation of the needed bounds to apply Theorem \ref{th: radii polynomial}.
\subsection{The Bound \texorpdfstring{$\mathcal{Y}_0$}{Y0}}\label{sec : Y0}
For scalar PDEs, Lemma 4.11 in \cite{unbounded_domain_cadiot} provides an explicit formula for the computation of the bound $\mathcal{Y}_0$. We show in the next lemma that such a result can similarly be derived for systems of PDEs. In particular, the obtained formula is compatible with the use of rigorous numerics (it can be computed thanks to the evaluation of norms of vectors).

\begin{lemma}\label{lem : bound Y_0}
Let $\mathcal{Y}_0 >0$ be such that 
\begin{equation}
    \mathcal{Y}_0 \bydef  |\om|^{\frac{1}{2}}\left(\|B^N  F(\mathbf{U}_0)\|_{2}^2 + \|(\bpi^{N_0}-\bpi^N)L\mathbf{U}_0 + (\bpi^{3N_0}-\bpi^N) G(\mathbf{U}_0)\|_{2}^2 \right)^{\frac{1}{2}}.
\end{equation}
Then $\|\mathbb{A}\mathbb{F}(\mathbf{u}_0)\|_{\mathcal{H}} \leq \mathcal{Y}_0.$
\end{lemma}

\begin{proof}
First, combining \eqref{def : definition of the norm and inner product Hl} and \eqref{def : the operator A} we get
\begin{align}
    \|\mathbb{A}\mathbb{F}(\mathbf{u}_0)\|_{\mathcal{H}} &= \|\mathbb{B}\mathbb{F}(\mathbf{u}_0)\|_{2}.
\end{align}
Then, since $\mathbf{u}_0 = \bgam^\dagger\left(\mathbf{U}_0\right) \in \mathcal{H}_{D_4}$, we have $\mathbb{F}(\mathbf{u}_0) = \bgam^\dagger\left(F(\mathbf{U}_0)\right)$. Moreover, using \eqref{def : the operator A} again, we get
\begin{align*}
    \mathbb{B}\mathbb{F}(\mathbf{u}_0) = \bgam^\dagger\left(BF(\mathbf{U}_0)\right)
\end{align*}
where $B$ is defined in \eqref{def : B_finite_infinite}.
Now, using Lemma \ref{lem : gamma and Gamma properties}, it yields
\begin{align}
    \|\mathbb{B}\mathbb{F}(\mathbf{u}_0)\|_{2} = |\om|^{\frac{1}{2}}\|BF(\mathbf{U}_0)\|_{2}.
\end{align}
We can now proceed similarly as in the proof of Lemma 4.11 in \cite{unbounded_domain_cadiot} to obtain the desired result.
\end{proof}
\subsection{The Bound \texorpdfstring{$\mathcal{Z}_2$}{Z2}}\label{sec : Z2}
Before we derive the bound $\mathcal{Z}_2$ for the case of the Gray-Scott system of equations, we state the following lemma which we will be useful in the analysis of $\mathcal{Z}_2$.
\begin{lemma}\label{lem : banach algebra}
Let $\kappa_2$ be defined in \eqref{def : definition kappa2 kappa3} and define $\kappa_0 >0$ as 
\begin{align}\label{def : kappa0}
    \kappa_0 \bydef \min\left\{ \max\left\{\left[\left(\lambda_1 \kappa_2 + \frac{1}{2\sqrt{\pi \lambda_2}}\right)^2 + \frac{1}{4\pi \lambda_2}\right]^{\frac{1}{2}}, \frac{\sqrt{2}}{2\sqrt{\pi \lambda_2}}\right\},  \frac{\kappa_2}{\lambda_2} \left((1-\lambda_2 \lambda_1)^2 + 1\right)^{\frac{1}{2}}\right\}.
\end{align}
Then, 
\begin{align}
\|u_1v_2\|_2 \leq \kappa_0 \|\mathbf{u}\|_{\mathcal{H}} \|\mathbf{v}\|_{\mathcal{H}}
\end{align}
for all $\mathbf{u}= (u_1,u_2), \mathbf{v}=(v_1,v_2) \in \mathcal{H}$.
\end{lemma}
\begin{proof}
The proof can be found in Appendix \ref{apen : Z2}.
\end{proof}
\par With Lemma \ref{lem : banach algebra} available to us, we are now ready to compute the $\mathcal{Z}_2$ bound. We state the result in the following lemma.
\begin{lemma}\label{lem : Bound Z_2}
Let $\kappa_0$ and  $(\kappa_2, \kappa_3)$ be the bounds satisfying \eqref{def : kappa0} and \eqref{def : definition kappa2 kappa3} respectively. Moreover, define $q \in L^\infty(\R^2)\cap L^2_{D_4}(\R^2)$ and its associated Fourier coefficients $Q$ as  \begin{equation}\label{def : q}
        q \bydef u_{0,2} + \mathbb{1}_{\om} -3\lambda_1 u_{0,1}~~ \text{ and }~~
        Q \bydef \gamma(q).
\end{equation} 

Then, let
$\mathcal{Z}_2 : (0, \infty) \to [0, \infty)$ be defined  as 
\begin{align}
    \mathcal{Z}_2(r) &\bydef  2\left(\kappa_2^2 + 4 \kappa_0^2\right)^{\frac{1}{2}}\sqrt{\varphi(\mathcal{Z}_{2,1},\mathcal{Z}_{2,2},\mathcal{Z}_{2,2},\mathcal{Z}_{2,3})}   + 3\kappa_3 \max\{1,\| B_{11}^N\|_{2}\}r, 
\end{align}
for all $r \geq 0$,
where
\begin{align}
    &\mathcal{Z}_{2,1} \bydef  \| B_{11}^N(\mathbb{Q}^2 + \mathbb{U}_{0,1}^2)(B_{11}^N)^{*}\|_{2} \\
    &\mathcal{Z}_{2,2} \bydef \sqrt{\| B_{11}^N (\mathbb{Q}^2 + \mathbb{U}_{0,1}^2)\pi_N(\mathbb{Q}^2 + \mathbb{U}_{0,1}^2) (B_{11}^N)^{*} \|_{2}} \\
    &\mathcal{Z}_{2,3} \bydef \|Q^2 + U_{0,1}^2\|_{1},
\end{align}
and where $\varphi$ is defined in Lemma \ref{lem : full_matrix_estimate}. Then, $\|\mathbb{A}\left( D\mathbb{F}(\mathbf{u}_0 + \mathbf{h}) -D\mathbb{F}(\mathbf{u}_0)\right)\|_{\mathcal{H}} \leq \mathcal{Z}_2(r)r$ for all $r>0$ and all $\mathbf{h} \in \overline{B_r(0)} \subset \mathcal{H}$.
\end{lemma}
\begin{proof}
The proof can be found in Appendix \ref{apen : Z2}.
\end{proof}
\subsection{The Bound \texorpdfstring{$\mathcal{Z}_1$}{Zfull1}}\label{sec : Z1full}
In this section, we study the bound $\mathcal{Z}_1$ satisfying $\|I_d - \mathbb{A} D\mathbb{F}(\mathbf{u}_0)\|_{\mathcal{H}} \leq \mathcal{Z}_1$. In other terms, $\mathcal{Z}_1$ controls the accuracy of $\mathbb{A}$ as an approximate inverse for $ D\mathbb{F}(\mathbf{u}_0) : \mathcal{H}_{D_4} \to L^2_{D_4}.$
We begin by defining ${v}_1, v_2 \in L^\infty(\R^2) \cap L^2_{D_4}(\R^2)$ as 
\[
v_1 \bydef 2u_{0,2}u_{0,1} + 2u_{0,1} - 3\lambda_1 u_{0,1}^2 ~~\text{and}~~v_2 \bydef u_{0,1}^2. 
\]
In particular, $D\mathbb{G}(\mathbf{u}_0)$ can then be written as 
\begin{align}
    D\mathbb{G}(\mathbf{u}_0) \bydef \begin{bmatrix}
        \mathbb{v}_1 & \mathbb{v}_2 \\
        0 & 0
    \end{bmatrix},\label{def : w_0_in_full_equation}
\end{align}
where $\mathbb{v}_1$ and $\mathbb{v}_2$ are the multiplication operators (cf. \eqref{def : multiplication operator}) associated to $v_1$ and $v_2$ respectively. Next,  define 
\begin{align}
V_1 \bydef \gamma(v_1)\text{,}~~ V_2 \bydef \gamma(v_2), ~~ V_1^{N} \bydef \pi^{2N} V_1\text{ and}~~ V_2^N \bydef \pi^{2N} V_2.\label{def : V_1 and V_2}
\end{align}
Since $\mathbf{U}_0 = \pi^{N_0} \mathbf{U}_0$, we have $V_i =  \pi^{2N_0}V_i$ ($i \in \{1,2\}$) by definition. 
Then, we define 
\begin{align}
     DG^N(\mathbf{U}_0) = \begin{bmatrix}
         \mathbb{V}_1^N & \mathbb{V}_2^N\\
         0 & 0
     \end{bmatrix},
\end{align}
where $\mathbb{V}_1^N, \mathbb{V}_2^N$ are the discrete convolution operators (cf. \eqref{def : discrete conv operator}) associated to $V_1^N, V^N_2$ respectively.
Similarly, 
\begin{align}
    D\mathbb{G}^N(\mathbf{u}_0) \bydef \bGam^\dagger(DG^N(\mathbf{U}_0)), ~~ v_1^N \bydef \gamma^\dagger(V_1^N) ~~\text{ and } ~~ v_2^N \bydef \gamma^\dagger(V_2^N).\label{def : bbDG_N}
\end{align}
 The above notations allow to decompose both $D\mathbb{G}(\mathbf{u}_0)$ and $D{G}(\mathbf{U}_0)$ using a truncation of size $N$. In particular, we can produce an explicit bound $\mathcal{Z}_1$.
\begin{lemma}\label{lem : Z_full_1}
Let $\left(\mathcal{Z}_{u,k,j}\right)_{k \in \{1,2\}, j \in \{1,2,3\}}$ be bounds satisfying
\begin{align}
&\mathcal{Z}_{u,1,1} \geq \|\mathbb{1}_{\mathbb{R}^2\setminus\om} \mathbb{L}_{11}^{-1} \mathbb{v}_1^N\|_{2} \\
    &\mathcal{Z}_{u,1,2} \geq \|\mathbb{1}_{\mathbb{R}^2\setminus\om}\mathbb{L}_{22}^{-1}\mathbb{v}_2^N\|_{2}
    \\
    &\mathcal{Z}_{u,1,3} \geq \|\mathbb{1}_{\mathbb{R}^2\setminus\om}\mathbb{L}_{22}^{-1}\mathbb{L}_{21}\mathbb{L}_{11}^{-1}\mathbb{v}_2^N\|_{2}\\
&\mathcal{Z}_{u,2,1} \geq \|\mathbb{1}_{\om} (\mathbb{L}_{11}^{-1} - \Gamma^\dagger(L_{11}^{-1}))\mathbb{v}_1^N\|_{2} \\
    &\mathcal{Z}_{u,2,2} \geq \|\mathbb{1}_{\om} (\mathbb{L}_{22}^{-1} - \Gamma^\dagger(L_{22}^{-1}))\mathbb{v}_2^N\|_{2}
    \\
    &\mathcal{Z}_{u,2,3} \geq \|\mathbb{1}_{\om}(\mathbb{L}_{22}^{-1}\mathbb{L}_{21}\mathbb{L}_{11}^{-1} - \Gamma^\dagger(L_{22}^{-1} L_{21} L_{11}^{-1}))\mathbb{v}_2^N\|_{2}.
\end{align} 
Moreover, given $k \in \{1,2\}$, define  $\mathcal{Z}_{u,k} \bydef \sqrt{(\mathcal{Z}_{u,k,1} + \mathcal{Z}_{u,k,3})^2 + \mathcal{Z}_{u,k,2}^2} $.
Then, it follows that $\mathcal{Z}_{u,1}$ and $\mathcal{Z}_{u,2}$ satisfy
\begin{align} 
\mathcal{Z}_{u,1} &\geq \|\mathbb{1}_{\mathbb{R}^2 \setminus \om} D\mathbb{G}^N(\mathbf{u}_0) \mathbb{L}^{-1}\|_{2}\label{def : Zu1}\\
    \mathcal{Z}_{u,2} &\geq \|\mathbb{1}_{\om}D\mathbb{G}^N(\mathbf{u}_0)(\bGam^\dagger(L^{-1}) - \mathbb{L}^{-1})\|_{2}.\label{def : Zu2}
\end{align}
Furthermore, let $Z_1$ be a non-negative constant satisfying
\begin{align}
Z_1 &\geq \|I_d - B(I_d + DG^N(\mathbf{U}_0)L^{-1})\|_{2}.
\end{align}
Also define  $\mathcal{Z}_u  \bydef \sqrt{\mathcal{Z}_{u,1}^2 + \mathcal{Z}_{u,2}^2} $.
Let $\varphi$ be defined as in \eqref{definition_of_phi}. Finally, defining $\mathcal{Z}_1>0$ as
\begin{equation}\label{eq : first definition Z1}
    \mathcal{Z}_1 \bydef Z_1 + \max\left\{1, \|B_{11}^N\|_{2}\right\} \left(\mathcal{Z}_{u} +  \varphi\left(1,0,\frac{|\lambda_1\lambda_2-1|}{\lambda_2},\frac{1}{\lambda_2}\right) \sqrt{\|V_1^N - V_1\|_{1}^2 + \|V_2^N - V_2\|_{1}^2}\right),
\end{equation}
it follows that $ \|I_d -{\mathbb{A}}D\mathbb{F}(\mathbf{u}_0)\|_{\mathcal{H}} \leq \mathcal{Z}_1.$
\end{lemma}
\begin{proof}
The proof can be found in Appendix \ref{apen : Z1}.
\end{proof}
\begin{remark}
    In practice, especially if $N_0$ is bigger than $N$, it is useful to work with $V_i^N$ instead of $V_i$. This leads to the error term controlled by $V_i^N - V_i$ (cf. Lemma \ref{lem : Z_full_1}), which is small if the coefficients of $V_i$ decay rapidly.
\end{remark}
We must now compute each bound: $Z_1, \mathcal{Z}_{u,1},$ and $\mathcal{Z}_{u,2}$. This will be the emphasis of the next three sections.
\subsubsection{Computation of \texorpdfstring{$Z_1$}{Z1}}\label{sec : Z1}
This section derives a formula for the bound $Z_1$ satisfying $\|I - B(I_d + DG^N(\mathbf{U}_0)L^{-1})\|_{2} \leq Z_1$. Note that this is the usual bound one computes when proving periodic solutions to partial differential equations (see \cite{van2021spontaneous} for instance). We outline its computation in the following lemma.

\begin{lemma}\label{lem : Z1_bound}
Let $M^N \bydef \bpi^N + DG^N(\mathbf{U}_0)L^{-1}$ and $M \bydef \bpi_N + M^N = I_d + DG^N(\mathbf{U}_0)L^{-1}$. Recall that $\varphi$ is defined as in \eqref{definition_of_phi}. Then, let $Z_1 > 0$ be such that
\begin{equation}\label{ineq : Z}
    Z_1 \bydef \varphi(Z_{1,1},Z_{1,2},Z_{1,3},Z_{1,4})
\end{equation}
where 
\begin{align}
&Z_{1,1} \bydef \sqrt{\|(\bpi^N - B^NM^N)(\bpi^N -  (M^N)^*(B^N)^*)\|_{2}} \\
    &Z_{1,2} \bydef \max_{n \in J_{\mathrm{red}}(D_4) \setminus I^N} \sqrt{\left(\frac{\sqrt{\|M_1\|_{2}}}{|l_{11}(\tilde{n})|}  + \frac{|l_{21}(\tilde{n})|\sqrt{\|M_2\|_{2}}}{|l_{22}(\tilde{n}) l_{11}(\tilde{n})|}\right)^2 + \frac{\|M_2\|_{2}}{|l_{22}(\tilde{n})|^2}} \\
    &Z_{1,3} \bydef \sqrt{\|\bpi^NL^{-*}DG^N(\mathbf{U}_0)^{*}\bpi_NDG^N(\mathbf{U}_0)L^{-1}\bpi^N\|_{2}} \\
    &Z_{1,4} \bydef \max_{n \in J_{\mathrm{red}}(D_4)\setminus I^N}\sqrt{\left( \frac{\|V_1^N\|_{1}}{|l_{11}(\tilde{n})|}  + \frac{|l_{21}(\tilde{n})|\|V_2^N\|_{1}}{|l_{22}(\tilde{n}) l_{11}(\tilde{n})|} \right)^2 + \frac{\|V_2^N\|_{1}^2}{|l_{22}(\tilde{n})|^2} } 
\end{align}\par and $M_1,M_2$ are given by
\begin{align*}
    M_1 \bydef   B_{11}^N \mathbb{V}_1^N \pi_N \mathbb{V}_1^N (B_{11}^N)^{*} ~~\text{and}~~
    M_2 \bydef  B_{11}^N \mathbb{V}_2^N \pi_N \mathbb{V}_2^N (B_{11}^N)^{*}.
\end{align*}
Then, $\|I_d - B(I_d + DG^N(\mathbf{U}_0)L^{-1})\|_{2}  = \|I_d - BM\|_{2} \leq Z_1$.
\end{lemma}
\begin{proof}
The proof can be found in Appendix \ref{apen : Z1}.
\end{proof}
\begin{remark}
Note that the maxima in the definition of $Z_{1,2}$ and $Z_{1,4}$ are well-defined and finite since, as a function of $n$, the functions under the maxima are strictly decreasing.  
\end{remark}
\subsubsection{Computation of \texorpdfstring{$\mathcal{Z}_{u,1}$}{Zu}}\label{sec : Zu1}

Let $f_{11}$, $f_{22}$ and $f_3$ be functions on $\R^2$ defined as follows
\\ \noindent\begin{minipage}{.5\linewidth}
\begin{equation}\label{def : f11}
 f_{11} \bydef \mathcal{F}^{-1}\left(\frac{1}{l_{11}}\right)  \end{equation}
\end{minipage}%
\begin{minipage}{.5\linewidth}
\begin{equation}\label{def : f22}
f_{22} \bydef \mathcal{F}^{-1}\left(\frac{1}{l_{22}}\right)\end{equation}
\end{minipage}
\begin{align}
    f_3\bydef \mathcal{F}^{-1}\left(\frac{l_{21}}{l_{11}l_{22}}\right).\label{def : f3}
\end{align}
In particular, notice that 
\begin{align}
    \mathbb{L}^{-1}_{11} u = f_{11}*u, ~~ \mathbb{L}^{-1}_{22} u = f_{22}*u ~~ \text{ and } ~~ \mathbb{L}_{22}^{-1}\mathbb{L}_{21}\mathbb{L}_{11}^{-1} u = f_3*u\label{f_j_action}
\end{align}
for all $u \in L^2(\R^2).$ We can therefore transform accordingly the inequalities satisfied by the bounds $\mathcal{Z}_{u,1,j}$ for $j = 1,2,3$ in Lemma \ref{lem : Z_full_1} and obtain 
\begin{align}\label{eq : operator norm Zu1}
    \mathcal{Z}_{u,1,1} \geq \|\mathbb{1}_{\mathbb{R}^2\setminus\om} f_{11}* \mathbb{v}_1^N\|_{2}, ~~ 
    \mathcal{Z}_{u,1,2} \geq \|\mathbb{1}_{\mathbb{R}^2\setminus\om}f_{22}*\mathbb{v}_2^N\|_{2} ~ \text{ and } ~
    \mathcal{Z}_{u,1,3} \geq \|\mathbb{1}_{\mathbb{R}^2\setminus\om}f_3*\mathbb{v}_2^N\|_{2}
\end{align}
where we define the linear operator $f * \mathbb{v} : L^2(\R^2) \to L^2(\R^2)$ as $(f*\mathbb{v})u \bydef f*(vu)$ for all $u \in L^2(\mathbb{R}^2)$.
The next proposition then shows that the operator norms in \eqref{eq : operator norm Zu1} can be bounded using norms of functions. 
\begin{prop}\label{prop : introduce_f}
    Let $f,v \in L^2(\R^2)$  and assume that $f^2,v^2 \in L^1(\R^2)$. Then, given a domain (possibly unbounded) $\Omega \subset \R^2$,  ${\|\mathbb{1}_{\Omega}(f^2*v^2)\|_{1}}$ is finite and we have
    \begin{align*}
        \|\mathbb{1}_{\Omega}f*(vu)\|_{2} \leq \sqrt{\|\mathbb{1}_{\Omega}(f^2*v^2)\|_{1}} \|u\|_{2}
    \end{align*}
    for all $u \in L^2(\R^2)$.
\end{prop}
\begin{proof}
Since $f^2,v^2 \in L^1(\R^2)$, Young's inequality for the convolution provides that ${\|\mathbb{1}_{\Omega}(f^2*v^2)\|_{1}}$ is finite. Then, given $u \in L^2(\R^2)$, we use Cauchy-Schwatz inequality to get
    \begin{align*}
       \|\mathbb{1}_{\Omega}(f*(vu))\|_{2}^2 &=  \int_{\Omega}\left| \int_{\R^2} f(x-y)v(y)u(y) dy \right|^{2} dx\\
       &\leq \|u\|_{2}^2 \int_{\Omega}\int_{\R^2} |f(x-y)v(y)|^2 dy dx = \|\mathbb{1}_{\Omega}(f^2*v^2)\|_{1} \|u\|_{2}^2,
    \end{align*}
    which concludes the proof.
\end{proof}
Using the above proposition, if we can control explicitly the functions $f_{11}$, $f_{22}$ and $f_3$, then we will be able to determine an explicit formula for the bounds $\mathcal{Z}_{u,1,j}$. 
Specifically, we can choose $\mathcal{Z}_{u,1,j}$ such that 
\begin{align}
    &\mathcal{Z}_{u,1,1} \geq \sqrt{\|\mathbb{1}_{\mathbb{R}^2\setminus\om} (f_{11}^2* ({v}_1^N)^2\|_{1}} \\
    &\mathcal{Z}_{u,1,2} \geq \sqrt{\|\mathbb{1}_{\mathbb{R}^2\setminus\om}(f_{22}^2*({v}_2^N)^2)\|_{1}}\label{eq : new definition of Zu1 in equation} \\
    &\mathcal{Z}_{u,1,3} \geq \sqrt{\|\mathbb{1}_{\mathbb{R}^2\setminus\om}(f_3^2*(\mathbb{v}_2^N)^2)\|_{1}}.
\end{align}
In order to compute the bounds \eqref{eq : new definition of Zu1 in equation} explicitly, we want to prove that $f_{11}$, $f_{22}$ and $f_3$ are exponentially decaying at infinity. Indeed, letting 
\begin{align}\label{def : definition of a1 and a2}
    a_1 \bydef \sqrt{\frac{1}{\lambda_1}} ~~ \text{ and }~~ a_2 \bydef \sqrt{\lambda_2},
\end{align}
 we claim that there exists non-negative constants $C_0(f_{11})$ and $C_0(f_{22})$ such that 
\\ \noindent\begin{minipage}{.5\linewidth}
\begin{equation}\label{def : constant C0f11}
 |f_{11}(x)| \leq C_0(f_{11}) \frac{e^{-a_1 |x|}}{|x|^{1/4}_1}  \end{equation}
\end{minipage}%
\begin{minipage}{.5\linewidth}
\begin{equation}\label{def : constant C0f22}
|f_{22}(x)| \leq C_0(f_{22}) \frac{e^{-a_2|x|}}{|x|^{1/4}_1}\end{equation}
\end{minipage}
\begin{align}
    |f_3(x)| \leq \begin{cases}
        C_0(f_{22}) \frac{e^{-a_2|x|}}{|x|^{\frac{1}{4}}} & \mathrm{\ if \ } \frac{1}{\lambda_1} \geq \lambda_2 \\
       \frac{1}{a_1^2} C_0(f_{11}) \frac{e^{-a_1|x|}}{|x|^{\frac{1}{4}}} & \mathrm{\ if \ } \frac{1}{\lambda_1} \leq \lambda_2 \end{cases}\label{def : C3}
\end{align}
\\ for all $x \in \mathbb{R}^2 \setminus \{0\}.$ We prove such a claim in the following lemma.
\begin{lemma}\label{lem : constants_for_f}
Let $f_{11}, f_{22},$ and $ f_3$ be defined in \eqref{def : f11}, \eqref{def : f22}, and \eqref{def : f3} respectively and let $a_1,a_2$ be defined in \eqref{def : definition of a1 and a2}. Moreover, let $C_0(f_{11})$ and $C_0(f_{22})$ be non-negative constants defined as 
\begin{align}
    C_0(f_{11}) \bydef \max \left\{a_1^2 2e^{\frac{5}{4}}\left(\frac{2}{a_1}\right)^{\frac{1}{4}},~ a_1^2 \sqrt{\frac{\pi}{2\sqrt{a_1}}} \right\} ~~\text{and}~~C_0(f_{22}) \bydef \max \left\{ 2e^{\frac{5}{4}}\left(\frac{2}{a_2}\right)^{\frac{1}{4}},~ \sqrt{\frac{\pi}{2\sqrt{a_2}}} \right\}.
\end{align}
Then, \eqref{def : constant C0f11}, \eqref{def : constant C0f22}, and \eqref{def : C3} are satisfied.
\end{lemma}
\begin{proof}
The proof can be found in Appendix \ref{apen : Z1}.
\end{proof} 

\begin{remark}
    Notice that $\frac{1}{l_{11}}, \frac{1}{l_{22}} \notin L^1(\R^2)$ but simply $\frac{1}{l_{11}}, \frac{1}{l_{22}} \in L^2(\R^2)$. This property gives birth to a singularity at zero for $f_{11}$ and $f_{22}$, which we must take into consideration. In particular, we control the singularity by $\frac{1}{|x|^{\frac{1}{4}}}$. This is notably different from the set-up introduced in \cite{unbounded_domain_cadiot} since it was assumed that $\frac{1}{l} \in L^1(\R^m)$. Under this assumption, $\mathcal{F}^{-1}(\frac{1}{l})$ is a continuous function on $\R^m$, and a uniform exponential decay can be obtained.
\end{remark}
With Lemma \ref{lem : constants_for_f} available, combined with \eqref{eq : new definition of Zu1 in equation}, we are now ready to compute $\mathcal{Z}_{u,1}$. In particular, we provide explicit formulas for the bounds $\mathcal{Z}_{u,1,j}$ in the following lemma.
\begin{lemma}\label{lem : lemma Z11full}
Let $a_1, a_2$ be given in \eqref{def : definition of a1 and a2} and let
 $$E_j \bydef \gamma\left(\mathbb{1}_{\om}(x) \left(\frac{1}{2}(\cosh(2a_jx_1) + \cosh(2a_jx_2)\right)\right)$$ for $j = 1,2$. Moreover, let $\left(\mathcal{Z}_{u,1,j}\right)_{j \in \{1,2,3\}}$ be bounds defined as 
\begin{align}
&\mathcal{Z}_{u,1,1} \bydef \frac{\sqrt{2}C_0(f_{11})(2\pi)^{\frac{1}{4}}e^{-a_1d}\sqrt{|\om|}}{a_1^{\frac{3}{4}}}\sqrt{(V_1^N,V_1^N* E_1)_2} \\
    &\mathcal{Z}_{u,1,2} \bydef \frac{\sqrt{2}C_0(f_{22})(2\pi)^{\frac{1}{4}}e^{-a_2d}\sqrt{|\om|}}{a_2^{\frac{3}{4}}} \sqrt{(V_2^N,V_2^N * E_2)_2} \\
    &\mathcal{Z}_{u,1,3} \bydef \begin{cases}
        \mathcal{Z}_{u,1,2} & \mathrm{\ if \ } \frac{1}{\lambda_1} \geq \lambda_2 \\
        \frac{\sqrt{2}\lambda_2C_0(f_{11})(2\pi)^{\frac{1}{4}}e^{-a_1d}\sqrt{|\om|}}{a_1^{\frac{3}{4}}}\sqrt{(V_2^N,V_2^N* E_1)_2} & \mathrm{\ if \ } \frac{1}{\lambda_1} \leq \lambda_2
    \end{cases}
\end{align}
where $V_1^N,V_2^N$ are given in \eqref{def : w_0_in_full_equation}. Then $ \mathcal{Z}_{u,1} =  
    \sqrt{\left( \mathcal{Z}_{u,1,1} + \mathcal{Z}_{u,1,3}\right)^2 + \mathcal{Z}_{u,1,2}^2}$ satisfies \eqref{def : Zu1}. That is, $\|\mathbb{1}_{\mathbb{R}^2 \setminus \om} D\mathbb{G}^N(\mathbf{u}_0)\mathbb{L}^{-1}\|_{2} \leq \mathcal{Z}_{u,1}.$ 
\end{lemma}

\begin{proof} 
The proof can be found in Appendix \ref{apen : Z1}.
\end{proof}
\subsubsection{Computation of \texorpdfstring{$\mathcal{Z}_{u,2}$}{Zu2}}\label{sec : Zu2}
In this section we focus our attention on the computation of a bound 
$\mathcal{Z}_{u,2}$ satisfying $$\|\mathbb{1}_{\om}D\mathbb{G}^N(\mathbf{u}_0)(\bGam^\dagger(L^{-1}) - \mathbb{L}^{-1})\|_{2} \leq \mathcal{Z}_{u,2}.$$ To compute such a bound, we opt for a different approach than the one presented in \cite{unbounded_domain_cadiot}. We present the application of such an approach on the Gray-Scott model, but a similar analysis could be carried out in the class of system of PDEs presented in Section \ref{sec : general set up system of PDEs}.  First, similarly to the computation of $\mathcal{Z}_{u,1}$, we use the fact that the action of $\mathbb{L}^{-1}$ can be written as a convolution (cf. \eqref{f_j_action}). We use this fact and the exponential decay of $f_{11}, f_{22}$ and $f_3$ to obtain an explicit formula for $\mathcal{Z}_{u,2}$. In particular, using Green's identity (cf. \eqref{green_step1}), we are able to derive this explicit upper bound which depends only on an integral computation on $\partial \om$. We can compute this integral using Lemma \ref{lem : first_int_Z_u_2} in the appendix. The obtained results are summarized in the following lemma.
\begin{lemma}\label{lem : lemma Z12full}
Let $a_1$ and $a_2$ be defined in \eqref{def : definition of a1 and a2}.Let $C_{11}(f_{11}), C_{12}(f_{11}), C_{11}(f_{22})$, and $ C_{12}(f_{22})$ be non-negative constants defined as 
\begin{align}
    &C_{11}(f_{11}) \bydef  a_1^3 \sqrt{\frac{\pi}{2}}\frac{1}{\sqrt{a_1 + 1}}\left(1 + \frac{1}{a_1}\right),~~
    C_{12}(f_{11}) \bydef a_1^2\sqrt{\frac{\pi}{2}} \left(\sqrt{2}a_1 + 1\right), \\
    &C_{11}(f_{22}) \bydef a_2\sqrt{\frac{\pi}{2}}\frac{1}{\sqrt{a_2  + 1}}\left(1+\frac{1}{a_2}\right), ~~
    C_{12}(f_{22}) \bydef \sqrt{\frac{\pi}{2}}\left(\sqrt{2}a_2 + 1\right).
\end{align}
Then, given $j \in \{1,2\}$, let $C_j, \mathcal{C}_{1,j}, $ and $\mathcal{C}_{2,j},$ be non-negative constants defined as
\begin{align}
    &C_j \bydef \sqrt{\frac{d^2}{16a_j^2\pi^5} + \frac{1}{a_j^4} + \frac{d}{a_j^3}} \\
    &\mathcal{C}_{1,j} \bydef 2\sqrt{|\om|} e^{-a_jd}\frac{C_{11}(f_{jj})e^{-a_j} + C_{12}(f_{jj})}{a_j} \label{def : C_epsilon}  \\
    \nonumber
    &\mathcal{C}_{2,j} \bydef 2\sqrt{|\om|} C_{11}(f_{jj})\left( \ln(2)^2 + 2 \ln(2) + 2\right)^{\frac{1}{2}}.
    \end{align}
Moreover, given $w \in H^2(\mathbb{R}^2)$, define
\begin{align}
    &C(w) \bydef \frac{1}{\sqrt{|\om|}}\left( \frac{1}{2}\|\mathbb{1}_{(d-1,d)^2}\partial_{x_1} w\|_2\|\mathbb{1}_{(d-1,d)^2} w\|_2 + \|\mathbb{1}_{(d-1,d)} w(d,\cdot)\|_2^2\right)^{\frac{1}{2}}.
\end{align}
Finally, let $E_j$ be defined as in Lemma \ref{lem : lemma Z11full} and
 let $\left(\mathcal{Z}_{u,2,j}\right)_{j \in \{1,2,3\}}$ be bounds defined as 
\begin{align}
&\mathcal{Z}_{u,2,1} \bydef  \frac{4C_1}{\sqrt{|\om|}} \left(\mathcal{C}_{1,1}\sqrt{(V_1^N,E_1 * V_1^N)_2} +\mathcal{C}_{2,1}C(v_1^N)\right) \\
&\mathcal{Z}_{u,2,2} \bydef \frac{4C_2}{\sqrt{|\om|}} \left(\mathcal{C}_{1,2}\sqrt{(V_2^N,E_2 * V_2^N)_2} +\mathcal{C}_{2,2}C(v_2^N)\right) 
\end{align}
\begin{align}
    &\mathcal{Z}_{u,2,3} \bydef  \min(C_1,C_2) \left(\frac{\mathcal{Z}_{u,2,2}}{C_2} + \frac{4\lambda_1}{\sqrt{|\om|}}\left(\mathcal{C}_{1,1}\sqrt{(V_2^N,E_1 * V_2^N)_2} +\mathcal{C}_{2,1}C(v_2^N)\right)\right).
\end{align}
Then, $ \mathcal{Z}_{u,2} = \sqrt{\left( \mathcal{Z}_{u,2,1} + \mathcal{Z}_{u,2,3}\right)^2 + \mathcal{Z}_{u,2,2}^2}$ satisfies \eqref{def : Zu2}, $\|\mathbb{1}_{\om}D\mathbb{G}^N(\mathbf{u}_0)(\bGam^\dagger(L^{-1}) - \mathbb{L}^{-1})\|_{2} \leq \mathcal{Z}_{u,2}$.
\end{lemma}
\begin{proof}
We begin by focusing on the term $\mathcal{Z}_{u,2,1}$. Let $u \in L^2_{D_4}(\R^2)$ such that $\|u\|_2=1$ and define $w_1 \bydef v^N_1 u$. Then, define $g\bydef \mathbb{1}_{\Omega_0} (\mathbb{L}_{11}^{-1} - \Gamma^\dagger(L_{11}^{-1}))w_1$. By construction, $g \in L^2_{D_4}(\mathbb{R}^2)$ and $\mathrm{supp}(g) \subset \overline{\Omega_0}$. Therefore, $g$ has a a representation as a sequence of Fourier coefficients $\gamma(g) \bydef (g_n)_{n \in J_{\mathrm{red}}(D_4)} \in \ell^2(J_{\mathrm{red}}(D_4)) $ (recall \eqref{def : definition of J_redD4} for a definition of $J_{\mathrm{red}}(D_4)$) such that
\[
g_n = \frac{1}{|\Omega_0|} \int_{\Omega_0} (\mathbb{L}_{11}^{-1} - \Gamma^\dagger(L_{11}^{-1}))w_1(x)e^{-2\pi i\tilde{n}x}dx
\]
for all $n \in J_{\mathrm{red}}(D_4).$
As shown in Theorem 3.9 of \cite{unbounded_domain_cadiot}, we have that
\begin{equation}
    g_n = - \frac{1}{|\om|}\int_{\mathbb{R}^2\setminus\om} \mathbb{L}_{11}^{-1} w_1(x) e^{-2\pi i\tilde{n}.x}dx.
\end{equation}
Now, observe that we can write $g_n$ as
\[
|2\pi \tilde{n}|^2 g_n = \frac{1}{|\om|}\int_{\mathbb{R}^2\setminus \Omega_0} \mathbb{L}_{11}^{-1} w_1(x) \Delta (e^{-2\pi i\tilde{n}x})dx.
\]
Moreover, we can apply Green's identity to obtain
{\small\begin{align}
 &\int_{\mathbb{R}^2\setminus \Omega_0} \mathbb{L}_{11}^{-1} w_1(x) \Delta (e^{-2\pi i\tilde{n}x})dx\label{green_step1} \\ 
\nonumber &= \int_{\mathbb{R}^2\setminus \Omega_0} \Delta \mathbb{L}_{11}^{-1} w_1(x) e^{-2\pi i\tilde{n}x} dx - \int_{\partial \Omega_0} e^{-2\pi i\tilde{n}x} \nabla \mathbb{L}_{11}^{-1}w_1(x).dS(x) + \int_{\partial \Omega_0} \mathbb{L}_{11}^{-1}w_1(x) \nabla \left(e^{-2\pi i\tilde{n}x}\right).dS(x).
\end{align}}
 Since $w \in L^2_{D_4}(\mathbb{R}^2)$, it follows that
\[
\int_{\partial \Omega_0} \mathbb{L}_{11}^{-1}w_1(x) \nabla \left(e^{-2\pi i\tilde{n}x}\right).dS(x) = 0.
\]
Hence, we obtain that \eqref{green_step1} is equivalent to
{\small\begin{align}
 &\int_{\mathbb{R}^2\setminus \Omega_0} \mathbb{L}_{11}^{-1} w_1(x) \Delta (e^{-2\pi i\tilde{n}x})dx= \int_{\mathbb{R}^2\setminus \Omega_0} \Delta \mathbb{L}_{11}^{-1} w_1(x) e^{-2\pi i\tilde{n}x} dx - \int_{\partial \Omega_0} e^{-2\pi i\tilde{n}x} \nabla \mathbb{L}_{11}^{-1}w_1(x).dS(x).\label{green_step}
\end{align}}
Now, recalling that $\mathbb{L}_{11} = \frac{1}{a_1^2}\left(\Delta - a_1^2 I_d\right)$, observe that the first term of \eqref{green_step} can be rewritten as
\begin{align}
\\ \nonumber \int_{\mathbb{R}^2\setminus \Omega_0} \Delta \mathbb{L}_{11}^{-1} w_1(x) e^{-2\pi i\tilde{n}x} dx &= \int_{\mathbb{R}^2\setminus \Omega_0} (\Delta - a_1^2 + a_1^2) \mathbb{L}_{11}^{-1} w_1(x) e^{-2\pi i\tilde{n}x} dx 
\\ \nonumber
&\hspace{-3.0cm}= \int_{\mathbb{R}^2\setminus \Omega_0} (\Delta -a_1^2) \mathbb{L}_{11}^{-1} w_1(x) e^{-2\pi i\tilde{n}x} dx + \int_{\mathbb{R}^2\setminus \Omega_0} a_1^2 \mathbb{L}_{11}^{-1} w_1(x) e^{-2\pi i\tilde{n}x} dx
\\ \nonumber
&\hspace{-3.0cm}= a_1^2 \int_{\mathbb{R}^2\setminus \Omega_0} \mathbb{L}_{11} \mathbb{L}_{11}^{-1} w_1(x) e^{-2\pi i\tilde{n}x} dx - a_1^2 g_n
\\ \nonumber
&\hspace{-3.0cm}= a_1^2\int_{\mathbb{R}^2\setminus \Omega_0} w_1(x) e^{-2\pi i\tilde{n}x} dx -  a_1^2|\om| g_n \\
&\hspace{-3.0cm} = -a_1^2 |\om| g_n,\label{step_to_refer_later_Zu23}
\end{align}
where we used that supp$(w_1) \subset \overline{\om}$ for the last step.
Therefore we find that 
\[
|2\pi \tilde{n}|^2 g_n = -\frac{1}{|\om|}\int_{\partial \Omega_0}   {e^{-2\pi i\tilde{n}x}} \nabla \mathbb{L}_{11}^{-1}w_1(x).dS(x) - a_1^2g_n
\]
or equivalently
\begin{align}
g_n = -\frac{1}{|\om|}\frac{1}{a_1^2+|2\pi \tilde{n}|^2}\int_{\partial \Omega_0}   {e^{-2\pi i\tilde{n}x}} \nabla \mathbb{L}_{11}^{-1}w_1(x).dS(x). \label{g_n_result_for_Zu23}
\end{align}
for all $n \in J_{\mathrm{red}}(D_4)$. This implies that
\begin{align}
|g_n| \leq \frac{1}{|\om|}\frac{1}{a_1^2+|2\pi \tilde{n}|^2} \int_{\partial \Omega_0}    \nabla |\mathbb{L}_{11}^{-1}w_1(x).dS(x)|.
\end{align}
for all $n \in J_{\mathrm{red}}(D_4)$. Next, notice that 
\begin{align}
\nonumber \int_{\partial \Omega_0}    \nabla |\mathbb{L}_{11}^{-1}w_1(x).dS(x)| &= 2\int_{-d}^d    |\partial_{x_1} \mathbb{L}_{11}^{-1}w_1(d,x_2)|dx_2 + 2\int_{-d}^d    |\partial_{x_2} \mathbb{L}_{11}^{-1}w_1(x_1,d)|dx_1 \\
&= 4\int_{-d}^d    |\partial_{x_1} \mathbb{L}_{11}^{-1}w_1(d,x_2)|dx_2\label{two_integrals_grad_L}
\end{align}
as $w_1 \in L^2_{D_4}(\mathbb{R}^2)$.  Therefore, using Parseval's identity, we get
\begin{align}
\nonumber
\|\cha \left(\mathbb{L}_{11} - \Gamma^\dagger(L_{11}^{-1})\right)w_1\|_2^2 &= |\om| \sum_{n \in J_{\mathrm{red}}(D_4)}\alpha_n|g_n|^2 \\
&\hspace{-0.3cm}\leq \frac{1}{|\om|} \left(\int_{-d}^d    |\partial_{x_1} \mathbb{L}_{11}^{-1}w_1(d,x_2)|dx_2\right)^2 \sum_{ n \in J_{\mathrm{red}}(D_4)}\frac{16\alpha_n}{(a_1^2+|2\pi \tilde{n}|^2)^2}.\label{g_n_ineq}
\end{align}
For $\mathcal{Z}_{u,2,2}$ we follow similar steps by using $\mathbb{L}_{22}$, $V_2$, and $E_2$ instead of $\mathbb{L}_{11}$, $V_1$, and  $E_1$ respectively.
Specifically, we obtain that
\begin{align}\label{eq : gn ineq 2}
    \|\cha \left(\mathbb{L}_{22}^{-1} - \Gamma^\dagger(L_{22}^{-1})\right)w_2\|_2^2 \leq \frac{1}{|\om|} \left(\int_{-d}^d    |\partial_{x_1} \mathbb{L}_{22}^{-1}w_2(d,x_2)|dx_2\right)^2 \sum_{ n \in J_{\mathrm{red}}(D_4)}\frac{16\alpha_n}{(a_2^2+|2\pi \tilde{n}|^2)^2}.
\end{align}
 For $\mathcal{Z}_{u,2,3}$, we  begin similarly and define $g = \mathbb{1}_{\om} (\mathbb{L}_{22}^{-1} \mathbb{L}_{21}\mathbb{L}_{11}^{-1} - \Gamma^\dagger(L_{22}^{-1} L_{21} L_{11}^{-1}))w_2$. Now, using \eqref{eq : equality of the lik}, we obtain
 \begin{align}
     \mathbb{L}_{22}^{-1} \mathbb{L}_{21} \mathbb{L}_{11}^{-1} = \mathbb{L}_{22}^{-1} - \lambda_1 \mathbb{L}_{11}^{-1}.
 \end{align}
 Hence, if we let
 \begin{align}
     &(g_1)_n = \frac{1}{|\om|} \int_{\om} (\mathbb{L}_{11}^{-1} - \Gamma^\dagger(L_{11}^{-1})) w_2(x) e^{-2\pi i \tilde{n} \cdot x} \\
     &(g_2)_n = \frac{1}{|\om|} \int_{\om} (\mathbb{L}_{22}^{-1} - \Gamma^\dagger(L_{22}^{-1})) w_2(x) e^{-2\pi i \tilde{n} \cdot x},
 \end{align}
we then have
\begin{align}
    g_n = (g_2)_n - \lambda_1 (g_1)_n.
\end{align}
 Therefore, using \eqref{g_n_ineq} and \eqref{eq : gn ineq 2}, it follows that
\begin{align*}
    |g_n| &\leq \frac{4}{a_2^2 + |2\pi \tilde{n}|^2}\int_{-d}^d |\partial_{x_1} \mathbb{L}_{22}^{-1} w_2(d,x_2)| dx_2 + \frac{4\lambda_1}{a_1^2 + |2\pi \tilde{n}|^2}\int_{-d}^d |\partial_{x_1} \mathbb{L}_{11}^{-1} w_2(d,x_2)| dx_2 \\
    &\leq \frac{4}{\min(a_1,a_2)^2 + |2\pi \tilde{n}|^2} \left( \int_{-d}^d |\partial_{x_1} \mathbb{L}_{22}^{-1} w_2(d,x_2)|dx_2 + \lambda_1 \int_{-d}^d |\partial_{x_1} \mathbb{L}_{11}^{-1} w_2(d,x_2)| dx_2\right)
\end{align*}
for all $n \in J_{\mathrm{red}}(D_4)$.
 Consequently, using Parseval's identity, we get
{\small\begin{align}
     &\|g\|_{2}^2 = |\om| \sum_{n \in J_{\mathrm{red}}(D_4)} \alpha_n |g_n|^2 \\
     &\leq \frac{16}{|\om|} \sum_{n \in J_{\mathrm{red}}(D_4)}\hspace{-0.14cm}\frac{\alpha_n}{(\min(a_1,a_2)^2 + |2\pi \tilde{n}|^2)^2} \left(\int_{-d}^d |\partial_{x_1} \mathbb{L}_{22}^{-1} w_2(d,x_2)| dx_2 + \lambda_1\int_{-d}^d |\partial_{x_1} \mathbb{L}_{11}^{-1} w_2(d,x_2)| dx_2\right)^2.
 \end{align}}
We now wish to compute $\sum_{n \in J_{\mathrm{red}}(D_4)} \frac{\alpha_n}{(a_j^2 + |2\pi\tilde{n}|^2)^2}$ for $j = 1,2$. We define $g_j(x) \bydef \frac{1}{\left(\frac{a_j^2}{4\pi^2} + \left|x\right|^2\right)}$, and write
\begin{align}
    \nonumber \sum_{n \in J_{\mathrm{red}}(D_4)} \frac{\alpha_n}{(a_j^2 + |2\pi\tilde{n}|^2)^2} = \sum_{n \in \mathbb{Z}^2} \frac{1}{(a_j^2 + |2\pi\tilde{n}|^2)^2} = \sum_{n \in \mathbb{Z}^2} \frac{1}{\left(a_j^2 + \left|\frac{\pi n}{d}\right|^2\right)^2}
    = \frac{1}{16\pi^4} \sum_{n \in \mathbb{Z}^2} g_j\left(\frac{n}{2d}\right)^2.
\end{align}
Using Lemma \ref{lem : riemann sum}, we obtain that
\begin{align}
    \nonumber \frac{1}{16\pi^4} \sum_{n \in \mathbb{Z}^2} g_j\left(\frac{n}{2d}\right)^2 &= \frac{1}{16\pi^4}\left[4d^2 \|g_j\|_{2}^2 + g_j(0)^2 + 8d \int_{0}^{\infty} g_j(x_1,0)^2 dx_1\right] \\ \nonumber
    &= \frac{1}{16\pi^4}\left[4d^2\left(\frac{1}{4a_j^2\pi}\right) + \frac{16\pi^4}{a_j^4} + 8d \left(\frac{2\pi^4}{a_j^3}\right)\right] \\ 
    &= \frac{d^2}{16a_j^2\pi^5} + \frac{1}{a_j^4} + \frac{d}{a_j^3} \bydef \sqrt{C_j}.\label{Sigma_j_val}
\end{align}
Then, we need to compute integrals of the form $\int_{-d}^d |\partial_{x_1} \mathbb{L}_{jj}^{-1} w_i(d,x_2)| dx_2$. For this, we use Lemma \ref{lem : first_int_Z_u_2}.
This concludes the computation of $\mathcal{Z}_{u,2}$.
\end{proof} 
\begin{remark}\label{rem : Zu2}
The computation of $\mathcal{Z}_{u,1}$  presented in Lemma \ref{lem : lemma Z11full} follows similar steps as the one presented in Theorem 3.9 of \cite{unbounded_domain_cadiot}. On the other hand, a different direction was chosen for the treatment of $\mathcal{Z}_{u,2}$. In particular, it appears that following Lemma \ref{lem : lemma Z12full} can provide a sharper upper bound than the one introduced in \cite{unbounded_domain_cadiot}. Indeed, by construction, one must have $\mathcal{Z}_{u,2} \geq \mathcal{Z}_{u,1}$ using Theorem 3.9 of \cite{unbounded_domain_cadiot}. Instead, using the above lemma, this is not necessarily the case anymore. In practice, we actually obtain $\mathcal{Z}_{u,2} \ll \mathcal{Z}_{u,1}$ (see \eqref{eq : values for the bounds} for instance), improving the computation of the bound $\mathcal{Z}_{u,2}$. Such an improvement of the bound can, in practice, facilitate the establishment of a successful computer-assisted proof.
\end{remark}

\subsection{Proof of Periodic Solutions}

In the previous sections, we leveraged the properties of Fourier series in order to approximate objects on the unbounded domain $\R^2$. Specifically, recall that our approximate solution $\mathbf{u}_0$ is constructed in Section \ref{def : construction of u0} thanks to its Fourier coefficients, that is $\mathbf{u}_0 \bydef \bgam^\dagger\left(\mathbf{U}_0\right)$. By construction, $\mathbf{u}_0$ is smooth and has its support contain in $\om = (-d,d)^2$. Consequently, we would like to conclude about the existence of a periodic solution of period $2q$, for all $q \in [d,\infty]$, when proving the existence of a localized pattern using Theorem \ref{th: radii polynomial}. This problem has been treated in \cite{sh_cadiot} for the case of the Swift-Hohenberg PDE and it was proven that, under certain conditions to satisfy, that there actually exists an unbounded branch of periodic solutions, parametrized by their period, which tends to the localized pattern as the period goes to infinity. We derive in the next theorem a similar result in the case of the Gray-Scott model.
\begin{theorem}[Family of periodic solutions]\label{th : radii periodic}
Let $A : \ell^2_{D_4} \to \mathscr{h}$ be defined  in Section  \ref{sec : the operator A} and let $\Omega_q \bydef (-q,q)^2$ for all $q>0$. Then, let $\mathcal{Y}_0$, $(\mathcal{Z}_{2,1},\mathcal{Z}_{2,2}, \mathcal{Z}_{2,3}$)  and $(\mathcal{Z}_1,\mathcal{Z}_u)$ be the bounds defined in Sections \ref{sec : Y0}, \ref{sec : Z2} and \ref{sec : Z1full} respectively. Also let 
\begin{align}
    \nonumber &\widehat{\kappa}_{0,1} \bydef   \sqrt{\left(\lambda_1 \widehat{\kappa}_2+ \sqrt{\frac{1}{4\pi \lambda_2} + \frac{1}{4d^2\lambda_2^2}+\frac{1}{2d}\frac{\pi}{\sqrt{\lambda_2}}}\right)^2 + \frac{1}{4\pi \lambda_2} + \frac{1}{4d^2\lambda_2^2}+\frac{1}{2d}\frac{\pi}{\sqrt{\lambda_2}}} \\
    &\widehat{\kappa}_{0,2} \bydef   \sqrt{2}\sqrt{\frac{1}{4\pi \lambda_2} + \frac{1}{4d^2\lambda_2^2}+\frac{1}{2d}\frac{\pi}{\sqrt{\lambda_2}}}.\label{def : kappa01_02}
\end{align}
Using \eqref{def : kappa01_02}, let 
\begin{align}
    &\widehat{\kappa}_2 \bydef \sqrt{\frac{1}{4\pi \lambda_1} + \frac{1}{4d^2} + \frac{1}{2d}\frac{\pi}{\sqrt{\lambda_1}}}\\
    &\widehat{\kappa}_3 \bydef \sqrt{2}\min\left\{ \frac{\widehat{\kappa}_2^2}{\lambda_2} , ~ \widehat{\kappa}_2 \sqrt{\frac{1}{4\pi \lambda_2} + \frac{1}{4d^2\lambda_2^2}+\frac{1}{2d}\frac{\pi}{\sqrt{\lambda_2}}}\right\}  \\
    &\widehat{\kappa}_0 \bydef \min\left\{\max(\widehat{\kappa}_{0,1},\widehat{\kappa}_{0,2}) , ~\widehat{\kappa}_2 \frac{1}{\lambda_2}\left(\left({1-\lambda_2\lambda_1}\right)^2 +  1 \right)^{\frac{1}{2}}\right\}.\label{eq : kappas in the Theorem periodic solutions}
\end{align}
Now, define $\widehat{\mathcal{Z}}_1$ and $\widehat{\mathcal{Z}}_2(r)$ as
\begin{align*}
    \widehat{\mathcal{Z}}_1 &\bydef \mathcal{Z}_1 +\max\{1, \|B_{11}^N\|_{2}\}\mathcal{Z}_{u}\\
     \widehat{\mathcal{Z}}_2(r) &\bydef  2\sqrt{\varphi(\mathcal{Z}_{2,1},\mathcal{Z}_{2,2},\mathcal{Z}_{2,2},\mathcal{Z}_{2,3})} \sqrt{\widehat{\kappa}_2^2 + 4\widehat{\kappa}_0^2} + 3\widehat{\kappa}_3 \max\{1, \|B_{11}^N\|_{2}\} r
\end{align*}
for all $r>0$, and $\widehat{\mathcal{Y}}_0 \bydef \mathcal{Y}_0$.
If there exists $r>0$ such that  
\begin{align}\label{eq : radii condition periodic}
    \frac{1}{2}\widehat{\mathcal{Z}}_2(r)r^2 - (1-\widehat{\mathcal{Z}}_1)r + \widehat{\mathcal{Y}}_0 <0, \ \mathrm{and} \ \widehat{\mathcal{Z}}_2(r) r + \widehat{\mathcal{Z}}_1 < 1,
\end{align}
then there exists a smooth curve 
\[
\left\{\tilde{\mathbf{u}}(q) : q \in [d,\infty]\right\} \subset C^\infty(\R^2) \times C^\infty(\R^2)
\]
such that $\tilde{\mathbf{u}}(q)$ is a $D_4$-symmetric periodic solution to \eqref{eq : gray_scott cov} with period $2q$ in both variables. In particular, $\tilde{\mathbf{u}}(\infty)$ is a localized pattern on $\R^2.$
\end{theorem}
\begin{proof}
Following \eqref{kappas_with_norms}, we proved that  
 \begin{align*}
        \|\mathbb{G}_2(\mathbf{u})\|_2 \leq  \kappa_2 \|\mathbf{u}\|_{\mathcal{H}}^2 ~~\text{ and }~~
         \|\mathbb{G}_3(\mathbf{u})\|_2 \leq \kappa_3 \|\mathbf{u}\|_{\mathcal{H}}^3, ~~\text{ for all } \mathbf{u} \in \mathcal{H}.
    \end{align*}
    where
\begin{align}
    \kappa_2 = \left\|\frac{1}{l_{11}}\right\|_{2} ~~ \text{ and } ~~ \kappa_3 = \sqrt{2}\min\left\{ \frac{1}{\lambda_2}\left\|\frac{1}{l_{11}}\right\|_{2}^2 , ~ \left\|\frac{1}{l_{11}}\right\|_{2} \left\|\frac{1}{l_{22}}\right\|_{2}\right\}. \label{eq : kappa2 and 3 in lemma}
\end{align}
Similarly, following the proof of Lemma \ref{lem : banach algebra} we proved that
\begin{align}
\|u_1v_2\|_2 \leq \kappa_0 \|\mathbf{u}\|_{\mathcal{H}} \|\mathbf{v}\|_{\mathcal{H}}
\end{align}
for all $\mathbf{u}= (u_1,u_2), \mathbf{v}=(v_1,v_2) \in \mathcal{H}$, where
{\small\begin{align}
\kappa_0 =  \min\left\{\max\left\{ \left( \left(\left\|\frac{\lambda_1}{l_{11}}\right\|_{2} +  \left\|\frac{1}{l_{22}}\right\|_2\right)^2 + \left\|\frac{1}{l_{22}}\right\|_2^2\right)^{\frac{1}{2}}, ~ \left\|\frac{\sqrt{2}}{l_{22}}\right\|_2\right\}, ~ \left\|\frac{1}{l_{11}}\right\|_2\frac{1}{\lambda_2}\left(\left({1-\lambda_2\lambda_1}\right)^2 +  1 \right)^{\frac{1}{2}}\right\}.\label{eq : kappa0 in lemma}
\end{align}}
Now, notice that $\kappa_0, \kappa_2$ and $\kappa_3$ defined above depend on the integral computations of $\|\frac{1}{l_{11}}\|_2$ and $\|\frac{1}{l_{22}}\|_2$. Following the proof of Theorem 3.7 in \cite{sh_cadiot}, we need to control the constants $\mathcal{K}_1$ and $\mathcal{K}_2$ given by
\begin{align*}
    \mathcal{K}_1 \bydef \sup_{q \in [d,\infty)} \frac{1}{\sqrt{|\om|}} \left(\sum_{n \in \mathbb{Z}^2} \frac{1}{l_{11}\left(\frac{n}{2q}\right)^2} \right)^{\frac{1}{2}} ~~\text{and}~~
    \mathcal{K}_2 \bydef \sup_{q \in [d,\infty)} \frac{1}{\sqrt{|\om|}} \left(\sum_{n \in \mathbb{Z}^2} \frac{1}{l_{22}\left(\frac{n}{2q}\right)^2} \right)^{\frac{1}{2}}
\end{align*}
and replace all instances of $\|\frac{1}{l_{11}}\|_2$ and $\|\frac{1}{l_{22}}\|_2$ in \eqref{eq : kappa2 and 3 in lemma} and \eqref{eq : kappa0 in lemma} by $\mathcal{K}_1$ and $\mathcal{K}_2$ respectively. 
In particular, if we can show that
\begin{align}
    &\widehat{\kappa}_2 \geq \mathcal{K}_1 \\
    &\widehat{\kappa}_3 \geq \sqrt{2} \min\left\{\frac{1}{\lambda_2} \widehat{\kappa}_2^2, \widehat{\kappa}_2\mathcal{K}_2\right\} \\
    &\widehat{\kappa}_0 \geq \min\left\{\max\left\{\left[\left(\lambda_1 \widehat{\kappa}_2 + \mathcal{K}_2\right)^2 + \mathcal{K}_2^2\right]^{\frac{1}{2}},\sqrt{2}\mathcal{K}_2\right\}, \widehat{\kappa}_2\frac{1}{\lambda_2} ((1-\lambda_2 \lambda_1)^2 + 1)^{\frac{1}{2}}\right\},\label{K_j_defs_kappas}
\end{align}
then, similarly as what was achieved in \cite{sh_cadiot}, we conclude the proof using a fixed point argument.

First, observe that $\frac{1}{l_{11}^2}$ and $ \frac{1}{l_{22}^2}$ are $D_4$-symmetric and decreasing in $|\xi|$. Hence, we can apply Lemma \ref{lem : riemann sum} to $\mathcal{K}_1$ and $\mathcal{K}_2$. Using \eqref{Sigma_j_val}, we obtain 
\begin{align}
    \mathcal{K}_1 \leq \sqrt{\left\|\frac{1}{l_{11}}\right\|_{2}^2 + \frac{1}{4d^2l_{11}(0)^2} + \int_{0}^{\infty} \frac{1}{l_{11}(x_1,0)^2} dx_1}
    &\leq \sqrt{\frac{1}{4\pi \lambda_1} + \frac{1}{4d^2} + \frac{d}{\sqrt{\lambda_1}}},\label{int_comp_periodic1}
    \end{align}
    \begin{align}
    \mathcal{K}_2 \leq \sqrt{\left\|\frac{1}{l_{22}}\right\|_{2}^2 + \frac{1}{4d^2l_{22}(0)^2} + \int_{0}^{\infty} \frac{1}{l_{22}(x_1,0)^2} dx_1}
    &\leq \sqrt{\frac{1}{4\pi\lambda_2} + \frac{1}{4d^2\lambda_2^2} + \frac{d}{\lambda_2^{\frac{3}{2}}}}.\label{int_comp_periodic2}
\end{align}
 This proves the theorem.
\end{proof}
Theorem \ref{th : radii periodic} provides the existence of a branch of periodic solutions if \eqref{condition radii polynomial} is satisfied for the newly defined bounds $\widehat{\mathcal{Y}}_0, \widehat{\mathcal{Z}}_1, \widehat{\mathcal{Z}}_2$. In particular, given $j \in \{0,2,3\}$ and using the definitions of $\hat{\kappa}_j$,  we have that $\kappa_j = \hat{\kappa}_j + \mathcal{O}(\frac{1}{d}).$ Consequently, this implies that $\mathcal{Z}_2$ and $\widehat{\mathcal{Z}}_2$ are close if $d$ is sufficiently big. Similarly, since $\widehat{\mathcal{Z}}_1 = \mathcal{Z}_1 + \max\{1, \|B_{11}^N\|_{2}\}\mathcal{Z}_{u}$, we obtain that $\widehat{\mathcal{Z}}_1 \approx \mathcal{Z}_1$ if $\mathcal{Z}_u$ is sufficiently small (which is the case as $d$ gets bigger). This implies that, in practice, if $d$ is big enough then  \eqref{eq : radii condition periodic} is easily satisfied if \eqref{condition radii polynomial} is. Consequently, one can prove the existence of a branch of periodic solutions limiting the localized pattern with very light additional analysis.
\section{The scalar case \texorpdfstring{$\lambda_1 \lambda_2 = 1$}{red}}\label{sec : bounds_reduced}
In this section, we focus on studying the special case $\lambda_1 \lambda_2 =1$. Recall that in this case, \eqref{eq : gray_scott cov} can be studied as the scalar equation \eqref{gray_scott_reduced}
\begin{align}\label{eq : reduced equation}
    \lambda_1 \Delta u + u^2 - \lambda_1 u^3 - u =0.
\end{align}
Recalling that $\mathbb{L}_{11}$ is defined as in \eqref{definition_of_L}, we introduce the \emph{reduced}  zero-finding problem $\mathbb{F}_r(u) = 0$ where $\mathbb{F}_r$ is defined as 
\begin{align}
    \mathbb{F}_r(u) \bydef \mathbb{L}_{11}u + \mathbb{G}_r(u) 
\end{align}
 and where
\begin{align}
    \mathbb{G}_r(u) \bydef (1 - \lambda_1 u)u^2
\end{align}
is the \emph{reduced} nonlinear term. In particular, \eqref{eq : reduced equation} can be studied using the set-up developed in \cite{unbounded_domain_cadiot}.
To match the analysis derived in the aforementioned paper, we re-define the Hilbert space $\mathcal{H}$ as follows
 \begin{align*}
     \mathcal{H} \bydef \left\{u \in L^2(\R^2), ~ \|u\|_{\mathcal{H}} <\infty\right\}
 \end{align*}
 where we also re-define the norm $\|\cdot\|_{\mathcal{H}}$ as 
\begin{align*}
     \|u\|_{\mathcal{H}} \bydef \|\mathbb{L}_{11}u\|_2
\end{align*}
for all $u \in \mathcal{H}.$ Similarly, we re-define $\mathcal{H}_{D_4}$ as the restriction of $\mathcal{H}$ to $D_4$-symmetric functions.
 We now prove that $\mathbb{F}_r$ is well-defined on $\mathcal{H}$ by proving the well-definedness of $\mathbb{G}_r: \mathcal{H} \to L^2(\mathbb{R}^2)$. In particular, we prove that the product from $\mathcal{H} \times \mathcal{H}$ to $L^2(\R)$ is a bounded bilinear operator.
\begin{lemma}\label{lem : banach algebra reduced}
Let $u,v,w \in \mathcal{H}$ and recall $\kappa_2 = \frac{1}{2\sqrt{\lambda_1 \pi}}$ from Corollary \ref{cor : banach algebra1}. Then,
\begin{align}
    \|uv\|_{2} \leq \kappa_2 \|u\|_{\mathcal{H}} \|v\|_{\mathcal{H}} ~~\text{and}~~ \|uvw\|_{2} \leq \kappa_2^2 \|u\|_{\mathcal{H}} \|v\|_{\mathcal{H}} \|w\|_{\mathcal{H}}.
\end{align}
\end{lemma}
\begin{proof}
The proof can be found in \cite{unbounded_domain_cadiot}. We use that $\max_{\xi \in \mathbb{R}^2} \left|\frac{1}{l_{11}}\right| = 1$ and $\left\|\frac{1}{l_{11}}\right\|_2 = \frac{1}{2\sqrt{\lambda_1 \pi}}$.
\end{proof}
Let $L_{11}$ be the Fourier coefficients representation of $\mathbb{L}_{11}$ on $\om$ as defined in Section \ref{sec : periodic_spaces}. Then, we re-define the Hilbert space $\mathscr{h}$ as 
\[
\mathscr{h} \bydef \left\{U \in \ell^2(J_{\mathrm{red}}(D_4)), ~ \|U\|_{\mathscr{h}} < \infty \right\}
\]
where $\|U\|_{\mathscr{h}} \bydef \|L_{11}U\|_2$ for all $U \in \mathscr{h}$.  Then, define  $F_r$ and $G_r$ as the periodic equivalents of $\mathbb{F}_r$ and $\mathbb{G}_r$, as done in Section \ref{sec : periodic_spaces}. Specifically, we have 
\begin{align*}
    G_r(U) \bydef U*U - \lambda_1 U*U*U  ~~ \text{ and } ~~ F_r(U) \bydef L_{11}U + G_r(U)
\end{align*}
for all $U \in \mathscr{h}$.

Now that we introduced the required notations to study the scalar PDE \eqref{eq : reduced equation}, we are in a position to expose our computer-assisted analysis. Let us fix $N \in \mathbb{N}$  controlling the numerical truncation of both our operators and sequences approximations respectively. That is, we choose $N_0 = N$.  Then, following Section \ref{sec : numerical construction}, we assume that we where able to construct an approximate solution $u_0 \in \mathcal{H}_{D_4}$ such that $u_0 = \gamma^\dagger(U_0)$ (cf. \eqref{single_gamma_dagger} for a definition of $\gamma^\dagger$) for some $U_0 \in {\mathscr{h}}$ with $U_0 = \pi^{N} U_0$ (cf. \eqref{def : piN and pisubN} for a definition of $\pi^N$). In particular, an equivalent of Theorem \ref{th: radii polynomial} can be derived in the case of the scalar PDE $\mathbb{F}_r =0$. We state the result in the next theorem for convenience whose proof is given in \cite{unbounded_domain_cadiot}.
\begin{theorem}\label{th: radii polynomial reduced}
Let $\mathbb{A}_r : L^2_{D_4}(\mathbb{R}^2) \to \mathcal{H}_{D_4}$ be a bounded linear operator. Moreover, let $\mathcal{Y}_0, \mathcal{Z}_1$ be non-negative constants and let $\mathcal{Z}_2 : (0, \infty) \to [0,\infty)$ be a non-negative function such that
  \begin{align}
    \|\mathbb{A}_r\mathbb{F}_r(u_0)\|_{\mathcal{H}} & \le \mathcal{Y}_0\label{def : Y0r}\\
    \|I_d - \mathbb{A}_rD\mathbb{F}_r(u_0)\|_{\mathcal{H}} &\le \mathcal{Z}_1\label{def : Z1r}\\
    \|\mathbb{A}_r\left({D}\mathbb{F}_r(v) - D\mathbb{F}_r(u_0)\right)\|_{\mathcal{H}} &\le \mathcal{Z}_2(s)s, ~~ \text{for all } v \in \overline{B_s(u_0)} \text{ and all } s>0.\label{def : Z2r}
\end{align}  
If there exists $s>0$ such that
\begin{equation}\label{condition radii polynomial reduced}
    \frac{1}{2}\mathcal{Z}_2(s)s^2 - (1-\mathcal{Z}_1)s + \mathcal{Y}_0 <0 \text{ and } \mathcal{Z}_1 + \mathcal{Z}_2(s)s < 1,
 \end{equation}
then there exists a unique $\tilde{u} \in \overline{B_s(u_0)} \subset \mathcal{H}_{D_4}$ such that $\mathbb{F}_r(\tilde{u})=0$. 
\end{theorem}
Having in mind the above theorem, we need to construct an approximate inverse $\mathbb{A}_r : L^2_{D_4}(\R^2) \to \mathcal{H}_{D_4}$ for $D\mathbb{F}_r(u_0)$ and compute the corresponding bounds $\mathcal{Y}_0, \mathcal{Z}_1$ and $\mathcal{Z}_2$. This is achieved in the next section.
\subsection{Computer-assisted analysis for the case \texorpdfstring{$\lambda_1\lambda_2 =1$}{lam1lam2}}
 Following the construction presented in Section \ref{sec : the operator A}, we start by building $B_r^N$, which is a numerical approximate inverse of $\pi^N DF_r(U_{0})L_{11}^{-1} \pi^N$. Then, we define
\begin{align}
    B_r \bydef \pi_N + B_r^N \text{ and } \mathbb{B}_r \bydef \mathbb{1}_{\R^2 \setminus \om} + \Gamma^\dagger(B_r).
\end{align}
Finally, the operator $\mathbb{A}_r : L^2_{D_4}(\R^2) \to \mathcal{H}_{D_4}$ is given by
\begin{align}
    \mathbb{A}_r \bydef \mathbb{L}_{11}^{-1} \mathbb{B}_r.\label{def : A_r}
\end{align}
In particular, $\mathbb{A}_r : L^2_{D_4}(\R^2) \to \mathcal{H}_{D_4}$ is a bounded linear operator since $\mathbb{L}_{11} : \mathcal{H}_{D_4} \to L^2_{D_4}(\mathbb{R}^2)$ is an isometric isomorphism. Moreover, using \cite{unbounded_domain_cadiot}, we have 
\begin{align}\label{eq : equality of norms A and B and BN}
    \|\mathbb{A}_r\|_{2,\mathcal{H}} = \|\mathbb{B}_r\|_2 = \max\{1, \|B^N_r\|_2\}.
\end{align}
We  now present the computation of the required bounds for the application of Theorem \ref{th: radii polynomial reduced}.
First, the bound $\mathcal{Y}_0$ can be computed thanks to Parseval's identity. We present the result in the following lemma, for which a proof is given in  \cite{unbounded_domain_cadiot}.
\begin{lemma}\label{lem : Y_0_reduced}
Let $\mathcal{Y}_0 > 0$ be such that
\begin{align}
    \mathcal{Y}_0 \bydef |\om|^{\frac{1}{2}}\left(\|B_{r}^N  F_r(U_{0})\|_{2}^2 + \|(\pi^{3N}-\pi^N) G_r(U_{0})\|_{2}^2 \right)^{\frac{1}{2}}.
\end{align}
Then, $\|\mathbb{A}_{r}\mathbb{F}_r(u_{0})\|_{\mathcal{H}} \leq \mathcal{Y}_0$.
\end{lemma}
Now, similarly as what was achieved in Lemma \ref{lem : Bound Z_2}, $\mathcal{Z}_2(s)$ is obtained using the mean value inequality for Banach spaces.
\begin{lemma}\label{lem : Z_2_reduced}
Let $s > 0$ and let $\mathcal{Z}_2(s) > 0$ be such that
\begin{align}
    \mathcal{Z}_2(s) \bydef \max\{1,\|B_{r}^N\|_{2}\} \left(2 \kappa_2 + 3\lambda_1 \kappa_2^2 s\right) + 6\lambda_1 \left(\|\mathbb{U}_{0}^{*} (B_{r}^N)^{*}\|_{2}^2 + \|U_{0}\|_{1}^2\right)^{\frac{1}{2}} \kappa_2,
\end{align}
where $\mathbb{U}_0$ is the discrete convolution operator associated to $U_0$ (cf. \eqref{def : discrete conv operator} for a definition).
Then, $\|\mathbb{A}_{r}(D\mathbb{F}_r(v) - D\mathbb{F}_r(u_{0}))\|_{\mathcal{H}} \leq \mathcal{Z}_2(s) s$ for all $v \in \overline{B_s(u_0)}$.
\end{lemma}
\begin{proof}
Let $h = v - u_{0} \in B_s(u_{0})$. Then, we get
\begin{align}
    \|\mathbb{A}_{r}(D\mathbb{F}_r(v) - D\mathbb{F}_r(u_{0}))\|_{\mathcal{H}} &= \|\mathbb{B}_{r}(D\mathbb{G}_r(h+u_{0}) - D\mathbb{G}_r(u_{0}))\|_{\mathcal{H},2} \\ \nonumber
    &= \|\mathbb{B}_{r}(2\mathbb{h} - 3\lambda_1 \mathbb{h}^2 - 6\lambda_1 \mathbb{h} \mathbb{u}_{0})\|_{\mathcal{H},2} \\ \nonumber
    &\leq \|\mathbb{B}_{r}\|_{2} (2\|\mathbb{h}\|_{\mathcal{H},2} + 3\lambda_1 \|\mathbb{h}^2\|_{\mathcal{H},2}) + 6\lambda_1 \|\mathbb{B}_{r} \mathbb{u}_{0}\|_{2} \|\mathbb{h}\|_{\mathcal{H},2}. 
\end{align}
Now, using Lemma \ref{lem : banach algebra reduced} combined with \eqref{eq : equality of norms A and B and BN}, we get
\begin{align*}
    \|\mathbb{B}_{r}\|_{2} (2\|\mathbb{h}\|_{\mathcal{H},2} + 3\lambda_1 \|\mathbb{h}^2\|_{\mathcal{H},2}) \leq \max\{1,\|B_r^N\|_{2}\} \left(2 \kappa_2 s + 3\lambda_1 \kappa_2^2 s^2\right).
\end{align*}
Now, recalling that $u_0 = \gamma^\dagger(U_0)$ and $\mathbb{B} = \mathbb{1}_{\mathbb{R}^2 \setminus \om} + \Gamma^\dagger(B_r)$, we get
\begin{align*}
    \|\mathbb{B}_r\mathbb{u}_0\|_2 = \| \Gamma^\dagger(B_r)\mathbb{u}_0\|_2 = \|B_r \mathbb{U}_0\|_2,
\end{align*}
where we used Lemma \ref{lem : gamma and Gamma properties} for the last step.  Recalling that $B_r = \pi_N + B^N_r$, we get 
\begin{align*}
    \|B_r\mathbb{U}_0\|_2^2 &\leq \|\pi^NB_r\mathbb{U}_0\|_2^2 + \|\pi_NB_r\mathbb{U}_0\|_2^2 
    =  \|B^N_r\mathbb{U}_0\|_2^2 + \|\pi_N\mathbb{U}_0\|_2^2 
    \leq \|B^N_r\mathbb{U}_0\|_2^2 + \|\mathbb{U}_0\|_2^2.
\end{align*}
We conclude the proof using that $\|\mathbb{U}_0\|_2 \leq \|U_0\|_1$ (cf. \eqref{young_inequality}) and $\|B^N_r\mathbb{U}_0\|_2 = \|\mathbb{U}_0^*(B^N_r)^*\|_2.$
\end{proof}
Finally, we discuss the bound $\mathcal{Z}_1$. Define
\begin{align}
    v_0 \bydef 2u_{0} -3\lambda_1 u_{0}^2 ~~\text{and}~~ V_0 \bydef \gamma(v_0).\label{def : v_0}
\end{align}
Then, observe that
\begin{align}
    \mathbb{v}_0 = D\mathbb{G}_r(u_{0}) ~~ \text{and}~~ \mathbb{V}_0 = DG_r(U_{0})
\end{align}
where $\mathbb{v}_0$ and $\mathbb{V}_0$ are the multiplication operators for $v_0$ and $V_0$ respectively. 
Using these notations, we present the computation of the bound $\mathcal{Z}_1$ in the next lemma, for which a proof is given in  \cite{unbounded_domain_cadiot}.
\begin{lemma}\label{lem : Z_full_1 reduced}
Let $Z_1, \mathcal{Z}_{u,1}$ and $\mathcal{Z}_{u,2}$ be non-negative bounds  satisfying
\begin{align}
\nonumber
Z_1 \geq \|I_d - B_r(I_d + \mathbb{V}_0L_{11}^{-1})\|_{2}\text{,}~~\mathcal{Z}_{u,1} \geq \|\mathbb{1}_{\mathbb{R}^2 \setminus \om} \mathbb{v}_0 \mathbb{L}_{11}^{-1}\|_{2}\text{ and}~~
    \mathcal{Z}_{u,2} \geq \|\mathbb{1}_{\om}\mathbb{v}_0 (\Gamma^\dagger(L_{11}^{-1}) - \mathbb{L}_{11}^{-1})\|_{2}.
\end{align}
\\
Then, defining $\mathcal{Z}_u  \bydef \sqrt{\mathcal{Z}_{u,1}^2 + \mathcal{Z}_{u,2}^2} $ and $\mathcal{Z}_1>0$ as
\begin{equation}\label{eq : first definition Z1 reduced}
    \mathcal{Z}_1 \bydef Z_1 + \max\{1, \|B_{r}^N\|_{2}\} \mathcal{Z}_{u},
\end{equation}
it follows that $ \|I_d -{\mathbb{A}_r}D\mathbb{F}_r(u_{0})\|_{\mathcal{H}} \leq \mathcal{Z}_1.$
\end{lemma}
Let us now compute $Z_1$ and $\mathcal{Z}_u$. We begin with $Z_1$.
\begin{lemma}\label{lem : Z1_bound reduced}
Let $M_{r}^N \bydef \pi^N + \mathbb{V}_0L_{11}^{-1}$ and $M_r \bydef \pi_N + M_r^N = I_d + \mathbb{V}_0L_{11}^{-1}$. Let $\varphi$ be defined as in \eqref{definition_of_phi}. Let $Z_1 > 0$ be such that
\begin{equation}\label{ineq : Z_reduced}
    Z_1 \bydef \varphi(Z_{1,1},Z_{1,2},Z_{1,3},Z_{1,4})
\end{equation}
where 
\begin{align}
&Z_{1,1} \bydef \sqrt{\|(\pi^N - B_{r}^NM_{r}^N)(\pi^N -  (M_{r}^N)^*(B_{r}^N)^*)\|_{2}} \\
    &Z_{1,2} \bydef \max_{n \in J_{\mathrm{red}}(D_4) \setminus I^N} \frac{1}{|l_{11}(\tilde{n})|} \sqrt{\|B_{r}^N \mathbb{V}_0\pi_N \mathbb{V}_0B_{r}^N\|_{2}} \\
    &Z_{1,3} \bydef \sqrt{\|\pi^N L_{11}^{-1} \mathbb{V}_0 \pi_N \mathbb{V}_0 L_{11}^{-1} \pi^N\|_{2}} \\
    &Z_{1,4} \bydef \max_{n \in J_{\mathrm{red}}(D_4) \setminus I^N }\frac{1}{|l_{11}(\tilde{n})|} \|V_0\|_{1}.
\end{align}
Then, $\|I_d - B_{r}M_{r}\|_{2} \leq Z_1$.
\end{lemma}
\begin{proof}
We begin as in the proof of Lemma \ref{lem : Z1_bound}. 
\begin{align}
    \nonumber I - B_{r}M_{r} &= \begin{bmatrix}
        \pi^N (I - B_{r}M_{r})\pi^N & \pi^N (I - B_{r}M_{r})\pi_N \\
        \pi_N (I - B_{r}M_{r})\pi^N & \pi_N (I - B_{r}M_{r})\pi_N
    \end{bmatrix} \\ \nonumber
    &= \begin{bmatrix}
        \pi^N (I - B_{r}M_{r})\pi^N & -\pi^N B_{r}M_{r} \pi_N \\
        \pi_N (I - M_{r})\pi^N & \pi_N (I - M_{r})\pi_N
    \end{bmatrix} \\
    &= \begin{bmatrix}
         \pi^N - B_{r}^NM_{r}^N & -B_{r}^N \mathbb{V}_0 L_{11}^{-1} \pi_N \\ 
        -\pi_N \mathbb{V}_0 L_{11}^{-1}\pi^N & -\pi_N \mathbb{V}_0 L_{11}^{-1} \pi_N
    \end{bmatrix}.\label{step_in_Z1_reduced}
\end{align}
Now, we examine $\pi^N - B_{r}^NM_{r}^N$.
\begin{align}
     \|\pi^N - B_{r}^NM_{r}^N\|_{2}^2 &= \|(\pi^N - B_{r}^NM_{r}^N) (\pi^N - (M_{r}^N)^{*} (B_{r}^N)^{*})\|_{2} \bydef Z_{1,1}^2.\label{step_in_Z11_reduced}
\end{align}
Following this, we need to investigate the term $- B_{r}^N \mathbb{V}_0 L_{11}^{-1} \pi_N$.
\begin{align}
    \nonumber \|-B_{r}^N \mathbb{V}_0 L_{11}^{-1} \pi_N\|_{2}^2 = \|B_{r}^N \mathbb{V}_0 L_{11}^{-1} \pi_N\|_{2}^2 
    = \|\pi_N L_{11}^{-1} \mathbb{V}_0 B_{11}^N\|_{2}^2  
    &\leq \|\pi_N L_{11}^{-1}\|_{2}^2 \|\pi_N\mathbb{V}_0 B_{r}^N\|_{2}^2 \\ 
    &\hspace{-3.5cm}\leq \max_{n \in J_{\mathrm{red}}(D_4) \setminus I^N} \frac{1}{|l_{11}(\tilde{n})|^2} \|B_{r}^N \mathbb{V}_0\pi_N \mathbb{V}_0B_{r}^N\|_{2} \bydef Z_{1,2}^2.\label{step_in_Z12_reduced}
    \end{align}
\par Thirdly, we consider $-\pi_N \mathbb{V}_0 L_{11}^{-1} \pi^N$.
\begin{align}
    \|-\pi_N \mathbb{V}_0 L_{11}^{-1} \pi^N\|_{2}^2 = \|\pi_N \mathbb{V}_0 L_{11}^{-1} \pi^N\|_{2}^2 = \|\pi^N L_{11}^{-1} \mathbb{V}_0 \pi_N \mathbb{V}_0 L_{11}^{-1} \pi^N\|_{2} \bydef Z_{1,3}^2\label{step_in_Z1_3_reduced}
\end{align}
Lastly, we will consider $-\pi_N \mathbb{V}_0 L_{11}^{-1} \pi_N$.
\begin{align}
    \nonumber \|-\pi_N \mathbb{V}_0 L_{11}^{-1} \pi_N\|_{2} = \|\pi_N \mathbb{V}_0 L_{11}^{-1} \pi_N\|_{2} 
    = \|\pi_N L_{11}^{-1} \mathbb{V}_0 \pi_N\|_{2}
    &\leq \|\pi_N L_{11}^{-1}\|_{2}\|\pi_N \mathbb{V}_0\pi_N\|_{2} \\ \nonumber
    &\leq \max_{n \in J_{\mathrm{red}}(D_4)\setminus I^N} \frac{1}{|l_{11}(\tilde{n})|} \|\mathbb{V}_0\|_{2} \\ 
\nonumber &\leq \max_{n \in J_{\mathrm{red}}(D_4)\setminus I^N} \frac{1}{|l_{11}(\tilde{n})|} \|V_0\|_{1} \\
&\bydef Z_{1,4}\label{step_in_Z1_4_reduced}
\end{align}
where we use Young's inequality (cf. \eqref{young_inequality}) on the final step. With \eqref{step_in_Z11_reduced}, \eqref{step_in_Z12_reduced}, \eqref{step_in_Z1_3_reduced}, and \eqref{step_in_Z1_4_reduced} computed,  we use Lemma \ref{lem : full_matrix_estimate} to compute the $Z_1$ bound.
\end{proof}
\par We now consider the bound $\mathcal{Z}_u$. We can reuse some of the computations presented in Lemmas \ref{lem : lemma Z11full} and \ref{lem : lemma Z12full}. 
\begin{lemma}\label{lem : Zu_reduced}
Let $a_1$ and $a_2$ be defined as in \eqref{def : definition of a1 and a2}. Let $\mathcal{Z}_{u,1},\mathcal{Z}_{u,2} > 0$ be defined as
\begin{align}
    &\mathcal{Z}_{u,1} \bydef \frac{\sqrt{2}C_0(f_{11})(2\pi)^{\frac{1}{4}}e^{-a_1d}\sqrt{|\om|}}{a_1^{\frac{3}{4}}}\sqrt{(V_0,E_1* V_0)_2}\\
    &\mathcal{Z}_{u,2} \bydef \frac{4C_1}{\sqrt{|\om|}}\left(\mathcal{C}_{1,1} \sqrt{(V_0,E_1 * V_0)_2} + \mathcal{C}_{2,1} C(v_0)\right)
\end{align}
where $E_1$ is defined as in Lemma \ref{lem : lemma Z11full}, $C_{1}$, $\mathcal{C}_{1,1}, \mathcal{C}_{2,1},$ and $C(v_0)$ are defined as in Lemma \ref{lem : lemma Z12full}. Defining $\mathcal{Z}_u = \sqrt{\mathcal{Z}_{u,1}^2 + \mathcal{Z}_{u,2}^2}$, we get $\|\mathbb{v}_0(\Gamma^\dagger(L_{11}^{-1}) - \mathbb{L}_{11}^{-1})\|_{2} \leq \mathcal{Z}_{u}$.
\end{lemma}
\begin{proof}
Following very similar steps to those in the proofs of Lemmas \ref{lem : lemma Z11full} and \ref{lem : lemma Z12full}, one obtains the desired result for the bound $\mathcal{Z}_u$.
\end{proof}
With $\mathcal{Y}_0, \mathcal{Z}_2, $ and $\mathcal{Z}_1$ now computed, we can perform computer assisted proofs in \eqref{gray_scott_reduced}. Furthermore, we present the following result which is the equivalent of Theorem \ref{th : radii periodic} and which allows us to obtain a proof of a branch of periodic solutions in the reduced equation \eqref{gray_scott_reduced}.
\begin{theorem}[Family of periodic solutions in the reduced equation]\label{th : radii periodic reduced}
Let $A_r : \ell^2_{D_4}(\mathbb{R}^2) \to \mathscr{h}$ be defined  as in \eqref{def : A_r}. Then, let $\mathcal{Y}_0,\mathcal{Z}_2, \mathcal{Z}_1$, and $\mathcal{Z}_{u}$ be the bounds defined in Lemmas \ref{lem : Y_0_reduced}, \ref{lem : Z_2_reduced}, \ref{lem : Z_full_1 reduced}, and \ref{lem : Zu_reduced} respectively. Moreover, let 
\begin{align}
    &\widehat{\kappa}_2 \bydef \sqrt{\frac{1}{4\pi \lambda_1} + \frac{1}{4d^2} + \frac{1}{2d}\frac{\pi}{\sqrt{\lambda_1}}}\label{eq : kappas in the Theorem periodic solutions reduced}
\end{align}
and let $\widehat{\mathcal{Z}}_1$ and $\widehat{\mathcal{Z}}_2(r)$ be  bounds satisfying
\begin{align*}
    \widehat{\mathcal{Z}}_1 &\bydef \mathcal{Z}_1 +\max\{1, \|B_{r}^N\|_{2}\}\mathcal{Z}_{u}\\
     \widehat{\mathcal{Z}}_2(s) &\bydef  \max\{1,\|B_r^N\|_{2}\} (2\widehat{\kappa}_2 + 3\lambda_1 \widehat{\kappa}_2^2 s) + 6\lambda_1 (\|\mathbb{U}_0^* (B_r^N)^*\|_{2}^2 + \|U_0\|_{1}^2)^{\frac{1}{2}} \widehat{\kappa}_2.
\end{align*}
for all $s>0$. Finally, define $\widehat{\mathcal{Y}}_0 \bydef \mathcal{Y}_0$.
If there exists $s>0$ such that 
\begin{align}\label{eq : radii condition periodic reduced}
    \frac{1}{2}\widehat{\mathcal{Z}}_2(s)s^2 - (1-\widehat{\mathcal{Z}}_1)s + \widehat{\mathcal{Y}}_0 <0, \ \mathrm{and} \ \hat{\mathcal{Z}}_2(s) s + \hat{\mathcal{Z}}_1 < 1,
\end{align}
then there exists a smooth curve 
\[
\left\{\tilde{u}(q) : q \in [d,\infty]\right\} \subset C^\infty(\R^2)
\]
such that $\tilde{u}(q)$ is a $D_4$-symmetric periodic solution to \eqref{gray_scott_reduced} with period $2q$ in both variables. In particular, $\tilde{u}(\infty)$ is a localized pattern on $\R^2.$
\end{theorem}
\section{Constructive proofs of existence of localized patterns}\label{sec : CAPs}
In this section, we present computer-assisted proofs of four branches of periodic solutions limiting  a stationary localized pattern in the 2D Gray-Scott model. In particular, we compute the values of $\mathcal{Y}_0, \mathcal{Z}_2$, and $\mathcal{Z}_1$ for each of the four patterns, and prove that they satisfy \eqref{condition radii polynomial} in Theorem \ref{th: radii polynomial} (or \eqref{condition radii polynomial reduced} in Theorem \ref{th: radii polynomial reduced} for the case of the reduced equation). Then, we compute $\widehat{\mathcal{Z}}_2$ and $\widehat{\mathcal{Z}}_1$ and prove that  Theorem \ref{th : radii periodic} is applicable (or Theorem \ref{th : radii periodic reduced} for the case of the reduced equation). Furthermore, since our analysis is derived in the Hilbert space $\mathcal{H}_{D_4}$, we obtain a proof of the $D_4$-symmetry by construction.

In order to make use of the $D_4$-symmetry at the numerical level, we  created our own sequence structure available at \cite{dominic_D_4_julia}. This Julia repository is fully compatible with RadiiPolynomial.jl \cite{julia_olivier}. More specifically, \cite{dominic_D_4_julia} includes the construction of multiplication operators (e.g $\mathbb{V}_1^N$ defined as in \eqref{def : V_1 and V_2}), the convolution, the Laplacian, evaluation, and norms specifically for sequences indexed on the reduced set $J_{\mathrm{red}}(D_4)$. This means one can directly input a sequence with $D_4$-symmetry with the reduced set of coefficients, and the remaining operations will be done on the reduced set without any need to enforce it throughout the process.

 We now provide computer-assisted proofs of four different patterns as well as four unbounded branches of periodic solutions converging to these patterns as the period tends to infinity. For each pattern, we provide explicit bounds satisfying condition \eqref{condition radii polynomial} in Theorem \ref{th: radii polynomial}. Note that the names of the pattern are only descriptive. In particular, the ``spike" and ``ring" patterns are not proven to be radially symmetric. We however obtain proofs of the $D_4$-symmetry by using \cite{dominic_D_4_julia}.
\subsection{Spike Pattern in the scalar case when \texorpdfstring{$\lambda_2\lambda_1=1$}{lamgam1}} 
\begin{theorem}\label{th : proof_of_spike_R1}
Let $\lambda_2 = 9$, $\lambda_1 = \frac{1}{9}$.
Moreover, let $s_0 \bydef 0.0005$. Then there exists a unique solution $\tilde{u}$ to \eqref{gray_scott_reduced} in $\overline{B_{s_0}(u_{spike})} \subset \mathcal{H}_{D_4}$ and we have that $\|\tilde{u}-u_{spike}\|_{\mathcal{H}} \leq s_0$. That is, $\tilde{u}$ is at most $s_0$ away from the approximation shown in Figure \ref{fig : 1D case}.
\par In addition, there exists a smooth curve 
\[
\left\{\tilde{u}(q) : q \in [d,\infty]\right\} \subset C^\infty(\R^2)
\]
such that $\tilde{u}(q)$ is a periodic solution to \eqref{gray_scott_reduced} with period $2q$ in both directions. In particular, $\tilde{u}(\infty) = \tilde{u}$ is a localized pattern on $\R^2.$ 
\end{theorem} 
\begin{proof}
Choose $N = 20, d = 4$ where $N$ is the number of Fourier coefficients and $d$ allows to define $\om \bydef (-d,d)^2$. Then, following the construction in Section \ref{sec : numerical construction}, we build $u_0 = \gamma^\dagger(U_0)$ and define $u_{spike} \bydef u_0$. Next, we construct $B_r^N$ using the approach described in Section \ref{sec : bounds_reduced}. In particular, we prove that 
\begin{align}
    \|B_r^N\|_2 \leq 4.24406,
\end{align}
which implies that $\|\mathbb{A}_r\|_{2,\mathcal{H}} \leq 4.24406$ using \eqref{eq : equality of norms A and B and BN}. Using Lemma \ref{lem : banach algebra reduced}, we prove that
\begin{align}
\kappa_2 \bydef 0.846285.
\end{align}
This allows us to compute the $\mathcal{Z}_2(s)$-bound introduced in Lemma \ref{lem : Z_2_reduced}. Next, using the constants defined in Lemmas \ref{lem : constants_for_f} and \ref{lem : lemma Z12full}, we obtain
\begin{align}
    C_{0}(f_{11}) \bydef 56.7699 \text{,~~}
    C_{11}(f_{11}) \bydef 22.5597 \text{, and}~~C_{12}(f_{11}) \bydef 59.1361.
\end{align}
Finally, using \cite{julia_cadiot_blanco_GS}, we we choose $s_0 = 0.0005$ and define 
\begin{align}\label{eq : values for the bounds}
\mathcal{Y}_0 \bydef 3\times10^{-4}\text{, }~\mathcal{Z}_2(s_0) \bydef  13.6 \text{,}~Z_1 \bydef 0.124 \text{,}~
\mathcal{Z}_{u,1} \bydef 0.011 \text{, and}~\mathcal{Z}_{u,2} \bydef 5.6 \times 10^{-4}.
    \end{align}
In particular, the above values satisfy the inequalities required to apply Theorem \ref{th : radii periodic reduced}. Moreover, this allows us to define $\widehat{\kappa}_2, \widehat{\mathcal{Z}}_1,$ and $\widehat{\mathcal{Z}}_2(s)$ as 
\begin{align}
    \widehat{\kappa}_2 \bydef 1.382\text{,}~~
\widehat{\mathcal{Z}}_1 \bydef 0.22\text{, and}~~\widehat{\mathcal{Z}}_2(s_0) \bydef 22.2,
\end{align}
and prove that \eqref{condition radii polynomial reduced} in Theorem \ref{th : radii periodic reduced} is satisfied.
\end{proof}
 \begin{figure}[H]
\centering
 \begin{minipage}[H]{0.5\linewidth}
  \centering\epsfig{figure=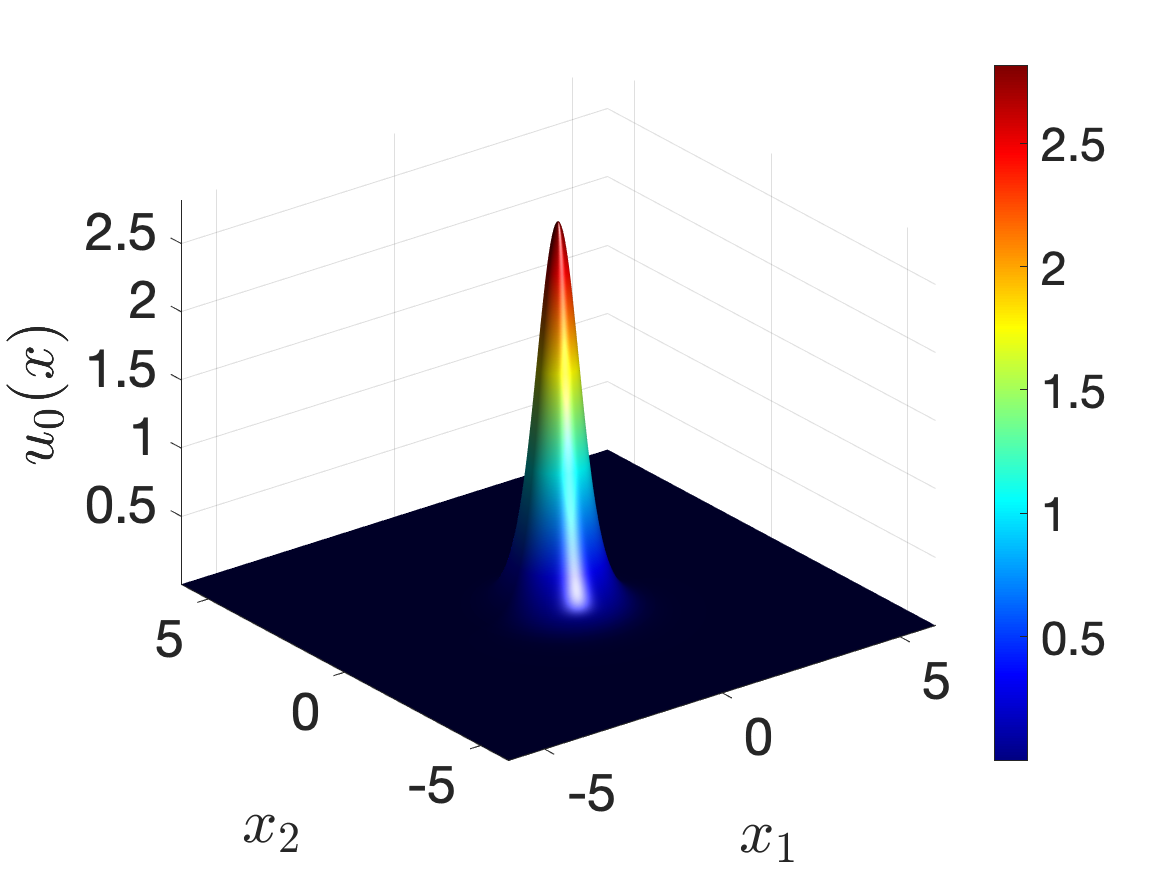,width=\linewidth}
 \end{minipage} 
 \caption{Approximation of a ``spike pattern" in the Gray-Scott reduced equation, $\lambda_2 = 9, \lambda_1 = \frac{1}{9}$.}
 \label{fig : 1D case}
 \end{figure}

 \subsection{Spike Pattern in the full system away from \texorpdfstring{$\lambda_2\lambda_1 =1$}{lamgamis1}}
 The spike pattern away from $\lambda_2 \lambda_1 = 1$ was found by taking the solution proven in Theorem \ref{th : proof_of_spike_R1} as the initial guess for our Newton method, changing the value of $\lambda_2$, and allowing Newton to converge to a new solution. 
 \begin{theorem}\label{th : proof_of_spike_away}
Let $\lambda_2 =10, \lambda_1 =\frac{1}{9}$. Moreover, let $r_0 \bydef 6 \times 10^{-6}$. Then there exists a unique solution $\tilde{\mathbf{u}}$ to \eqref{eq : gray_scott cov} in $\overline{B_{r_0}(\mathbf{u}_{spike})} \subset \mathcal{H}_{D_4}$ and we have that $\|\tilde{\mathbf{u}}-\mathbf{u}_{spike}\|_{\mathcal{H}} \leq r_0$. That is, $\tilde{\mathbf{u}}$ is at most $r_0$ away from the approximation shown in Figure \ref{fig : spike pattern}.
\par In addition, there exists a smooth curve 
\[
\left\{\tilde{\mathbf{u}}(q) : q \in [d,\infty]\right\} \subset C^\infty(\R^2) \times C^\infty(\R^2)
\]
such that $\tilde{\mathbf{u}}(q)$ is a periodic solution to \eqref{gray_scott_reduced} with period $2q$ in both directions. In particular, $\tilde{\mathbf{u}}(\infty) = \tilde{\mathbf{u}}$ is a localized pattern on $\R^2.$ 
\end{theorem}
\begin{proof}
Choose $N_0 = 50, N = 20, d = 8$. Then, we perform the full construction described in Section \ref{sec : numerical construction} to build $\mathbf{u}_0 = \bgam^\dagger(\mathbf{U}_0)$ and define $\mathbf{u}_{spike} \bydef \mathbf{u}_0$. Next, we construct $B^N$ using the approach described in Section \ref{sec : the operator A}. Since the computations in Section \ref{sec : bounds} only require us to compute $\|B_{11}^N\|_{2}$, we find
\begin{align}
    \|B_{11}^N\|_{2} \leq 4.1398 .
\end{align}
Using Corollary \ref{cor : banach algebra1} and Lemma \ref{lem : banach algebra}, we prove that
\begin{align}
    \nonumber \kappa_2 \bydef 0.846285 \text{,}~~\kappa_3 \bydef 0.101286 \text{, and}~~\kappa_0 \bydef 0.0851493.
\end{align}
This allows us to compute the $\mathcal{Z}_2(r)$ bound defined in Section \ref{sec : Z2}. Next, using the constants defined in Lemmas \ref{lem : constants_for_f} and \ref{lem : lemma Z12full}, we obtain
\begin{align}
    C_0(f_{11}) \bydef 56.7699
 \text{,}~~C_{0}(f_{22}) \bydef 6.22524 \text{,}~~C_{11}(f_{11}) \bydef 22.5597 \text{,}~~C_{12}(f_{11}) \bydef 59.1361 \text{,}
 \end{align}
 \begin{align}
 C_{11}(f_{22}) \bydef 2.55697 \text{, and}~~C_{12}(f_{22}) \bydef 6.85831.
\end{align}
Finally, using \cite{julia_cadiot_blanco_GS}, we choose $r_0 \bydef 6 \times 10^{-6}$ and define
\begin{align}
    \mathcal{Y}_0 \bydef 2.7 \times 10^{-6} \text{,}~\mathcal{Z}_2(r_0) \bydef  16.1625 \text{,}~Z_1 \bydef 0.494 \text{,}~
    \mathcal{Z}_{u,1} \bydef 4.35 \times 10^{-5} \text{, and}~\mathcal{Z}_{u,2} \bydef 4.65 \times 10^{-7}   \end{align}
and prove that these values satisfy Theorem \ref{th: radii polynomial}. 
\par Following this, we compute the $\widehat{\kappa}_j$'s for $j \in \{0,2,3\}$ and define $\widehat{\mathcal{Z}}_1$ and $\widehat{\mathcal{Z}}_2(r)$ as
\begin{align}
    \nonumber \widehat{\kappa}_2 \bydef 1.14419 \text{,}~~ \widehat{\kappa}_3 \bydef 0.185143 \text{,}~~ \widehat{\kappa}_0 \bydef 0.115123\text{,}~~\widehat{\mathcal{Z}}_1 \bydef 0.4952\text{, and}~~\widehat{\mathcal{Z}}_2(r_0) \bydef 21.8517,
\end{align}
we obtain the Theorem \ref{th : radii periodic} is satisfied.
\end{proof}
 \begin{figure}[H]
\centering
 \begin{minipage}{.5\linewidth}
  \centering\epsfig{figure=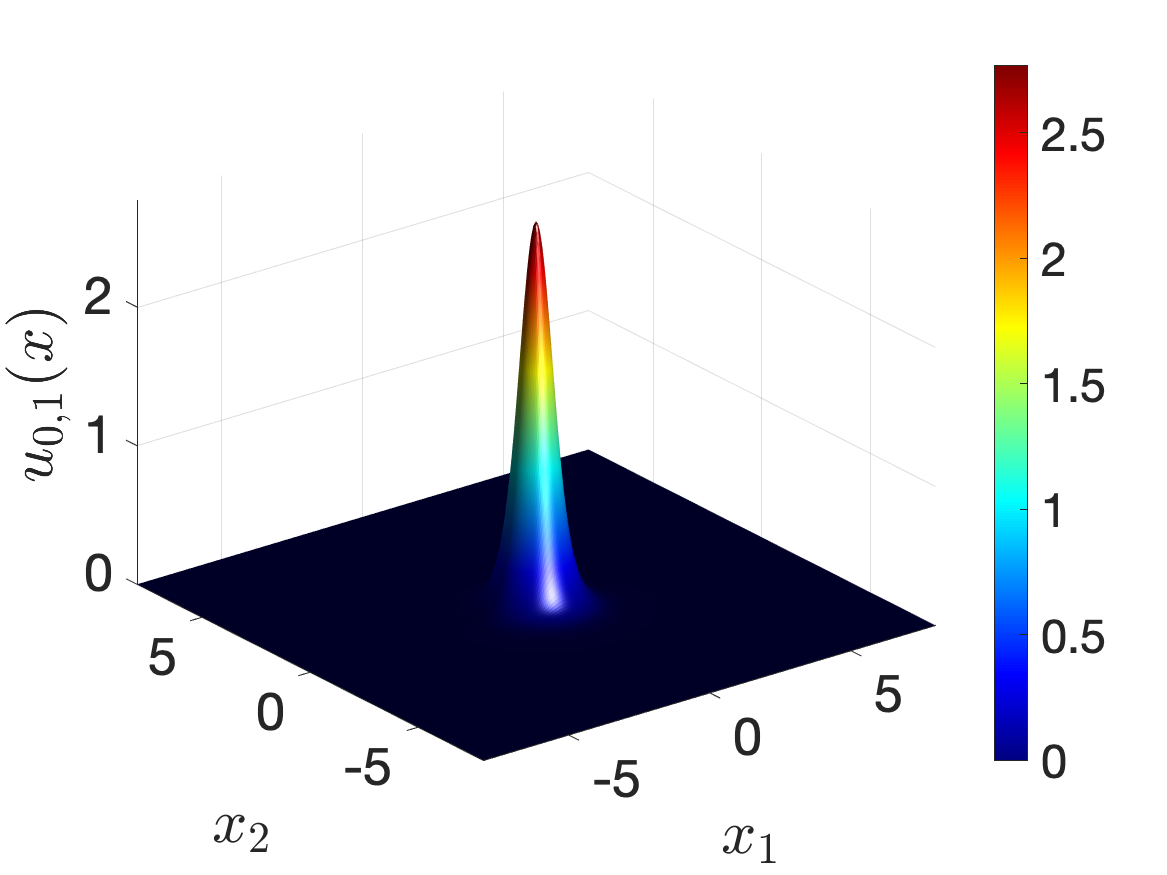,width=\linewidth}
 \end{minipage}%
 \begin{minipage}{.5\linewidth}
  \centering\epsfig{figure=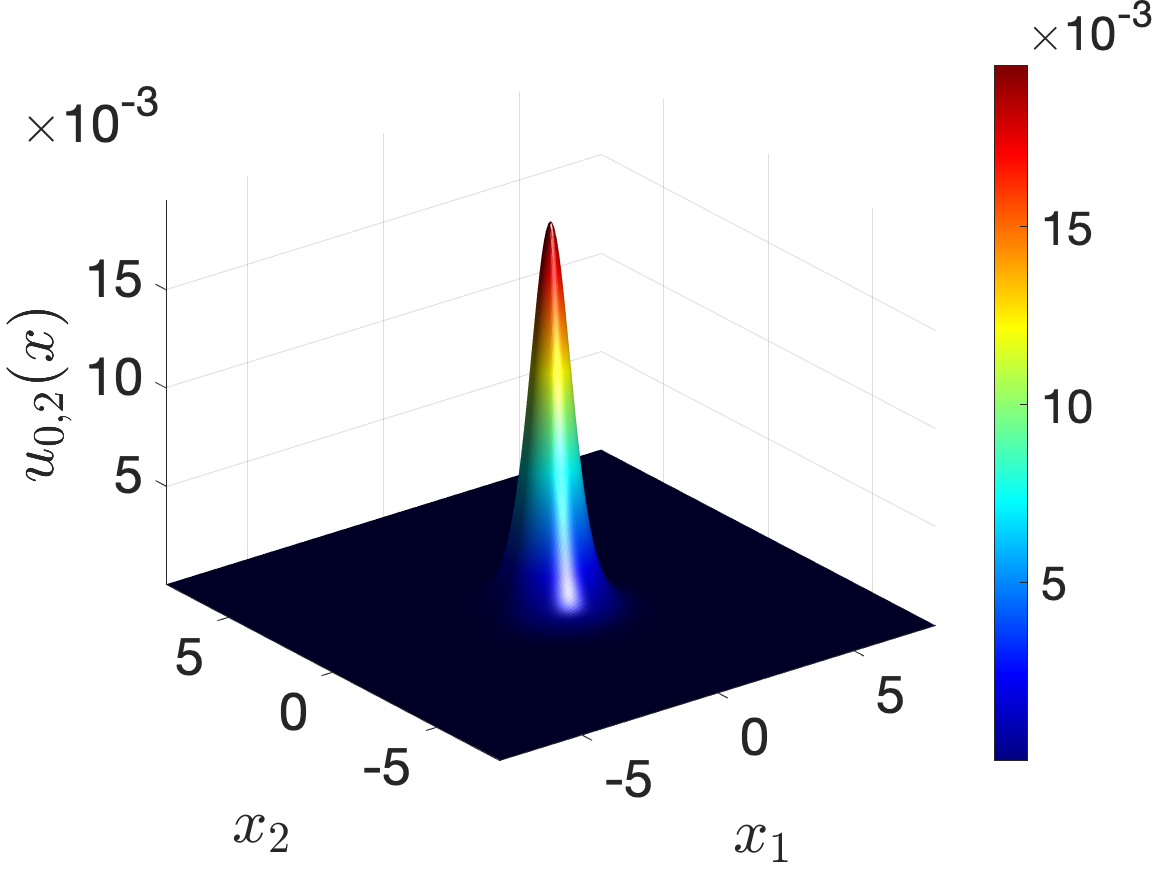,width=\linewidth}
 \end{minipage} 
 \caption{Approximation of a ``spike pattern" in the Gray-Scott system of equations with $\lambda_2 = 10, \lambda_1 = \frac{1}{9}$. (L) is $u_{0,1}$ and (R) is $u_{0,2}$. }
 \label{fig : spike pattern}
 \end{figure}
\subsection{Ring Pattern}
The ring pattern was found by performing numerical continuation from the spike pattern proven in Theorem \ref{th : proof_of_spike_R1}. In particular, the ring appeared after the first fold encountered on the branch.  
\begin{theorem}\label{th : proof_of_ring}
Let $\lambda_2 =3.73, \lambda_1 =0.0567$. Moreover, let $r_0 \bydef 6 \times 10^{-6}$. Then there exists a unique solution $\tilde{\mathbf{u}}$ to \eqref{eq : gray_scott cov} in $\overline{B_{r_0}(\mathbf{u}_{ring})} \subset \mathcal{H}_{D_4}$ and we have that $\|\tilde{\mathbf{u}}-\mathbf{u}_{ring}\|_{\mathcal{H}} \leq r_0$. That is, $\tilde{\mathbf{u}}$ is at most $r_0$ away from the approximation shown in Figure \ref{fig : ring pattern}.
\par In addition, there exists a smooth curve 
\[
\left\{\tilde{\mathbf{u}}(q) : q \in [d,\infty]\right\} \subset C^\infty(\R^2) \times C^\infty(\R^2)
\]
such that $\tilde{\mathbf{u}}(q)$ is a periodic solution to \eqref{eq : gray_scott cov} with period $2q$ in both directions. In particular, $\tilde{\mathbf{u}}(\infty) = \tilde{\mathbf{u}}$ is a localized pattern on $\R^2.$ 
\end{theorem}
\begin{proof}
Choose $N_0 = 80, N = 60, d = 10$. The proof is obtained similarly as that of Theorem \ref{th : proof_of_spike_away}. In particular, we define
\begin{align}
    \kappa_2 \bydef 1.18469 \text{,}~~\kappa_3 \bydef 0.244715 \text{, and}~~\kappa_0 \bydef 0.258464 \end{align}
    \begin{align}
    \mathcal{Y}_0 \bydef  2.3\times 10^{-6} \text{,}~~\mathcal{Z}_2(r_0) \bydef 11235 \text{,}~~Z_1 \bydef 0.317 \text{,}~~
    \mathcal{Z}_{u,1} \bydef  4.71 \times 10^{-5}\text{, and}~~\mathcal{Z}_{u,2} \bydef 1.04 \times 10^{-6}
\end{align}
and prove that these values satisfy Theorem \ref{th: radii polynomial}. Additionally, we define
\begin{align}
    \widehat{\kappa}_2 \bydef 1.43724 \text{,}~~\widehat{\kappa}_3 \bydef 0.651838\text{,}~~\widehat{\kappa}_0 \bydef 0.490695\text{,}~~\widehat{\mathcal{Z}}_1 \bydef 0.428923 \text{, and}~~\widehat{\mathcal{Z}}_2(r_0) \bydef 15127.2
\end{align}
and prove that these values satisfy Theorem \ref{th : radii periodic}.
\end{proof}
 \begin{figure}[H]
\centering
 \begin{minipage}{.5\linewidth}
  \centering\epsfig{figure=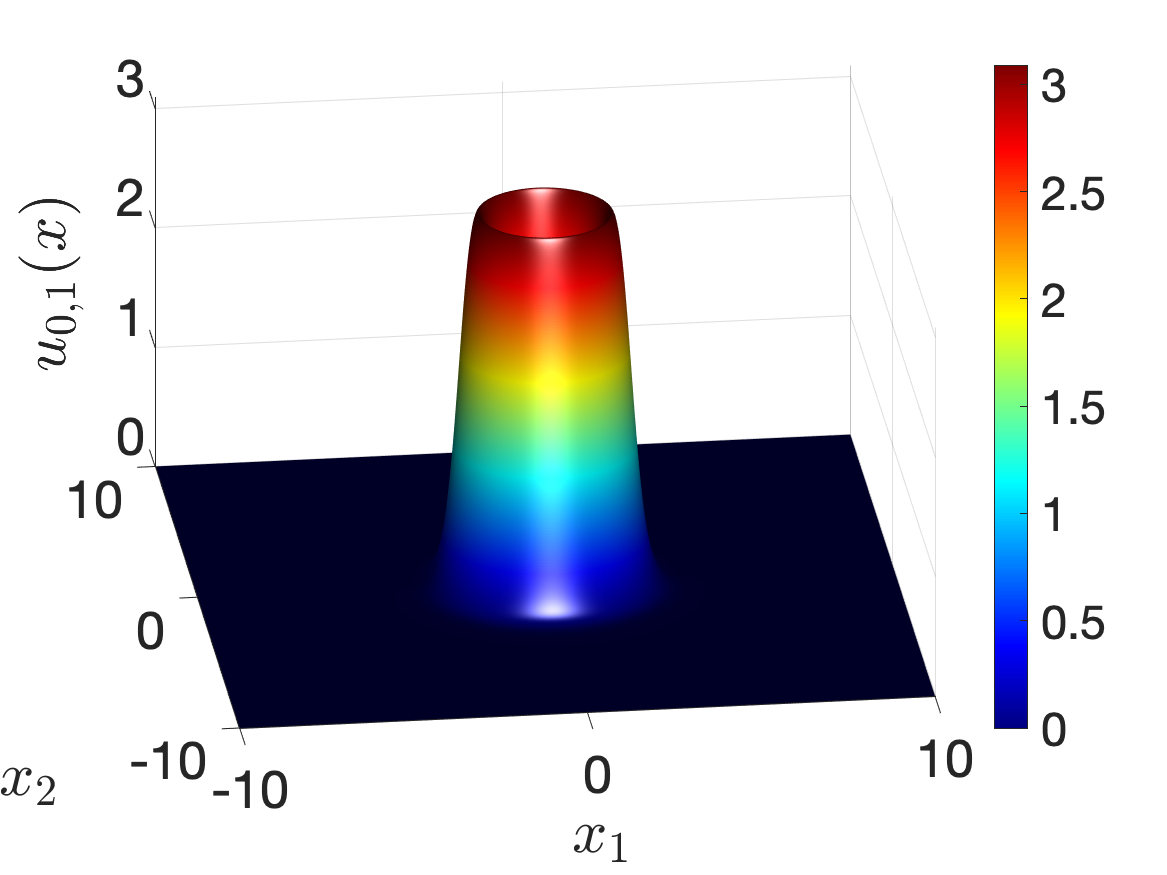,width=\linewidth}
 \end{minipage}%
 \begin{minipage}{.5\linewidth}
  \centering\epsfig{figure=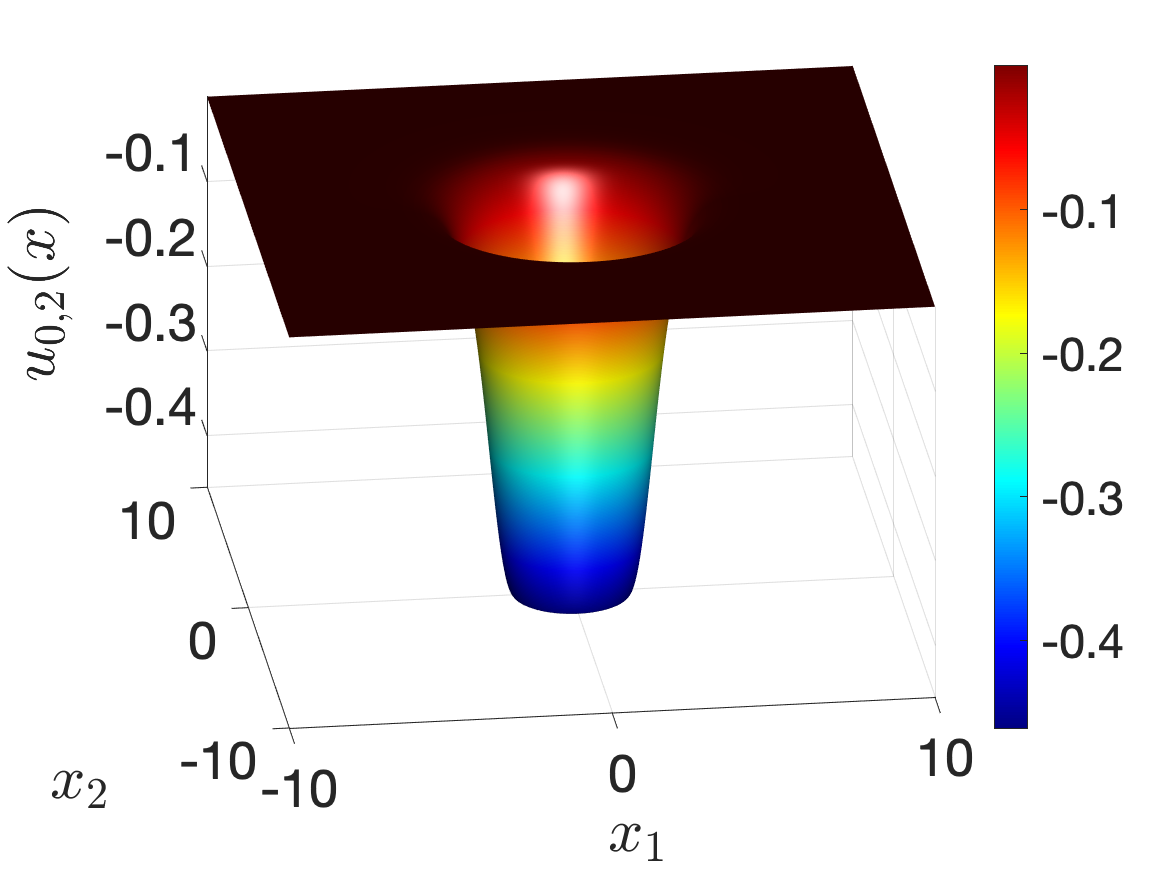,width=\linewidth}
 \end{minipage} \caption{Approximation of a ``ring pattern" in the Gray-Scott system of equations with $\lambda_2 = 3.73, \lambda_1 = 0.0567$. (L) is $u_{0,1}$ and (R) is $u_{0,2}$.}
 \label{fig : ring pattern}
 \end{figure}
\subsection{Leaf Pattern}
By performing more numerical continuation from the ring pattern proven in Theorem \ref{th : proof_of_ring}, we found the ``leaf" pattern. This pattern was encountered after passing through multiple folds and bifurcations.  
\begin{theorem}\label{th : proof_of_leaf}
Let $\lambda_2 =3.74, \lambda_1 =0.0566$. Moreover, let $r_0 \bydef 6 \times 10^{-6}$. Then there exists a unique solution $\tilde{\mathbf{u}}$ to \eqref{eq : gray_scott cov} in $\overline{B_{r_0}(\mathbf{u}_{leaf})} \subset \mathcal{H}_{D_4}$ and we have that $\|\tilde{\mathbf{u}}-\mathbf{u}_{leaf}\|_{\mathcal{H}} \leq r_0$. That is, $\tilde{\mathbf{u}}$ is at most $r_0$ away from the approximation shown in Figure \ref{fig : leaf_pattern}.
\par In addition, there exists a smooth curve 
\[
\left\{\tilde{\mathbf{u}}(q) : q \in [d,\infty]\right\} \subset C^\infty(\R^2) \times C^\infty(\R^2)
\]
such that $\tilde{\mathbf{u}}(q)$ is a periodic solution to \eqref{gray_scott_reduced} with period $2q$ in both directions. In particular, $\tilde{\mathbf{u}}(\infty) = \tilde{\mathbf{u}}$ is a localized pattern on $\R^2.$ 
\end{theorem}  
\begin{proof}
Choose $N_0 = 240, N = 180, d = 22$. The proof is obtained similarly to that of Theorem \ref{th : proof_of_spike_away}.  
In particular, we define
\begin{align}
    \kappa_2 \bydef 1.18574\text{,}~~\kappa_3 \bydef 0.244603 \text{, and}~~\kappa_0 \bydef 0.258144\end{align}
    \begin{align}
    \mathcal{Y}_0 \bydef  1.1\times 10^{-6} \text{,}~~\mathcal{Z}_2(r_0) \bydef 51812 \text{,}~~Z_1 \bydef 0.515\text{,}
    ~~
    \mathcal{Z}_{u,1} \bydef 6\times 10^{-6} \text{, and}~~\mathcal{Z}_{u,2} \bydef 1.7 \times 10^{-6}
\end{align}
and prove that these values satisfy Theorem \ref{th: radii polynomial}. Additionally, we define
\begin{align}
    \widehat{\kappa}_2 \bydef 1.30637 \text{,}~~\widehat{\kappa}_3 \bydef 0.445831\text{,}~~\widehat{\kappa}_0 \bydef 0.397017\text{,}~~\widehat{\mathcal{Z}}_1 \bydef 0.5798\text{, and}~~\widehat{\mathcal{Z}}_2(r_0) \bydef 61246.4 
\end{align}
and prove that these values satisfy Theorem \ref{th : radii periodic}.
\end{proof}
 \begin{figure}[H]
\centering
 \begin{minipage}{.5\linewidth}
  \centering\epsfig{figure=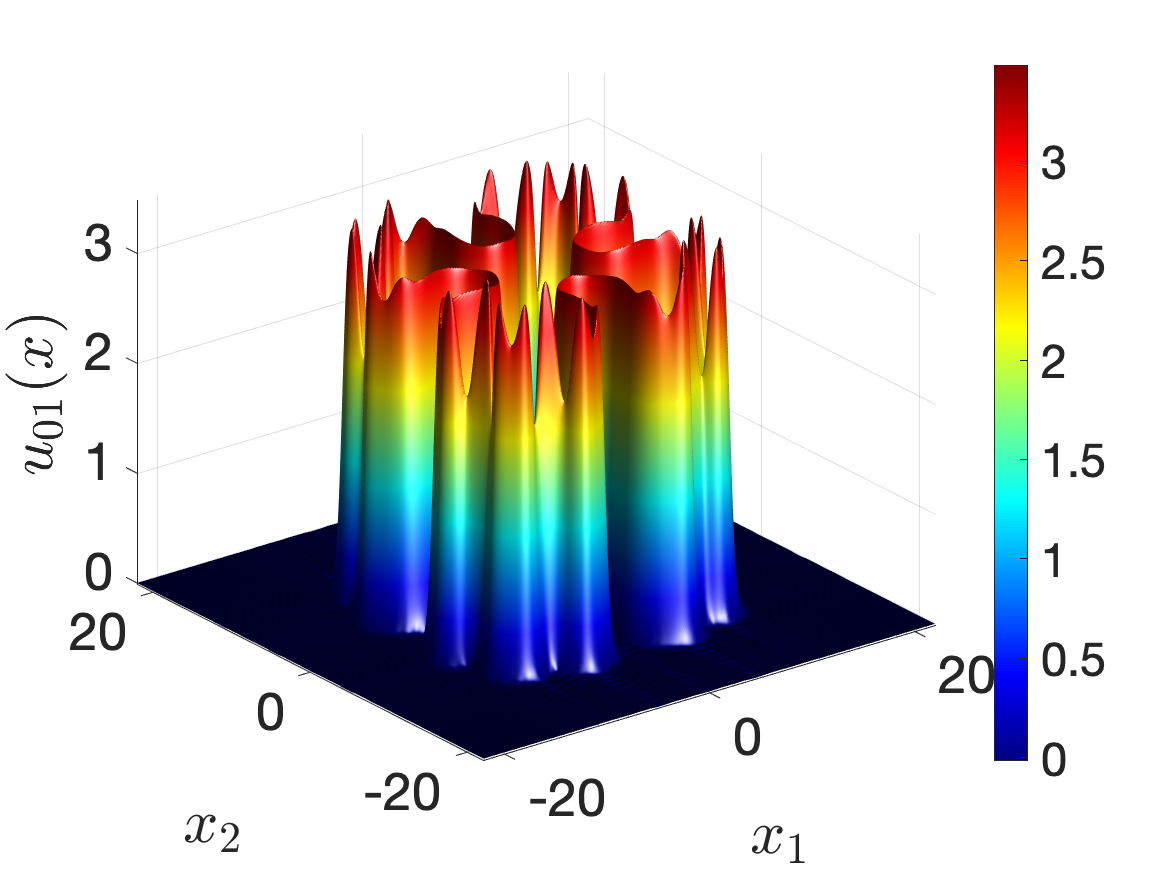,width=\linewidth}
 \end{minipage}%
 \begin{minipage}{.5\linewidth}
  \centering\epsfig{figure=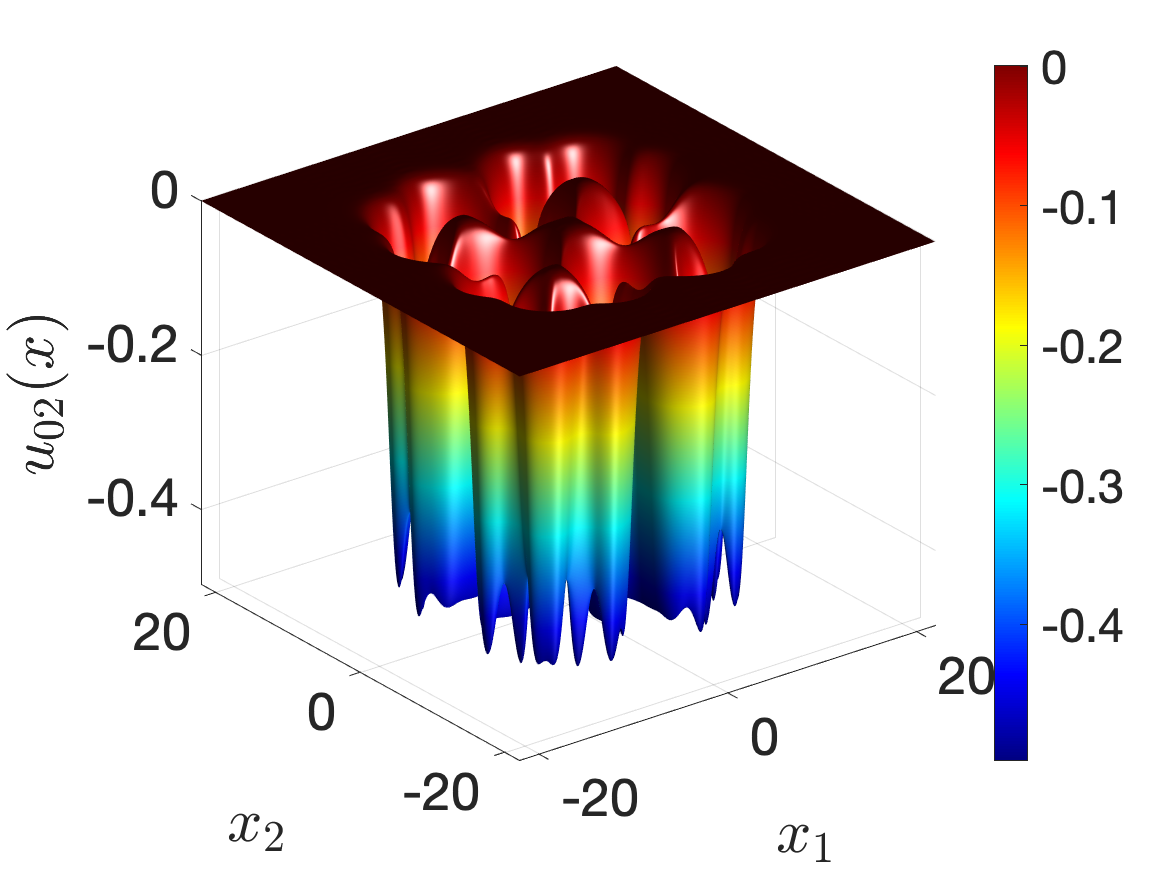,width=\linewidth}
 \end{minipage} \caption{ Approximation of the  ``leaf pattern" in the Gray-Scott system of equations with $\lambda_2 = 3.74, \lambda_1 = 0.0566$. (L) is $u_{0,1}$ and (R) is $u_{0,2}$.}
 \label{fig : leaf_pattern}
 \end{figure}
 To the best of our knowledge, this is the first non-radial localized pattern proven in the Gray-Scott model. Note that the value of $N_0$ is much larger than the previous patterns at $240$. This was necessary due to the slow decay of the Fourier coefficients. The use of $D_4$-symmetric Fourier series made the proof of the leaf possible though despite limited memory. 
\section{Conclusion and Future Directions}\label{sec : conclusion}
In this paper, we extended the theory of \cite{unbounded_domain_cadiot} to derive a general computer-assisted methodology for proving the existence of stationary localized patterns in autonomous semi-linear systems of PDEs. We demonstrated the efficiency of the approach on the 2D Gray-Scott system of equations by constructively proving four localized patterns with high accuracy. Our results} include, to the best of our knowledge, the first non-radial localized pattern proven in Gray-Scott. This approach could be directly applied to any other reaction diffusion system provided Assumptions \ref{ass:A(1)} and \ref{ass : LinvG in L1} are satisfied. In the event that $\mathbb{G}$ is not of the form \eqref{def: G and j}, further work would be required. More specifically, if $\mathbb{G}$ is non-polynomial, then $\mathbb{G}(\mathbf{u}_0)$ has infinitely many non-vanishing Fourier coefficients. A technique was provided by the authors of \cite{nonpoly_suspension} to handle such nonlinearities rigorously. It would be interesting to investigate the applicability of this method to treat problems on unbounded domains. Combining these techniques would allow one to treat a larger class of reaction diffusion models, such as the Gierer-Meinhardt model.
 \par It would also be of interest to prove the existence a branch of solutions starting from $\lambda_2 \lambda_1 = 1$. This would allow to have a 2D-equivalent of the results proven in \cite{jay_jp_gs} on unbounded domains. In the aforementioned paper, the existence of spike solutions in 1D Gray-Scott away from $\lambda_2 \lambda_1 = 1$ was verified by performing a rigorous continuation method. It would be of interest to obtain a similar result in the two-dimensional case. Furthermore, as mentioned in Section \ref{sec : numerical construction}, we began by obtaining an approximate solution of a radially symmetric "spike" pattern in \eqref{gray_scott_reduced}. This solution is depicted in Figure \ref{fig : 1D case}. From here, we performed numerical continuation to obtain the other approximate solutions used in Figures \ref{fig : spike pattern}, \ref{fig : ring pattern}, and \ref{fig : leaf_pattern}. It is natural to ask if this numerical approach can be made rigorous. One way to do so is to use the approach described in \cite{polynomial_chaos}. In this paper, the author describes a rigorous method for proving branches of solutions using Chebyshev series.  The author of \cite{cadiot_witham} adapted the method to unbounded domains on the nonlocal Witham equation. This is demonstrated in Section 4.6 of the aforementioned paper. This is currently under investigation. Additionally, in \cite{wei_breakup}, numerical solutions to Gray-Scott were shown on a bounded 2D domain. It would be of interest to verify if these patterns are indeed localized patterns and then attempt to prove their existence with the present methodology.
 \par Another point of interest is stability of the obtained localized patterns.  This question has been investigated in Section 5.3 of \cite{unbounded_domain_cadiot}, as it provides the necessary tools for proving isolated eigenvalues. In particular, the aforementioned approach provides a direct way to prove spectral instability. To prove the stability of a pattern, one would have to control regions of non-existence of eigenvalues, as presented in \cite{cadiot_witham} or \cite{unbounded_domain_cadiot}.   By using a similar approach to the two aforementioned papers, we could attempt to prove the (in)stability of our patterns in Gray-Scott. This is also under investigation.
\par Finally, for the purpose of Gray-Scott, our numerical investigation was completely performed using our $D_4$-Fourier code available at \cite{dominic_D_4_julia}. Now that we have proven one non-radial localized pattern in Gray-Scott, it is possible more could exist, possibly with symmetries other than $D_4$. In particular, the method presented in this manuscript could easily be modified to match diverse symmetries. We believe the proof of the leaf pattern justifies further numerical investigation into the model for possible existence of more non-radial localized patterns. With the method developed in this paper, one could prove these non-radial solutions, should they exist. Computing these solutions is a highly non-trivial process, but their discovery would provide more insight as to what kind of phenomena is possible in the Gray-Scott system of equations. 
\appendix
\renewcommand{\theequation}{A.\arabic{equation}}
\setcounter{equation}{0}
\section{Computation of \texorpdfstring{$\mathcal{Z}_2$}{Z2}}\label{apen : Z2}
We begin by proving Lemma \ref{lem : banach algebra}.
\begin{lemma}
Let $\kappa_2$ be defined in \eqref{def : definition kappa2 kappa3} and define $\kappa_0 >0$ as 
\begin{align}
    \kappa_0 \bydef \min\left\{ \max\left\{\left[\left(\lambda_1 \kappa_2 + \frac{1}{2\sqrt{\pi \lambda_2}}\right)^2 + \frac{1}{4\pi \lambda_2}\right]^{\frac{1}{2}}, \frac{\sqrt{2}}{2\sqrt{\pi \lambda_2}}\right\},  \frac{\kappa_2}{\lambda_2} \left((1-\lambda_2 \lambda_1)^2 + 1\right)^{\frac{1}{2}}\right\}.
\end{align}
Then, 
\begin{align}
\|u_1v_2\|_2 \leq \kappa_0 \|\mathbf{u}\|_{\mathcal{H}} \|\mathbf{v}\|_{\mathcal{H}}
\end{align}
for all $\mathbf{u}= (u_1,u_2), \mathbf{v}=(v_1,v_2) \in \mathcal{H}$.
\end{lemma}
\begin{proof}
First, notice that 
\begin{align*}
    \|u_1v_1\|_2 = \left\|\begin{bmatrix}
        1 & 0 \\
        0 & 0 
    \end{bmatrix} \mathbf{u} \begin{bmatrix}
        1 & 0 \\
        0 & 0 
    \end{bmatrix} \mathbf{v}\right\|_2 \leq \kappa_2 \|\mathbf{u}\|_{\mathcal{H}} \|\mathbf{v}\|_{\mathcal{H}}
\end{align*}
using Corollary \ref{cor : banach algebra1}.
Then, observe that we have
\begin{align*}
    \|u_1v_2\|_2 = \left\|\begin{bmatrix}
        1 & 0 \\
        0 & 0 
    \end{bmatrix} \mathbf{u} \begin{bmatrix}
        0 & 1 \\
        0 & 0 
    \end{bmatrix} \mathbf{v}\right\|_2.
\end{align*}
Now,  
\begin{align}
    \begin{bmatrix}
        0 & 1 \\
        0 & 0 
    \end{bmatrix} l^{-1}(\xi) = \begin{bmatrix}
        -\frac{l_{21}(\xi)}{l_{11}(\xi)l_{22}(\xi)} & \frac{1}{l_{22}(\xi)} \\
        0 & 0
    \end{bmatrix}
    = \begin{bmatrix}
    \frac{\lambda_1}{l_{11}(\xi)} - \frac{1}{l_{22}(\xi)} & \frac{1}{l_{22}(\xi)} \\
    0 & 0
    \end{bmatrix}
\end{align}
for all $\xi \in \R^2$.
Consequently, using \eqref{eq : equality of the lik}, we have 
\begin{align*}
    \left\| \begin{bmatrix}
        0 & 1 \\
        0 & 0 
    \end{bmatrix} l^{-1}\right\|_{\mathcal{M}_1} = \sup_{\xi \in \R^2} \sup\limits_{\substack{x \in \R^2\\ |x|_2=1}} \left|\frac{\lambda_1 x_1}{l_{11}(\xi)} - \frac{x_1}{l_{22}(\xi)} + \frac{x_2}{l_{22}(\xi)}\right|
    &= \sup_{\xi \in \R^2} \left(\left(\frac{\lambda_1 }{l_{11}(\xi)} - \frac{1}{l_{22}(\xi)}\right)^2 +  \frac{1}{l_{22}(\xi)^2} \right)^{\frac{1}{2}}\\
    &= \sup_{\xi \in \R^2} \left(\left(\frac{l_{21}(\xi)}{l_{22}(\xi)l_{11}(\xi)}\right)^2 +  \frac{1}{l_{22}(\xi)^2} \right)^{\frac{1}{2}}\\
    &= \left(\left(\frac{1-\lambda_2\lambda_1}{\lambda_2}\right)^2 +  \frac{1}{\lambda_2^2} \right)^{\frac{1}{2}}
\end{align*}
On the other hand,
\begin{align}
    \left\| \begin{bmatrix}
        0 & 1 \\
        0 & 0 
    \end{bmatrix} l^{-1}\right\|_{\mathcal{M}_2} = \left( \left\|\frac{\lambda_1}{l_{11}}- \frac{1}{l_{22}}\right\|_2^2 + \left\|\frac{1}{l_{22}}\right\|_2^2\right)^{\frac{1}{2}}
    &\leq \left[ \left(\lambda_1 \left\|\frac{1}{l_{11}}\right\|_{2} + \left\|\frac{1}{l_{22}}\right\|_{2}\right)^2 + \left\|\frac{1}{l_{22}}\right\|_{2}^2\right]^{\frac{1}{2}}.\label{kappa0_with_norms}
\end{align}
We conclude the proof combining Lemma \ref{Banach algebra} and the computations performed in Corollary \ref{cor : banach algebra1}.
\end{proof}
We now present the proof of Lemma \ref{lem : Bound Z_2}. Let us first recall the lemma, and then begin the proof.
\begin{lemma}
Let $\kappa_0$ and  $(\kappa_2, \kappa_3)$ be the bounds satisfying \eqref{def : kappa0} and \eqref{def : definition kappa2 kappa3} respectively. Moreover, define $q \in L^\infty(\R^2)\cap L^2_{D_4}(\R^2)$ and its associated Fourier coefficients $Q$ as  \begin{equation}
        q \bydef u_{0,2} + \mathbb{1}_{\om} -3\lambda_1 u_{0,1}~~ \text{ and }~~
        Q \bydef \gamma(q).
\end{equation} 

Then, let
$\mathcal{Z}_2 : (0, \infty) \to [0, \infty)$ be defined  as 
\begin{align}
    \mathcal{Z}_2(r) &\bydef  2\left(\kappa_2^2 + 4 \kappa_0^2\right)^{\frac{1}{2}}\sqrt{\varphi(\mathcal{Z}_{2,1},\mathcal{Z}_{2,2},\mathcal{Z}_{2,2},\mathcal{Z}_{2,3})}   + 3\kappa_3 \max\{1,\| B_{11}^N\|_{2}\}r, 
\end{align}
for all $r \geq 0$,
where
\begin{align}
    &\mathcal{Z}_{2,1} \bydef  \| B_{11}^N(\mathbb{Q}^2 + \mathbb{U}_{0,1}^2)(B_{11}^N)^{*}\|_{2} \\
    &\mathcal{Z}_{2,2} \bydef \sqrt{\| B_{11}^N (\mathbb{Q}^2 + \mathbb{U}_{0,1}^2)\pi_N(\mathbb{Q}^2 + \mathbb{U}_{0,1}^2) (B_{11}^N)^{*} \|_{2}} \\
    &\mathcal{Z}_{2,3} \bydef \|Q^2 + U_{0,1}^2\|_{1},
\end{align}
and where $\varphi$ is defined in Lemma \ref{lem : full_matrix_estimate}. Then, $\|\mathbb{A}\left( D\mathbb{F}(\mathbf{u}_0 + \mathbf{h}) -D\mathbb{F}(\mathbf{u}_0)\right)\|_{\mathcal{H}} \leq \mathcal{Z}_2(r)r$ for all $r>0$ and all $\mathbf{h} \in \overline{B_r(0)} \subset \mathcal{H}$.
\end{lemma}
\begin{proof}
Let $\mathbf{h} \in B_r(0), \mathbf{z} \in B_1(0) \subset \mathcal{H}$. Then, let $g_2 : \left(\mathcal{H}\right)^2 \to \mathcal{H}$ and $g_3 : \left(\mathcal{H}\right)^3 \to \mathcal{H}$ be defined as
\begin{align*}
        g_2(\mathbf{u}, \mathbf{v}) &\bydef  \mathbb{G}_{2,1}\mathbf{u}  \mathbb{G}_{2,2}\mathbf{v}\\
        g_3(\mathbf{u}, \mathbf{v}, \mathbf{w}) &\bydef \mathbb{G}_{3,1}\mathbf{u}  \mathbb{G}_{3,2}\mathbf{v} \mathbb{G}_{3,3}\mathbf{w}
    \end{align*}
    for all $\mathbf{u}, \mathbf{v}, \mathbf{w} \in \mathcal{H}.$
In particular, notice that 
\begin{align*}
    g_2(\mathbf{u}, \mathbf{v})=  g_2(\mathbf{v}, \mathbf{u})\\
    g_3(\mathbf{u}, \mathbf{v}, \mathbf{w}) =  g_3(\mathbf{u}, \mathbf{w}, \mathbf{v})
\end{align*}
for all $\mathbf{u}, \mathbf{v}, \mathbf{w} \in \mathcal{H}$.
Therefore, given $\mathbf{v} \in \mathcal{H}_{D_4}$, we have
\begin{align}\label{eq : step 1 in proof Z2}
    D\mathbb{G}(\mathbf{v}) &= g_2(\mathbf{v}, \cdot) + g_2(\mathbf{v}, \cdot) + g_3(\mathbf{v}, \mathbf{v}, \cdot) + g_3(\mathbf{v}, \cdot, \mathbf{v}) + g_3(\cdot , \mathbf{v}, \mathbf{v})\\
        & = 2g_2(\mathbf{v}, \cdot) + 2g_3(\mathbf{v}, \mathbf{v}, \cdot) +   g_3(\cdot , \mathbf{v}, \mathbf{v}).\label{DG_computation_Z2}
\end{align}
Using \eqref{eq : step 1 in proof Z2}, observe that
\begin{align}
\nonumber &\left(D\mathbb{G}(\mathbf{u}_0 + \mathbf{h}) - D\mathbb{G}(\mathbf{u}_0)\right) \mathbf{z} \\ 
&=  2g_2(\mathbf{h},\mathbf{z}) +  2g_3(\mathbf{h},\mathbf{u}_0, \mathbf{z}) + 2g_3(\mathbf{h},\mathbf{h}, \mathbf{z})  + 2 g_3(\mathbf{z},\mathbf{u}_0, \mathbf{h}) + 2 g_3(\mathbf{z},\mathbf{h},\mathbf{u}_0) +  g_3(\mathbf{z},\mathbf{h},\mathbf{h}).
    \end{align}
Simplifying, we obtain
\begin{align}\label{eq : step 2 proof Z2}
    \left(D\mathbb{G}(\mathbf{u}_0 + \mathbf{h}) - D\mathbb{G}(\mathbf{u}_0)\right) \mathbf{z} =  D^2\mathbb{G}(\mathbf{u}_0)(\mathbf{h},\mathbf{z}) +  g_3(\mathbf{z},\mathbf{h},\mathbf{h}) + 2g_3(\mathbf{h},\mathbf{h}, \mathbf{z}).
\end{align}
Now, by definition of $\mathbb{B}$ in \eqref{def : definition of B by blocks} and $\mathbb{G}$ in \eqref{definition_of_G}, we have 
\begin{align}
    \mathbb{B}D\mathbb{G}(\mathbf{v}) = \begin{bmatrix}
        \mathbb{B}_{11} & 0 \\
        0 &
        0
    \end{bmatrix}D\mathbb{G}(\mathbf{v})\label{B_times_DG}
\end{align}
for all $\mathbf{v} \in \mathcal{H}_{D_4}$. Consequently, using Lemma \ref{cor : banach algebra1}, we get
\begin{align}
    \|\mathbb{B}\left(g_3(\mathbf{z},\mathbf{h},\mathbf{h}) + 2g_3(\mathbf{h},\mathbf{h}, \mathbf{z})\right)\|_2 &\leq \|\mathbb{B}_{11}\|_{2}\left( \|g_3(\mathbf{z},\mathbf{h},\mathbf{h})\|_2 + 2\|g_3(\mathbf{h},\mathbf{h}, \mathbf{z})\|_2\right)\\
    &\leq 3\kappa_3 \|\mathbb{B}_{11}\|_{2} \|\mathbf{z}\|_{\mathcal{H}} \|\mathbf{h}\|_{\mathcal{H}}^2 \leq 3\kappa_3 \|\mathbb{B}_{11}\|_{2} r^2.\label{eq : step 3 proof Z2}
\end{align}
Therefore, combining \eqref{eq : step 2 proof Z2} and \eqref{eq : step 3 proof Z2}, we obtain
\begin{align}
    \|\mathbb{B}(D\mathbb{G}(\mathbf{u}_0 + \mathbf{h}) - D\mathbb{G}(\mathbf{u}_0))\mathbf{z}\|_{2} &\leq \|\mathbb{B}D^2\mathbb{G}(\mathbf{u}_0)(\mathbf{h},\mathbf{z})\|_{2} +  3\kappa_3 \|\mathbb{B}_{11}\|_{2} r^2.\label{Z_2_higher_order}
\end{align}
Now, following the steps of the proof of Lemma 3.4 from \cite{unbounded_domain_cadiot}, we  similarly obtain
\begin{align}
    \|\mathbb{B}_{11}\|_{2} = \max\{1,\|B_{11}^N\|_{2}\}.\label{eq : Parseval_on_operator_B11}
\end{align}
In particular, using \eqref{eq : Parseval_on_operator_B11}, the equation \eqref{Z_2_higher_order} becomes
\begin{align}
     \|\mathbb{B}(D\mathbb{G}(\mathbf{u}_0 + \mathbf{h}) - D\mathbb{G}(\mathbf{u}_0))\mathbf{z}\|_{2} \leq  \|\mathbb{B}D^2\mathbb{G}(\mathbf{u}_0)(\mathbf{h},\mathbf{z})\|_{2} +  3\kappa_3 \max\{1,\|B_{11}^N\|_2\} r^2.
\end{align}
Now, let us examine the term $\|\mathbb{B}D^2\mathbb{G}(\mathbf{u}_0)(\mathbf{h},\mathbf{z})\|_{2}$. To begin, observe that
\begin{align*}
    D^2\mathbb{G}(\mathbf{u}_0)(\mathbf{z},\mathbf{h}) = 2\mathbb{q}\mathbb{h}_1z_1 + 2 \mathbb{u}_{0,1}\mathbb{h}_1z_2 +  2 \mathbb{u}_{0,1}\mathbb{h}_2z_1 &= z_1 \left(2\mathbb{q}\mathbb{h}_1 + 2 \mathbb{u}_{0,1}\mathbb{h}_2\right) + 2 \mathbb{h}_1 \mathbb{u}_{0,1}z_2\\
        &= \begin{bmatrix}
            2\mathbb{q}\mathbb{h}_1 + 2 \mathbb{u}_{0,1}\mathbb{h}_2 & 2  \mathbb{u}_{0,1}\mathbb{h}_1\\
            0 & 0
        \end{bmatrix} \begin{bmatrix}
            z_1\\
            z_2
        \end{bmatrix}\\
        & =  \begin{bmatrix}
            2\mathbb{q}  & 2  \mathbb{u}_{0,1}
        \end{bmatrix}\begin{bmatrix}
            \mathbb{h}_1 & 0 \\
            \mathbb{h}_2 & \mathbb{h}_1
        \end{bmatrix} \begin{bmatrix}
            z_1\\
            z_2
    \end{bmatrix}
\end{align*}
where $\mathbb{q}$ is the multiplication operator (cf. \eqref{def : multiplication operator}) associated to $q$ defined in \eqref{def : q}.
Using \eqref{B_in_full_equation}, we get
\begin{align}
    \nonumber \left\|\mathbb{B}D^2\mathbb{G}(\mathbf{u}_0)(\mathbf{z},\mathbf{h})\right\|_{2} &= 2\left\|\begin{bmatrix}
\mathbb{B}_{11}\mathbb{q} &   \mathbb{B}_{11}\mathbb{u}_{0,1}
\end{bmatrix}\begin{bmatrix}
            \mathbb{h}_1 & 0 \\
            \mathbb{h}_2 & \mathbb{h}_1
        \end{bmatrix} \begin{bmatrix}
            z_1\\
            z_2
        \end{bmatrix}\right\|_{2}\\ 
    &\leq 2\left\|\begin{bmatrix}
    \mathbb{B}_{11}\mathbb{q} &   \mathbb{B}_{11}\mathbb{u}_{0,1}
    \end{bmatrix}\right\|_{2} \left\|\begin{bmatrix}
            \mathbb{h}_1 & 0 \\
            \mathbb{h}_2 & \mathbb{h}_1
        \end{bmatrix} \begin{bmatrix}
            z_1\\
            z_2
        \end{bmatrix}\right\|_{2}.\label{step_in_Z2}
\end{align}
Let us now treat the term $\left\|\begin{bmatrix} \mathbb{h}_1 & 0 \\
    \mathbb{h}_2 & \mathbb{h}_1
    \end{bmatrix} \begin{bmatrix}
        z_1\\
        z_2
        \end{bmatrix}\right\|_{2}$. We have
\begin{align}
\nonumber \left\|\begin{bmatrix}
        \mathbb{h}_1 & 0 \\
        \mathbb{h}_2 & \mathbb{h}_1
    \end{bmatrix}\begin{bmatrix}
        z_1 \\
        z_2
    \end{bmatrix}\right\|_{2} &= \left\| \begin{bmatrix}
        \mathbb{h}_1 z_1 \\
        \mathbb{h}_2 z_1 + \mathbb{h}_1 z_2
    \end{bmatrix}\right\|_{2} \\ \nonumber
    &= \left(\|\mathbb{h}_1 z_1\|_{2}^2 + \|\mathbb{h}_2z_1 + \mathbb{h}_1z_2\|_{2}^2\right)^{\frac{1}{2}} \\ \nonumber
    &\leq \left(\|\mathbb{h}_1 z_1\|_{2}^2 + \left(\|\mathbb{h}_2z_1\|_{2} + \|\mathbb{h}_1z_2\|_{2}\right)^2\right)^{\frac{1}{2}} \\ \nonumber
    &\leq \left(\kappa^2_2 \|\mathbf{h}\|_{\mathcal{H}} \|\mathbf{z}\|_{\mathcal{H}} + (\kappa_0 \|\mathbf{h}\|_{\mathcal{H}}\|\mathbf{z}\|_{\mathcal{H}} + \kappa_0 \|\mathbf{h}\|_{\mathcal{H}}\|\mathbf{z}\|_{\mathcal{H}})^2\right)^{\frac{1}{2}} \\ \nonumber
    &= \left(\kappa_2^2 + 4 \kappa_0^2\right)^{\frac{1}{2}}r
\end{align}
where we used Corollary \ref{cor : banach algebra1} and Lemma \ref{lem : banach algebra}.  Then, we focus on
$\|\mathbb{B}_{11}\begin{bmatrix}
    \mathbb{q} & \mathbb{u}_{0,1}
\end{bmatrix}\|_{2}$. Using the properties of the adjoint, we obtain that
\begin{align}
\left\|\mathbb{B}_{11}\begin{bmatrix}
    \mathbb{q} & \mathbb{u}_{0,1}
\end{bmatrix}\right\|_{2}^2 &= \left\|\mathbb{B}_{11}(\mathbb{q}^2 + \mathbb{u}_{0,1}^2) \mathbb{B}_{11}^{*}\right\|_{2}\label{adjoint_to_get_single_operator_Z2}.
\end{align}
Recalling that $\mathbb{B}_{11} = \Gamma^\dagger(\pi_N + B_{11}^N)$  by definition, we get
\begin{align}
    \nonumber \left\|\mathbb{B}_{11}(\mathbb{q}^2 + \mathbb{u}_{0,1}^2)\mathbb{B}_{11}^{*}\right\|_{2} &= \left\|\Gamma^\dagger(\pi_N + B_{11}^N)(\mathbb{q}^2 + \mathbb{u}_{0,1}^2)\Gamma^\dagger(\pi_N + B_{11}^N)^{*}\right\|_{2}. \end{align}
Since $\Gamma^\dagger(K_1)\Gamma^\dagger(K_2) = \Gamma^\dagger(K_1 K_2)$ for all $K_1, K_2 \in \mathcal{B}(\ell^2(J_{\mathrm{red}}(D_4)))$, we have
    \begin{align}
    \left\|\Gamma^\dagger(\pi_N + B_{11}^N)(\mathbb{q}^2 + \mathbb{u}_{0,1}^2)\Gamma^\dagger(\pi_N + B_{11}^N)^{*}\right\|_{2} &= \left\|\Gamma^\dagger\left((\pi_N + B_{11}^N) (\mathbb{Q}^2 + \mathbb{U}_{0,1}^2) (\pi_N + B_{11}^N)^{*}\right)\right\|_{2}\\
    & = \left\|(\pi_N + B_{11}^N) (\mathbb{Q}^2 + \mathbb{U}_{0,1}^2) (\pi_N + B_{11}^N)^{*}\right\|_{2}\label{step_before_estimate_Z2}, 
    \end{align}
    where we used Lemma \ref{lem : gamma and Gamma properties} on the last step and where $\mathbb{Q}$ and $ \mathbb{U}_{0,1}$ are the discrete convolution operators (cf. \eqref{def : discrete conv operator}) associated to $Q$ and $U_{0,1}$ respectively.
Now, we decompose  \eqref{step_before_estimate_Z2} as follows
{\begin{align}
    \left\|(\pi_N + B_{11}^N) (\mathbb{Q}^2 + \mathbb{U}_{0,1}^2) (\pi_N + B_{11}^N)^{*}\right\|_{2}&= \left\| \begin{bmatrix}
        B_{11}^N(\mathbb{Q}^2 + \mathbb{U}_{0,1}^2) (B_{11}^N)^{*} & B_{11}^N (\mathbb{Q}^2 + \mathbb{U}_{0,1}^2) \pi_N \\
        \pi_N (\mathbb{Q}^2 + \mathbb{U}_{0,1}^2) (B_{11}^N)^{*} & \pi_N (\mathbb{Q}^2 + \mathbb{U}_{0,1}^2)\pi_N
    \end{bmatrix}\right\|_{2}.\label{step_before_varphi_Z2} 
    \end{align}}
It remains to examine each of the blocks of the right-hand side of  \eqref{step_before_varphi_Z2} and to apply Lemma \ref{lem : full_matrix_estimate} to bound the full operator norm. First, we consider the upper right and lower left blocks of \eqref{step_before_varphi_Z2}. Notice these blocks are adjoints of each other, hence their spectral norm is identical. Equivalently,  it suffices to study $\|B_{11}^N (\mathbb{Q}^2 + \mathbb{U}_{0,1}^2) \pi_N\|_{2}$.
In particular, we obtain 
\begin{equation}\label{second_and_third_block_Z2_estimated}
    \|B_{11}^N(\mathbb{Q}^2 + \mathbb{U}_{0,1}^2)\pi_N\|_{2} = \sqrt{\| B_{11}^N(\mathbb{Q}^2 + \mathbb{U}_{0,1}^2) \pi_N (\mathbb{Q}^2 + \mathbb{U}_{0,1}^2) (B_{11}^N)^{*} \|_{2}} \bydef \mathcal{Z}_{2,2}
\end{equation} 
Then, by Young's inequality \eqref{young_inequality}, we have 
\begin{align}
    \|\pi_N (\mathbb{Q}^2 + \mathbb{U}_{0,1}^2) \pi_N \|_{2} \leq \|\mathbb{Q}^2 + \mathbb{U}_{0,1}^2\|_{2}  &\leq \|Q^2 + V_0^2\|_{1} \bydef \mathcal{Z}_{2,3}\label{fourth_block_Z2_estimated}
\end{align}
All together, combining \eqref{second_and_third_block_Z2_estimated}, \eqref{fourth_block_Z2_estimated} and applying Lemma \ref{lem : full_matrix_estimate}, we obtain
\begin{align}
 &\|\mathbb{B}_{11} \begin{bmatrix}\mathbb{q} & \mathbb{u}_{0,1} \end{bmatrix}\|_{2} \leq \varphi(\mathcal{Z}_{2,1},\mathcal{Z}_{2,2},\mathcal{Z}_{2,2},\mathcal{Z}_{2,3}) \label{B_term_block_one}\end{align}
meaning that
\begin{align}
    \|\mathbb{B}D^2\mathbb{G}(\mathbf{u}_0)(\mathbf{z},\mathbf{h})\|_{2} \leq 2\left(\kappa_2^2 + 4 \kappa_0^2\right)^{\frac{1}{2}} r\sqrt{\varphi(\mathcal{Z}_{2,1},\mathcal{Z}_{2,2},\mathcal{Z}_{2,2},\mathcal{Z}_{2,3})} .\label{bbM_1_result}
\end{align} 
Combining \eqref{Z_2_higher_order}   with \eqref{bbM_1_result}, we obtain the desired $\mathcal{Z}_2$ bound.
\end{proof}
\renewcommand{\theequation}{B.\arabic{equation}}
\setcounter{equation}{0}
\section{Computation of \texorpdfstring{$\mathcal{Z}_1$}{Z1}}\label{apen : Z1}
We begin by proving Lemma \ref{lem : Z_full_1}. We first recall the lemma. 
\begin{lemma}
Let $\left(\mathcal{Z}_{u,k,j}\right)_{k \in \{1,2\}, j \in \{1,2,3\}}$ be bounds satisfying
\begin{align}
&\mathcal{Z}_{u,1,1} \geq \|\mathbb{1}_{\mathbb{R}^2\setminus\om} \mathbb{L}_{11}^{-1} \mathbb{v}_1^N\|_{2} \\
    &\mathcal{Z}_{u,1,2} \geq \|\mathbb{1}_{\mathbb{R}^2\setminus\om}\mathbb{L}_{22}^{-1}\mathbb{v}_2^N\|_{2}
    \\
    &\mathcal{Z}_{u,1,3} \geq \|\mathbb{1}_{\mathbb{R}^2\setminus\om}\mathbb{L}_{22}^{-1}\mathbb{L}_{21}\mathbb{L}_{11}^{-1}\mathbb{v}_2^N\|_{2}\\
&\mathcal{Z}_{u,2,1} \geq \|\mathbb{1}_{\om} (\mathbb{L}_{11}^{-1} - \Gamma^\dagger(L_{11}^{-1}))\mathbb{v}_1^N\|_{2} \\
    &\mathcal{Z}_{u,2,2} \geq \|\mathbb{1}_{\om} (\mathbb{L}_{22}^{-1} - \Gamma^\dagger(L_{22}^{-1}))\mathbb{v}_2^N\|_{2}
    \\
    &\mathcal{Z}_{u,2,3} \geq \|\mathbb{1}_{\om}(\mathbb{L}_{22}^{-1}\mathbb{L}_{21}\mathbb{L}_{11}^{-1} - \Gamma^\dagger(L_{22}^{-1} L_{21} L_{11}^{-1}))\mathbb{v}_2^N\|_{2}.
\end{align} 
Moreover, given $k \in \{1,2\}$, define  $\mathcal{Z}_{u,k} \bydef \sqrt{(\mathcal{Z}_{u,k,1} + \mathcal{Z}_{u,k,3})^2 + \mathcal{Z}_{u,k,2}^2} $.
Then, it follows that $\mathcal{Z}_{u,1}$ and $\mathcal{Z}_{u,2}$ satisfy
\begin{align} 
\mathcal{Z}_{u,1} &\geq \|\mathbb{1}_{\mathbb{R}^2 \setminus \om} D\mathbb{G}^N(\mathbf{u}_0) \mathbb{L}^{-1}\|_{2}\\
    \mathcal{Z}_{u,2} &\geq \|\mathbb{1}_{\om}D\mathbb{G}^N(\mathbf{u}_0)(\bGam^\dagger(L^{-1}) - \mathbb{L}^{-1})\|_{2}.
\end{align}
Furthermore, let $Z_1$ be a non-negative constant satisfying
\begin{align}
Z_1 &\geq \|I_d - B(I_d + DG^N(\mathbf{U}_0)L^{-1})\|_{2}.
\end{align}
Also define  $\mathcal{Z}_u  \bydef \sqrt{\mathcal{Z}_{u,1}^2 + \mathcal{Z}_{u,2}^2} $.
Let $\varphi$ be defined as in \eqref{definition_of_phi}. Finally, defining $\mathcal{Z}_1>0$ as
\begin{equation}
    \mathcal{Z}_1 \bydef Z_1 + \max\left\{1, \|B_{11}^N\|_{2}\right\} \left(\mathcal{Z}_{u} +  \varphi\left(1,0,\frac{|\lambda_1\lambda_2-1|}{\lambda_2},\frac{1}{\lambda_2}\right) \sqrt{\|V_1^N - V_1\|_{1}^2 + \|V_2^N - V_2\|_{1}^2}\right),
\end{equation}
it follows that $ \|I_d -{\mathbb{A}}D\mathbb{F}(\mathbf{u}_0)\|_{\mathcal{H}} \leq \mathcal{Z}_1.$
\end{lemma}
\begin{proof}
Motivated by the proof of Theorem 3.5 in \cite{unbounded_domain_cadiot}, we proceed similarly. First, using the definition of the norm $\|\cdot\|_{\mathcal{H}}$, we get
\begin{align}
     \|I_d - \mathbb{A}D\mathbb{F}(\mathbf{u}_0)\|_{\mathcal{H}} 
     &= \|I_d - \mathbb{B} ( I_d + D\mathbb{G}(\mathbf{u}_0)\mathbb{L}^{-1})\|_2, \label{ineq : zero step in Zu} \end{align}
and using triangle inequality, we obtain
{\small\begin{align}\label{ineq : first step in Zu}
 \|I_d - \mathbb{B} ( I_d + D\mathbb{G}(\mathbf{u}_0)\mathbb{L}^{-1})\|_2  
   \leq \|I_d - \mathbb{B}(I_d +D\mathbb{G}^N(\mathbf{u}_0)\mathbb{L}^{-1})\|_{2} + \|\mathbb{B}(D\mathbb{G}^N(\mathbf{u}_0) - D\mathbb{G}(\mathbf{u}_0))\mathbb{L}^{-1}\|_{2}.
\end{align}}
For the first term in \eqref{ineq : first step in Zu}, we decompose as follows
{\small\begin{align}
     \|I_d - \mathbb{B}(I_d +D\mathbb{G}^N(\mathbf{u}_0)\mathbb{L}^{-1})\|_{2} &\le \left\|I_d - \mathbb{B}\left(I_d + D\mathbb{G}^N(\mathbf{u}_0)\bGam^\dagger\left(L^{-1}\right) \right)\right\|_2 
     + \left\| \mathbb{B}D\mathbb{G}^N(\mathbf{u}_0)\big(\bGam^\dagger\left(L^{-1}\right) - \mathbb{L}^{-1}\big)\right\|_2
     \label{ineq : second step in Zu}.
\end{align}}
For the first term on the right-hand side of \eqref{ineq : second step in Zu}, the proof of Theorem 3.5 in \cite{unbounded_domain_cadiot} provides 
\begin{align*}
    \left\|I_d - \mathbb{B}(I_d + D\mathbb{G}^N(\mathbf{u}_0)\bGam^\dagger\left(L^{-1}\right) )\right\|_2 = \|I - B(I_d + DG^N(\mathbf{U}_0)L^{-1})\|_{2}.
\end{align*}

For the term $\| \mathbb{B}D\mathbb{G}^N(\mathbf{u}_0)\big(\bGam^\dagger\left(L^{-1}\right) - \mathbb{L}^{-1}\big)\|_2$, we combine \eqref{B_times_DG} and the proof of Theorem 3.5 in \cite{unbounded_domain_cadiot} to get
\begin{align}
   \left\| \mathbb{B}D\mathbb{G}^N(\mathbf{u}_0)\left(\bGam^\dagger\left(L^{-1}\right) - \mathbb{L}^{-1}\right)\right\|_2 \leq \max\left\{1,\|B_{11}^N\|_{2}\right\} \mathcal{Z}_u = \max\left\{1,\|B_{11}^N\|_{2}\right\}\sqrt{\mathcal{Z}_{u,1}^2 + \mathcal{Z}_{u,2}^2}.
\end{align}
Now, we examine $\mathcal{Z}_{u,1}\geq \|\mathbb{1}_{\mathbb{R}^2\setminus \om}D\mathbb{G}^N(\mathbf{u}_0) \mathbb{L}^{-1}\|_{2}.$ To do so,
let $\mathbf{u} = (u_1,u_2) \in L^2$ such that $\|\mathbf{u}\|_{2} = 1$. Then, observe that
\begin{align}
    \nonumber \|\mathbb{1}_{\mathbb{R}^2\setminus \om}\mathbb{L}^{-*}D\mathbb{G}^N(\mathbf{u}_0)^{*}\mathbf{u}\|_{2}
    &=\left\|\mathbb{1}_{\mathbb{R}^2\setminus \om} \begin{bmatrix}
        \mathbb{L}_{11}^{-1}& -\mathbb{L}_{22}^{-1} \mathbb{L}_{21} \mathbb{L}_{11}^{-1} \\
        0 & \mathbb{L}_{22}^{-1}
    \end{bmatrix}\begin{bmatrix}
      \mathbb{v}_1^N & 0 \\ \mathbb{v}_2^N & 0
    \end{bmatrix}\begin{bmatrix}
        u_1 \\ u_2
    \end{bmatrix}\right\|_{2} \\ \nonumber
    &= \left\|\mathbb{1}_{\mathbb{R}^2\setminus \om}\begin{bmatrix}
        \mathbb{L}_{11}^{-1} & -\mathbb{L}_{22}^{-1} \mathbb{L}_{21} \mathbb{L}_{11}^{-1}\\
        0 & \mathbb{L}_{22}^{-1}
    \end{bmatrix}\begin{bmatrix}
        \mathbb{v}_1^Nu_1 \\ \mathbb{v}_2^Nu_1
    \end{bmatrix}\right\|_{2} \\ \nonumber
    &= \left\|\mathbb{1}_{\mathbb{R}^2\setminus \om} \begin{bmatrix}
        \mathbb{L}_{11}^{-1}\mathbb{v}_1^N u_1 -\mathbb{L}_{22}^{-1} \mathbb{L}_{21} \mathbb{L}_{11}^{-1} \mathbb{v}_2^N u_1 \\
        \mathbb{L}_{22}^{-1} \mathbb{v}_2^N u_1
    \end{bmatrix}\right\|_{2} \\ \nonumber
    &\hspace{-1.0cm}\leq \sqrt{\left( \|\mathbb{1}_{\mathbb{R}^2\setminus \om}\mathbb{L}_{11}^{-1} \mathbb{v}_1^N u_1 \|_{2} + \|\mathbb{1}_{\mathbb{R}^2\setminus \om}\mathbb{L}_{22}^{-1} \mathbb{L}_{21} \mathbb{L}_{11}^{-1}\mathbb{v}_2^N u_1\|_{2}\right)^2 + \|\mathbb{1}_{\mathbb{R}^2\setminus \om}\mathbb{L}_{22}^{-1} \mathbb{v}_2^N u_1\|_{2}^2} \\
    &\hspace{-1.0cm}\leq \sqrt{(\mathcal{Z}_{u,1,1}^2 + \mathcal{Z}_{u,1,3}^2)^2+\mathcal{Z}_{u,1,2}^2}.
\end{align}
By performing the same steps for $\mathcal{Z}_{u,2}$, we similarly obtain
\begin{align}
    \|\mathbb{1}_{\om}D\mathbb{G}^N(\mathbf{u}_0)(\bGam^\dagger(L^{-1}) - \mathbb{L}^{-1})\|_{2} \leq \sqrt{(\mathcal{Z}_{u,2,1} + \mathcal{Z}_{u,2,3})^2 + \mathcal{Z}_{u,2,2}^2}.
\end{align}
\par We now return to \eqref{ineq : first step in Zu} and consider the term $\|\mathbb{B}(D\mathbb{G}^N(\mathbf{u}_0) - D\mathbb{G}(\mathbf{u}_0))\mathbb{L}^{-1}\|_{2}$. To begin, observe that
\begin{align}
    \|\mathbb{B}(D\mathbb{G}^N(\mathbf{u}_0) - D\mathbb{G}(\mathbf{u}_0))\mathbb{L}^{-1}\|_{2} &\leq \|\mathbb{B}_{11}\|_{2}\|\mathbb{L}^{-1}\|_{2} \|D\mathbb{G}^N(\mathbf{u}_0) - D\mathbb{G}(\mathbf{u}_0)\|_{2}.
\end{align}
Using Lemma \ref{lem : full_matrix_estimate}, we obtain that
\begin{align}
    \nonumber \|\mathbb{L}^{-1}\|_{2} \leq \varphi(\|\mathbb{L}_{11}^{-1}\|_{2},0, \|\mathbb{L}_{22}^{-1} \mathbb{L}_{21}\mathbb{L}_{11}^{-1}\|_{2},\|\mathbb{L}_{22}^{-1}\|_{2}) &\leq \varphi\left(1, 0, |\lambda_1 \lambda_2 - 1|\|\mathbb{L}_{22}^{-1}\|_{2} \|\mathbb{L}_{11}^{-1}\|_{2}, \frac{1}{\lambda_2}\right) \\ \nonumber
    &\leq \varphi\left(1,0,\frac{|\lambda_1\lambda_2-1|}{\lambda_2},\frac{1}{\lambda_2}\right)
\end{align}
where we used that $|l_{11}| \geq 1, |l_{22}| \geq \frac{1}{\lambda_2}$. Hence, 
{\small\begin{align}
    \nonumber \|\mathbb{B}_{11}\|_{2}\|\mathbb{L}^{-1}\|_{2} \|D\mathbb{G}^N(\mathbf{u}_0) - D\mathbb{G}(\mathbf{u}_0)\|_{2} 
    &\leq \max\{1,\|B_{11}^N\|_{2}\} \varphi\left(1,0,\frac{|\lambda_1\lambda_2-1|}{\lambda_2},\frac{1}{\lambda_2}\right) \left\| \begin{bmatrix}
        \mathbb{v}_1^N - \mathbb{v}_1 & \mathbb{v}_2^N - \mathbb{v}_2 \\
        0 & 0
    \end{bmatrix}\right\|_{2} \\ \nonumber
    &\hspace{-0.7cm}\leq \max\{1,\|B_{11}^N\|_{2}\} \varphi\left(1,0,\frac{|\lambda_1\lambda_2-1|}{\lambda_2},\frac{1}{\lambda_2}\right) \sqrt{\|\mathbb{v}_1^N - \mathbb{v}_1\|_{2}^2 + \|\mathbb{v}_2^N - \mathbb{v}_2\|_{2}^2} \\ \nonumber
    &\hspace{-0.7cm}\leq \max\{1,\|B_{11}^N\|_{2}\} \varphi\left(1,0,\frac{|\lambda_1\lambda_2-1|}{\lambda_2},\frac{1}{\lambda_2}\right) \sqrt{\|V_1^N - V_1\|_{1}^2 + \|V_2^N - V_2\|_{1}^2}
\end{align}}
where the last step followed from Young's inequality \eqref{young_inequality}. This concludes the proof.
\end{proof}
We now proceed with the computation of $Z_1$ in Lemma \ref{lem : Z1_bound}.
\begin{lemma}
Let $M^N \bydef \bpi^N + DG^N(\mathbf{U}_0)L^{-1}$ and $M \bydef \bpi_N + M^N = I_d + DG^N(\mathbf{U}_0)L^{-1}$. Let $\varphi$ be defined as in \eqref{definition_of_phi}. Then, let $Z_1 > 0$ be such that
\begin{equation}
    Z_1 \bydef \varphi(Z_{1,1},Z_{1,2},Z_{1,3},Z_{1,4})
\end{equation}
where 
\begin{align}
&Z_{1,1} \bydef \sqrt{\|(\bpi^N - B^NM^N)(\bpi^N -  (M^N)^*(B^N)^*)\|_{2}} \\
    &Z_{1,2} \bydef \max_{n \in J_{\mathrm{red}}(D_4) \setminus I^N} \sqrt{\left(\frac{\sqrt{\|M_1\|_{2}}}{|l_{11}(\tilde{n})|}  + \frac{|l_{21}(\tilde{n})|\sqrt{\|M_2\|_{2}}}{|l_{22}(\tilde{n}) l_{11}(\tilde{n})|}\right)^2 + \frac{\|M_2\|_{2}}{|l_{22}(\tilde{n})|^2}} \\
    &Z_{1,3} \bydef \sqrt{\|\bpi^NL^{-*}DG^N(\mathbf{U}_0)^{*}\bpi_NDG^N(\mathbf{U}_0)L^{-1}\bpi^N\|_{2}} \\
    &Z_{1,4} \bydef \max_{n \in J_{\mathrm{red}}(D_4)\setminus I^N}\sqrt{\left( \frac{\|V_1^N\|_{1}}{|l_{11}(\tilde{n})|}  + \frac{|l_{21}(\tilde{n})|\|V_2^N\|_{1}}{|l_{22}(\tilde{n}) l_{11}(\tilde{n})|} \right)^2 + \frac{\|V_2^N\|_{1}^2}{|l_{22}(\tilde{n})|^2} } 
\end{align}\par and $M_1,M_2$ are given by
\begin{align*}
    M_1 \bydef   B_{11}^N \mathbb{V}_1^N \pi_N \mathbb{V}_1^N (B_{11}^N)^{*} ~~\text{and}~~
    M_2 \bydef  B_{11}^N \mathbb{V}_2^N \pi_N \mathbb{V}_2^N (B_{11}^N)^{*}.
\end{align*}
Then, $\|I_d - B(I_d + DG^N(\mathbf{U}_0)L^{-1})\|_{2}  = \|I_d - BM\|_{2} \leq Z_1$.
\end{lemma}
\begin{proof}
We begin by examining $I - BM$.
\begin{align}
    \nonumber I - BM &= \begin{bmatrix}
        \bpi^N (I - BM)\bpi^N & \bpi^N (I - BM)\bpi_N \\
        \bpi_N (I - BM)\bpi^N & \bpi_N (I - BM)\bpi_N
    \end{bmatrix} \\ \nonumber
    &= \begin{bmatrix}
        \bpi^N (I - BM)\bpi^N & -\bpi^N BM \bpi_N \\
        \bpi_N (I - M)\bpi^N & \bpi_N (I - M)\bpi_N
    \end{bmatrix} \\
    &= \begin{bmatrix}
         \bpi^N - B^NM^N & -B^N DG^N(\mathbf{U}_0) L^{-1} \bpi_N \\ 
        -\bpi_N DG^N(\mathbf{U}_0) L^{-1}\bpi^N & -\bpi_N DG^N(\mathbf{U}_0) L^{-1} \bpi_N
    \end{bmatrix}.\label{step_in_Z1}
\end{align}
Now, we will use Lemma \ref{lem : full_matrix_estimate} on \eqref{step_in_Z1}. This means we must compute the norm of each of the blocks of \eqref{step_in_Z1}. We begin with $\bpi^N - B^NM^N$.
\begin{align}
     \|\bpi^N - B^NM^N\|_{2}^2 &= \|(\bpi^N - B^NM^N) (\bpi^N - (M^N)^{*} (B^N)^{*})\|_{2} \bydef Z_{1,1}^2.\label{definition_of_Z11}
\end{align}
Next, we examine $- B^N DG(\mathbf{U}_0) L^{-1} \bpi_N$. Using the properties of the adjoint, we get
\begin{align}
    \nonumber \|-B^N DG^N(\mathbf{U}_0) L^{-1} \bpi_N\|_{2}^2    &= \|\bpi_N L^{-*} DG^N(\mathbf{U}_0)^{*} (B^N)^{*}\|_{2}^2 \\ \nonumber 
    &= \left\| \bpi_N \begin{bmatrix}
        L_{11}^{-1} & -L_{22}^{-1} L_{21}L_{11}^{-1} \\
        0 & L_{22}^{-1}
    \end{bmatrix}\begin{bmatrix}
        \mathbb{V}_1^N & 0 \\
        \mathbb{V}_2^N & 0
    \end{bmatrix}\begin{bmatrix}
        (B_{11}^N)^{*} & 0 \\
        (B_{12}^N)^{*} & \pi^N
    \end{bmatrix} \right\|_{2}^2  \\   
    &= \left\| \begin{bmatrix}
        \pi_N L_{11}^{-1} \mathbb{V}_1^N(B_{11}^N)^{*} - \pi_N L_{22}^{-1} L_{21} L_{11}^{-1} \mathbb{V}_2^N (B_{11}^N)^{*} & 0 \\
        \pi_N L_{22}^{-1} \mathbb{V}_2^N (B_{11}^N)^{*} & 0
\end{bmatrix}\right\|_{2}^2.\label{step_before_breaking_blocks_Z1} \end{align}
Let us now examine the term $\pi_N L_{22}^{-1} \mathbb{V}_2^N (B_{11}^N)^{*}$. We have
\begin{align}
    \|\pi_N L_{22}^{-1} \mathbb{V}_2^N (B_{11}^N)^{*}\|_{2}^2 \leq \|\pi_N L_{22}^{-1}\|_{2}^2 \|\pi_N \mathbb{V}_2^N (B_{11}^N)^{*} \|_{2}^2
    &\leq \max_{n \in J_{\mathrm{red}}(D_4) \setminus I^N} \frac{1}{|l_{22}(\tilde{n})|^2} \|B_{11}^N \mathbb{V}_2^N \pi_N \mathbb{V}_2^N (B_{11}^N)^{*} \|_{2} \\
    &\bydef \max_{n \in J_{\mathrm{red}}(D_4) \setminus I^N} \frac{1}{|l_{22}(\tilde{n})|^2} \|M_2\|_{2}.\label{definition_of_M2}
\end{align}
In a similar way, we can estimate the operator norm $\|\pi_N L_{11}^{-1} \mathbb{V}_1^N(B_{11}^N)^{*}  - \pi_N L_{22}^{-1} L_{21} L_{11}^{-1} \mathbb{V}_2^N (B_{11}^N)^{*}\|_{2}$ as follows
{\small\begin{align}
    \nonumber \|\pi_N L_{11}^{-1} \mathbb{V}_1^N(B_{11}^N)^{*}  - \pi_N L_{22}^{-1} L_{21} L_{11}^{-1} \mathbb{V}_2^N (B_{11}^N)^{*}\|_{2}^2
    &\leq \left( \|\pi_N L_{11}^{-1} \mathbb{V}_1^N (B_{11}^N)^{*} \|_{2} + \|\pi_N L_{22}^{-1} L_{21} L_{11}^{-1} \mathbb{V}_2^N (B_{11}^N)^{*} \|_{2}\right)^2 \\ \nonumber
    &\hspace{-1.8cm}\leq ( \|\pi_N L_{11}^{-1}\|_{2} \|\pi_N \mathbb{V}_1^N(B_{11}^N)^{*} \|_{2} + \|\pi_N L_{22}^{-1} L_{21} L_{11}^{-1}\|_{2}\|\pi_N \mathbb{V}_2^N (B_{11}^N)^{*} \|_{2})^2 \\ &\hspace{-1.8cm}\bydef \max_{n \in J_{\mathrm{red}}(D_4) \setminus I^N} \left( \frac{1}{|l_{11}(\tilde{n})|} \sqrt{\|M_1\|_{2}} + \frac{|l_{21}(\tilde{n})|}{|l_{22}(\tilde{n}) l_{11}(\tilde{n})|} \sqrt{\|M_2\|_{2}}\right)^2\label{M1_definition}
\end{align}}
where the last step followed from taking the adjoint. Hence, we get that
{\small\begin{align}
    \| B^NDG^N(\mathbf{U}_0) L^{-1} \bpi_N\|_{2}^2 
    \leq \max_{n \in J_{\mathrm{red}}(D_4) \setminus I^N} \left(\frac{\sqrt{\|M_1\|_{2}}}{|l_{11}(\tilde{n})|}  + \frac{|l_{21}(\tilde{n})|\sqrt{\|M_2\|_{2}}}{|l_{22}(\tilde{n}) l_{11}(\tilde{n})|} \right)^2 + \frac{\|M_2\|_{2}}{|l_{22}(\tilde{n})|^2} 
    &\bydef Z_{1,2}^2.\label{step_in_Z1_2}
\end{align}}
\par Next, we consider the term $-\bpi_N DG^N(\mathbf{U}_0) L^{-1} \bpi^N$.
\begin{align}
    \|-\bpi_N DG^N(\mathbf{U}_0) L^{-1} \bpi^N\|_{2}^2 
    &= \|\bpi^N L^{-*} DG^N(\mathbf{U}_0)^{*} \bpi_N DG^N(\mathbf{U}_0) L^{-1} \bpi^N\|_{2} \bydef Z_{1,3}^2\label{step_in_Z1_3}
\end{align}
Lastly, we treat the block $-\bpi_N DG^N(\mathbf{U}_0) L^{-1} \bpi_N$.
\begin{align}
    \nonumber \|-\bpi_N DG^N(\mathbf{U}_0) L^{-1} \bpi_N\|_{2}^2     &= \|\bpi_N (L^{-1})^{*} DG^N(\mathbf{U}_0)^{*} \bpi_N\|_{2}^2 \\ \nonumber
    &= \left\| \bpi_N \begin{bmatrix}
        L_{11}^{-1} & -L_{22}^{-1} L_{21} L_{11}^{-1} \\
        0 & L_{22}^{-1}
    \end{bmatrix}\begin{bmatrix}
        \mathbb{V}_1^N & 0 \\
        \mathbb{V}_2^N & 0
    \end{bmatrix}\bpi_N\right\|_{2}^2 \\ 
    \nonumber &= \left\| \begin{bmatrix}
        \pi_NL_{11}^{-1}\mathbb{V}_1^N\pi_N - \pi_N L_{22}^{-1} L_{21} L_{11}^{-1} \mathbb{V}_2^N \pi_N & 0 \\
        \pi_N L_{22}^{-1} \mathbb{V}_2^N \pi_N & 0 
    \end{bmatrix}\right\|_{2}^2 \\
    &\hspace{-2.0cm}\leq \max_{n \in J_{\mathrm{red}}(D_4) \setminus I^N} \left( \frac{1}{|l_{11}(\tilde{n})|} \|\mathbb{V}_1^N\|_{2} + \frac{|l_{21}(\tilde{n})|}{|l_{22}(\tilde{n}) l_{11}(\tilde{n})|} \|\mathbb{V}_2^N\|_{2}\right)^2 + \frac{1}{(l_{22})_N^2} \|\mathbb{V}_2^N\|_{2}^2\label{last_term_Z1_before_breaking_blocks}
\end{align}
where we used Lemma \ref{lem : full_matrix_estimate}. Now, we use Young's inequality on \eqref{last_term_Z1_before_breaking_blocks} to obtain
{\footnotesize\begin{align}
    \max_{n \in J_{\mathrm{red}}(D_4) \setminus I^N} \left( \frac{\|\mathbb{V}_1^N\|_{2}}{|l_{11}(\tilde{n})|}  + \frac{|l_{21}(\tilde{n})|\|\mathbb{V}_2^N\|_{2}}{|l_{22}(\tilde{n}) l_{11}(\tilde{n})|} \right)^2 + \frac{\|\mathbb{V}_2^N\|_{2}^2}{(l_{22})_N^2} 
    &\leq \max_{n \in J_{\mathrm{red}}(D_4) \setminus I^N} \left( \frac{\|V_1^N\|_{1}}{|l_{11}(\tilde{n})|}  + \frac{|l_{21}(\tilde{n})|\|V_2^N\|_{1}}{|l_{22}(\tilde{n}) l_{11}(\tilde{n})|} \right)^2 + \frac{\|V_2^N\|_{1}^2}{(l_{22})_N^2} \label{final_simplification_before_step_in_Z1_4}
\end{align}}
where $V_1, V_2$ are defined in \eqref{def : w_0_in_full_equation}. Hence, in total, we obtain that
\begin{align}
    \|\bpi_N DG^N(\mathbf{U}_0) L^{-1} \bpi_N\|_{2}^2 \leq \max_{n \in J_{\mathrm{red}}(D_4)\setminus I^N}\left( \frac{\|V_1^N\|_{1}}{|l_{11}(\tilde{n})|} + \frac{|l_{21}(\tilde{n})|\|V_2^N\|_{1}}{|l_{22}(\tilde{n}) l_{11}(\tilde{n})|} \right)^2 + \frac{\|V_2^N\|_{1}^2}{|l_{22}(\tilde{n})|^2} 
    &\bydef Z_{1,4}^2.\label{step_in_Z1_4}
\end{align}
With \eqref{step_in_Z1}, \eqref{step_in_Z1_2}, \eqref{step_in_Z1_3}, and \eqref{step_in_Z1_4} computed, we are able to compute the $Z_1$ bound.
\end{proof}
Next, we prove Lemma \ref{lem : constants_for_f}. This lemma is necessary to compute the $\mathcal{Z}_u$ bound.
\begin{lemma}
Let $f_{11}, f_{22},$ and $ f_3$ be defined in \eqref{def : f11}, \eqref{def : f22}, and \eqref{def : f3} respectively and let $a_1,a_2$ be defined in \eqref{def : definition of a1 and a2}. Moreover, let $C_0(f_{11})$ and $C_0(f_{22})$ be non-negative constants defined as 
\begin{align}
    C_0(f_{11}) \bydef \max \left\{a_1^2 2e^{\frac{5}{4}}\left(\frac{2}{a_1}\right)^{\frac{1}{4}},~ a_1^2 \sqrt{\frac{\pi}{2\sqrt{a_1}}} \right\} ~~\text{and}~~C_0(f_{22}) \bydef \max \left\{ 2e^{\frac{5}{4}}\left(\frac{2}{a_2}\right)^{\frac{1}{4}},~ \sqrt{\frac{\pi}{2\sqrt{a_2}}} \right\}.
\end{align}
Then, \eqref{def : constant C0f11}, \eqref{def : constant C0f22}, and \eqref{def : C3} are satisfied.
\end{lemma}
\begin{proof}
 Using the Hankel transform tables of \cite{poularikas2018transforms} (Chapter 9), we readily obtain
{\small\begin{align}\label{def : computation of f11 and f22 formulas}
    f_{11}(x) = \mathcal{F}^{-1}\left(\frac{1}{l_{11}(\xi)}\right)(x) = -a_1^2 K_0\left(a_1 |x|\right) ~~\text{and}~~ f_{22}(x) = \mathcal{F}^{-1}\left(\frac{1}{l_{22}(\xi)}\right)(x) = -K_0\left(a_2|x|\right)
\end{align}}
for all $x \in \R^2 \setminus\{0\}$, where $K_0$ is the modified Bessel function of the second kind. 
First, consider $0 < z < \frac{1}{a_j}$ for $j = 1,2$. Using the series expansion of $K_0$ given in \cite{watson_bessel}, we obtain
\begin{align*}
    K_0\left(a_jz\right) =\sum_{k=0}^\infty \frac{\psi(k+1)}{4^k(k!)^2}\left(a_j z\right)^{2k}-\ln\left(\frac{a_j z}{2}\right)\sum_{k=0}^\infty \frac{1}{4^k(k!)^2}\left(a_j z\right)^{2k},\label{K_0_series}
\end{align*}
where $\psi$ is the digamma function. In particular, using the definition of $\psi$, we get $\psi(k+1) = \sum_{n=1}^k\frac{1}{n} - C_{Euler} \leq \frac{1}{2} + \ln(k)$ for all $k \geq 1$, where $C_{Euler} \approx 0.58\dots$ is the Euler–Mascheroni constant.  Consequently, $\frac{|\psi(k+1)|}{k!} \leq \frac{1}{2k} + \frac{\ln(k)}{k!} \leq 1$ for all $k \geq 0$. Therefore, if $0<z<\frac{1}{a_j}$, then
\begin{align*}
    \left|K_0\left(a_jz\right)\right| &\leq  \sum_{k=0}^\infty \frac{|\psi(k+1)|}{4^k(k!)^2} + \left|\ln\left(\frac{a_jz}{2}\right)\right|\sum_{k=0}^\infty \frac{1}{4^k(k!)^2}\\
    &\leq \sum_{k=0}^\infty \frac{1}{4^k k!} + \left|\ln\left(\frac{a_jz}{2}\right)\right|\sum_{k=0}^\infty \frac{1}{4^kk!}\\
    &= e^{\frac{1}{4}} +e^{\frac{1}{4}}\left|\ln\left(\frac{a_j z}{2}\right)\right|\\
    &=e^{\frac{1}{4}}\left(1 + \left|\ln\left(\frac{a_jz}{2}\right)\right|\right).
\end{align*}
Observe that for $0<z<a_j$, we have
\begin{align}
    1 + \left|\ln\left(\frac{a_jz}{2}\right)\right| \leq 2\frac{\left(\frac{2}{a_j}\right)^{\frac{1}{4}}}{z^{\frac{1}{4}}} = 2\left(\frac{2}{a_j}\right)^{\frac{1}{4}} \frac{1}{z^{\frac{1}{4}}}.
\end{align} 
Then, using again that $0 < z < \frac{1}{a_j}$, we have
\begin{align}
    \nonumber 2\left(\frac{2}{a_j}\right)^{\frac{1}{4}}\frac{1}{z^{\frac{1}{4}}} &= 
    2\left(\frac{2}{a_j}\right)^{\frac{1}{4}}e^{-a_j z}e^{a_j z}\frac{1}{z^{\frac{1}{4}}}
     \leq 2e\left(\frac{2}{a_j}\right)^{\frac{1}{4}}\frac{e^{-a_j z}}{z^{\frac{1}{4}}}.
\end{align}
\\ Hence, we obtain 
\begin{align}
   \left|K_0\left(a_j z\right)\right| \leq 2e^{\frac{5}{4}} \left(\frac{2}{a_j}\right)^{\frac{1}{4}}\frac{e^{-a_j z}}{z^{\frac{1}{4}}}.\label{K_0_a_j_bound}
\end{align}
Using \eqref{K_0_a_j_bound}, we obtain 
\begin{align}
    |f_{11}(x)| \leq C_{0,1}(f_{11}) \frac{e^{-a_1 |x|}}{|x|^{\frac{1}{4}}} ~~\text{and}~~|f_{22}(x)| \leq C_{0,1}(f_{22}) \frac{e^{-a_2 |x|}}{|x|^{\frac{1}{4}}}\label{C01f11_def}
\end{align}
where we define
\begin{align}
    \nonumber C_{0,1}(f_{11}) \bydef a_1^2 2e^{\frac{5}{4}}\left(\frac{2}{a_1}\right)^{\frac{1}{4}} ~~\text{and}~~C_{0,1}(f_{22}) \bydef 2e^{\frac{5}{4}}\left(\frac{2}{a_2}\right)^{\frac{1}{4}}.
\end{align}

Next, we fix $z \geq \frac{1}{a_j}$ and using Theorem 3.1 in \cite{yangchu2017inequalities}
\begin{align}
\left|K_0\left(a_j z\right)\right| \leq  \sqrt{\frac{\pi}{2a_j z}}  e^{-a_j z} \leq \sqrt{\frac{\pi}{2\sqrt{a_j}}}\frac{e^{-a_j z}}{z^{\frac{1}{4}}}. 
\end{align}
Consequently, defining $C_{0,2}(f_{11}) \bydef a_1^2 \sqrt{\frac{\pi}{2\sqrt{a_1}}}, C_{0,2}(f_{22}) \bydef \sqrt{\frac{\pi}{2\sqrt{a_2}}}$, and using  \eqref{C01f11_def}, we obtain that
\begin{align}
    \nonumber C_0(f_{11}) \bydef \max(C_{0,1}(f_{11}), C_{0,2}(f_{11}))~~\text{and}~~C_0(f_{22}) \bydef \max(C_{0,1}(f_{22}),C_{0,2}(f_{22}))
\end{align}
provides the desired result. We now consider $\left|f_3(x)\right|$. Recalling \eqref{eq : equality of the lik}, we get
\begin{align}
    \nonumber \frac{l_{21}(\xi)}{l_{11}(\xi) l_{22}(\xi)}
    &=  \frac{1}{|2\pi \xi|^2 + a_1^2} - \frac{1}{|2\pi \xi|^2 + a_2^2}
\end{align}
for all $\xi \in \R^2$, which implies that
\begin{align}
    \nonumber \mathcal{F}^{-1}\left(\frac{l_{21}(\xi)}{l_{11}(\xi) l_{22}(\xi)}\right)(x)
    &= K_{0}\left(a_2 |x|\right) - K_{0}\left(a_1 |x|\right). 
\end{align}
We now separate two cases. If $\frac{1}{\lambda_1} \geq \lambda_2$, then using that $K_0$ is a decreasing function, we have that
\begin{align} \left|K_0(a_2 |x|) - K_0\left(a_1|x|\right)\right| \leq   |K_0(a_2|x|)| = |f_{22}(x)| \leq C_0(f_{22}) \frac{e^{-a_2|x|}}{|x|^{\frac{1}{4}}},\label{f_3_bound1}
\end{align}
for all $x \in \R^2\setminus \{0\}.$
If instead $\frac{1}{\lambda_1} \leq \lambda_2$, then 
\begin{align}
    \left|K_0(a_2 |x|) - K_0\left(a_1|x|\right)\right| \leq   \left|K_0\left(a_1|x|\right)\right| = \frac{1}{a_1^2}|f_{11}(x)| \leq \frac{1}{a_1^2}C_0(f_{11}) \frac{e^{-a_1|x|}}{|x|^{\frac{1}{4}}},
\end{align}
which concludes the proof.
\end{proof} 
Finally, we compute the $\mathcal{Z}_{u,1}$ bound defined in Lemma \ref{lem : lemma Z11full}.
\begin{lemma}
Let $a_1, a_2$ be given in \eqref{def : definition of a1 and a2} and let
 $$E_j \bydef \gamma\left(\mathbb{1}_{\om}(x) \left(\frac{1}{2}(\cosh(2a_jx_1) + \cosh(2a_jx_2)\right)\right)$$ for $j = 1,2$. Moreover, let $\left(\mathcal{Z}_{u,1,j}\right)_{j \in \{1,2,3\}}$ be bounds defined as 
\begin{align}
&\mathcal{Z}_{u,1,1} \bydef \frac{\sqrt{2}C_0(f_{11})(2\pi)^{\frac{1}{4}}e^{-a_1d}\sqrt{|\om|}}{a_1^{\frac{3}{4}}}\sqrt{(V_1^N,V_1^N* E_1)_2} \\
    &\mathcal{Z}_{u,1,2} \bydef \frac{\sqrt{2}C_0(f_{22})(2\pi)^{\frac{1}{4}}e^{-a_2d}\sqrt{|\om|}}{a_2^{\frac{3}{4}}} \sqrt{(V_2^N,V_2^N * E_2)_2} \\
    &\mathcal{Z}_{u,1,3} \bydef \begin{cases}
        \mathcal{Z}_{u,1,2} & \mathrm{\ if \ } \frac{1}{\lambda_1} \geq \lambda_2 \\
        \frac{\sqrt{2}\lambda_2C_0(f_{11})(2\pi)^{\frac{1}{4}}e^{-a_1d}\sqrt{|\om|}}{a_1^{\frac{3}{4}}}\sqrt{(V_2^N,V_2^N* E_1)_2} & \mathrm{\ if \ } \frac{1}{\lambda_1} \leq \lambda_2
    \end{cases}
\end{align}
where $V_1^N,V_2^N$ are given in \eqref{def : w_0_in_full_equation}. Then $ \mathcal{Z}_{u,1} =  
    \sqrt{\left( \mathcal{Z}_{u,1,1} + \mathcal{Z}_{u,1,3}\right)^2 + \mathcal{Z}_{u,1,2}^2}$ satisfies \eqref{def : Zu1}. That is, $\|\mathbb{1}_{\mathbb{R}^2 \setminus \om} D\mathbb{G}^N(\mathbf{u}_0)\mathbb{L}^{-1}\|_{2} \leq \mathcal{Z}_{u,1}.$ 
\end{lemma}

\begin{proof} 
Observe that it is sufficient to provide an upper bound for the terms on the right-hand side of \eqref{eq : new definition of Zu1 in equation}.
Starting with ${\|\mathbb{1}_{\mathbb{R}^2 \setminus \om} (f_{11}^2 * (v_1^N)^2)\|_{1}}$ and using that $v_1^N \in L^2_{D_4}(\mathbb{R}^2)$, we get
\begin{align}
    \|\mathbb{1}_{\mathbb{R}^2 \setminus \om}  (f_{11}^2 * (v_1^N)^2)\|_{1}
    &= \int_{\mathbb{R}^2 \setminus \om} \int_{\om} f_{11}(y-x)^2 v_1^N(x)^2 dx dy \\ 
    &\leq 2\left[ \int_{\mathbb{R}} \int_{d}^{\infty} \int_{\om} f_{11}(y-x)^2 v_1^N(x)^2 dx dy + \int_{d}^{\infty} \int_{\mathbb{R}} f_{11}(y-x)^2 v_1^N(x)^2 dx dy\right] \\
    &\leq 4 \int_{\mathbb{R}} \int_{d}^{\infty} \int_{\om} f_{11}(y-x)^2 v_1^N(x)^2 dx dy.
\end{align}
 We now use Lemma \ref{lem : constants_for_f} and Fubini's theorem to obtain
\begin{align}
    4 \int_{\mathbb{R}} \int_{d}^{\infty} \int_{\om} f_{11}(y-x)^2 v_1^N(x)^2 dx dy &\leq 4C_0(f_{11})^2  \int_{\mathbb{R}} \int_{d}^{\infty} \int_{\om}\frac{e^{-2a_1|y-x|_1}}{\sqrt{|y-x|_1}} v_1^N(x)^2dx dy \\ 
    &= 4C_0(f_{11})^2 \int_{\om} v_1^N(x)^2\int_{\mathbb{R}} \int_{d}^{\infty} \frac{e^{-2a_1|y-x|_1}}{\sqrt{|y-x|_1}} dy dx.\label{eq : cauchy S for E2}
\end{align}
But now notice that 
\begin{align}\label{eq : computation exponential decay 2}
 \nonumber
     \int_{\mathbb{R}} \int_{d}^\infty  \frac{e^{-2a_1|y-x|_1}}{\sqrt{|y-x|_1}}dy_1dy_2
     = & \int_{\mathbb{R}} \int_{d}^\infty  \frac{e^{-2a_1(|y_1-x_1|+|y_2-x_2|)}}{\sqrt{|y_1-x_1|+|y_2-x_2|}}dy_1dy_2 \\ \nonumber 
     \leq & \int_{d}^\infty e^{-2a_1(y_1-x_1)}dy_1 \int_{\mathbb{R}}\frac{e^{-2a_1|y_2-x_2|}}{\sqrt{|y_2 - x_2|}}dy_2\\ 
     = &\sqrt{\frac{\pi}{2a_1^3}} e^{2a_1x_1}
\end{align}
for all $x \in \om.$
Therefore, combining \eqref{eq : cauchy S for E2} and \eqref{eq : computation exponential decay 2}, we get
\begin{align}
4C_0(f_{11})^2\int_{\om} v_1^N(x)^2\int_{\mathbb{R}} \int_{d}^{\infty} \frac{e^{-2a_1|y-x|_1}}{\sqrt{|y-x|_1}} dy dx \leq \frac{2\sqrt{2\pi}e^{-2a_1d}}{\sqrt{a_1^3}}C_0(f_{11})^2 \int_\om v_1^N(x)^2e^{2a_1x_1}dx.\label{middle_estimate_Zu1}
\end{align}
Since we are assuming $D_4$-symmetry, observe that for $k = 1,2$
\begin{align}
    &\nonumber \int_{\om} v_1^N(x)^2 e^{ 2a_1x_k} dx = \int_{\om} v_1^N(x)^2 e^{-2a_1x_k} dx = \int_{\om} v_1^N(x)^2 \cosh(2a_1x_k) dx \\
    &=
    \frac{1}{2}\int_{\om} v_1^N(x)^2(\cosh(2a_1 x_1)+\cosh(2a_1 x_2)) dx.\label{step_in_Zu}
\end{align}
In particular, $e_1 : x \mapsto \frac{1}{2}\left(\cosh(2a_1 x_1)+\cosh(2a_1 x_2)\right)$ is $D_4$-symmetric, hence its Fourier coefficients in $\ell^2(J_{\mathrm{red}}(D_4))$, denoted $E_1 \bydef \gamma(\cha e_1)$, are well-defined.   
Now, using Parseval's identity, we can write \eqref{step_in_Zu} as
\begin{align}
 \frac{1}{2}\int_{\om} v_1^N(x)^2(\cosh(2a_1 x_1)+\cosh(2a_1 x_2)) dx
    &= |\om|(V_1^N,V_1^N* E_1)_2.\label{inner_prod_def} 
\end{align}
 Combining \eqref{middle_estimate_Zu1} and \eqref{inner_prod_def}, we prove the desired result for $\mathcal{Z}_{u,1,1}$. One can then prove similar results for $\mathcal{Z}_{u,1,2}$ and $\mathcal{Z}_{u,1,3}$, which concludes the proof of the lemma.
\end{proof}
\renewcommand{\theequation}{C.\arabic{equation}}
\setcounter{equation}{0}
\section{Integral computation in the proof of Lemma \ref{lem : lemma Z12full}}\label{apen : AppendixZu2}
In this section, we provide the technical analysis required to prove Lemma \ref{lem : lemma Z12full}. 
\begin{lemma}\label{lem : first_int_Z_u_2}
Given $i, j \in \{1,2\}$, define $w_i(x) = v_i^N(x)u(x)$ where $u \in L^2_{D_4}(\mathbb{R}^2)$ such that $\|u\|_{2} = 1$. Let $\mathcal{C}_{1,j},\mathcal{C}_{2,j},$ and $C(v_j^N)$ be non-negative constants defined as in Lemma \ref{lem : lemma Z12full}.
Then, 
    \begin{align}
        &\int_{-d}^d |\partial_{x_1} \mathbb{L}_{jj}^{-1} w_i(d,x_2)| dx_2 \leq \mathcal{C}_{1,j} \sqrt{(V_i^N,E_{j} * V_i^N)_2} + \mathcal{C}_{2,j}C(v_j^N)\label{L11_Z_u2_comp}
        \end{align}
where $E_j$ is defined in Lemma \ref{lem : lemma Z11full}.
\end{lemma}
\begin{proof}
In Lemma \ref{lem : constants_for_f}, we computed explicit constants for the exponential decay of $f_{11}$ and $f_{22}$. We begin by doing the same for $\partial_{x_1} f_{11}$ and $\partial_{x_1} f_{22}$. Let $a_1$ and $a_2$ be defined as in \eqref{def : definition of a1 and a2}. Let $C_{11}(f_{11}), C_{12}(f_{11}), C_{11}(f_{22})$, and $ C_{12}(f_{22})$ be non-negative constants defined as in Lemma \ref{lem : lemma Z12full}. Then, we have
\begin{align}\label{eq : inequality for the derivatives of f11 and f22}
 \hspace{-0.5cm}|\partial_{x_1}f_{11}(x)| \leq \begin{cases}
     C_{11}(f_{11}) e^{-a_1|x|} & \mathrm{if \ } |x| > 1 \\
     C_{12}(f_{11}) \frac{e^{-a_1|x|}}{|x|} & \mathrm{if \ } |x| < \sqrt{2} 
    \end{cases},
|\partial_{x_1}f_{22}(x)| \leq \begin{cases}
     C_{11}(f_{22}) e^{-a_2|x|} & \mathrm{if \ } |x| > 1 \\
     C_{12}(f_{22}) \frac{e^{-a_2|x|}}{|x|} & \mathrm{if \ } |x| < \sqrt{2} 
    \end{cases}
\end{align}
\\ for all $x \in \mathbb{R}^2 \setminus \{0\}.$ Let $j \in \{1,2\}$ and recall that
\[
\partial_{x_1} K_0(a_j|x|) = - \frac{a_jx_1}{|x|} K_1(a_j|x|) 
\]
for all $x \in \R^2  \setminus \{0\}$. Therefore, using \eqref{def : computation of f11 and f22 formulas}, we get
\begin{align*}
    \left|\partial_{x_1}f_{11}(x)\right| &= a_1^3\left|\frac{x_1}{|x|} K_1(a_1|x|)\right| \leq  a_1^3|K_1(a_1|x|)| \\
    \left|\partial_{x_1}f_{22}(x)\right| &= a_2\left|\frac{x_1}{|x|} K_1(a_2|x|)\right| \leq  a_2|K_1(a_2|x|)|
    \end{align*}
for all $x \in \R^2 \setminus \{0\}$,  where we used that $K_1$ is a decreasing function on the last step. Let us now consider the case $|x| > 1$.
Using Corollary 3.4 from \cite{yangzheng2017inequalities}, for $z > 1$, we get
\begin{align}
     |K_1(a_j z)| \leq \sqrt{\frac{\pi}{2}} \frac{1}{\sqrt{a_j z+1}}\left(1 + \frac{1}{a_j z}\right) e^{-a_j z} &\leq \sqrt{\frac{\pi}{2}} \frac{1}{\sqrt{a_j + 1}}\left(1 + \frac{1}{a_j}\right) e^{-a_j z}.
\end{align}
Hence, we obtain
\begin{align}
    &a_1^3 \left| K_1\left(a_1 |x|\right)\right| \leq a_1^3 \sqrt{\frac{\pi}{2}}\frac{1}{\sqrt{a_1  + 1}} \left(1 + \frac{1}{a_1}\right) e^{-a_1|x|} \bydef C_{11}(f_{11}) e^{-a_1|x|} \\
    &a_2 \left| K_1\left(a_2|x|\right)\right| \leq a_2 \sqrt{\frac{\pi}{2}}\frac{1}{\sqrt{a_2  + 1}} \left(1 + \frac{1}{a_2 }\right) e^{-a_2|x|} \bydef C_{11}(f_{22}) e^{-a_2 |x|}\label{C11_away_def}
\end{align}
for all $|x| > 1$.
Now, for $0 < z < \sqrt{2}$, notice that
\begin{align}
    \sqrt{\frac{\pi}{2}}\frac{1}{\sqrt{a_j z + 1}}\left(1 + \frac{1}{a_j z}\right) e^{-a_j z}  \leq \sqrt{\frac{\pi}{2}} \left(1 + \frac{1}{a_j z}\right) e^{-a_j z} 
    &\leq \sqrt{\frac{\pi}{2}} \frac{\sqrt{2}a_j + 1}{a_j } \frac{e^{-a_j z}}{z},
\end{align}
which implies that
\begin{align}
    a_1^3 \left| K_1\left(a_1|x|\right)\right| &\leq a_1^2 \sqrt{\frac{\pi}{2}} \left(\sqrt{2}a_1 + 1\right) \frac{e^{-a_1|x|}}{|x|} \bydef C_{12}(f_{11}) \frac{e^{-a_1|x|}}{|x|} \\
    a_2\left| K_1\left(a_2|x|\right)\right| &\leq \sqrt{\frac{\pi}{2}} \left(\sqrt{2}a_2 + 1 \right) \frac{e^{-a_2|x|}}{|x|} \bydef C_{12}(f_{22}) \frac{e^{-a_2|x|}}{|x|}\label{C12_near_def}
\end{align}
for all $0<|x| < \sqrt{2}.$ This proves equation \eqref{eq : inequality for the derivatives of f11 and f22}. 
\par Now, we go back to the study of $\int_{-d}^d |\partial_{x_1} \mathbb{L}_{jj}^{-1} w_i(d,x_2)| dx_2$. We only consider the case $j = i = 1$ as the other cases follow similarly. 
To begin, recall that $\mathbb{L}^{-1}_{11}v = f_{11}*v$ for all $v \in L^2(\R^2)$ and observe that
\begin{align}
     &\int_{-d}^d    |\partial_{x_1} \mathbb{L}_{11}^{-1}w_1(d,x_2)|dx_2 = \int_{-d}^d    \left| \int_{\Omega_0} \partial_{x_1}f_{11}(y_1 - d,y_2 - x_2) w_1(y_1,y_2)dy_1 dy_2 \right| dx_2 \\ 
    & \leq \int_{\om} \int_{-d}^d |w_1(y_1,y_2)| |\partial_{x_1} f_{11}(y_1 - d, y_2 - x_2)| dx_2 dy_1 dy_2.\label{first_int_Zu2}\end{align}
Now, define $\mathcal{D}_0 \bydef (d-1,d)$ and $\Omega_1 \bydef (\om \times (-d,d)) \setminus \mathcal{D}_0^3$. Then, we decompose \eqref{first_int_Zu2} as 
\begin{align}
    \int_{\om} \int_{-d}^d |w_1(y_1,y_2)| |\partial_{x_1} f_{11}(y_1 - d, y_2 - x_2)| dx_2 dy_1 dy_2 \bydef I_1 + I_2
\end{align}
where
\begin{align}
    &I_1 \bydef \int_{\mathcal{D}_0^3}  |w_1(y_1,y_2)| |\partial_{x_1} f_{11}(y_1 - d, y_2 - x_2)| dx_2 dy_1 dy_2 \\
    &I_2 \bydef \int_{\Omega_1}  |w_1(y_1,y_2)| |\partial_{x_1} f_{11}(y_1 - d, y_2 - x_2)| dx_2 dy_1 dy_2. 
    \end{align}
 Beginning with $I_2$, observe that for  $(y_1,y_2,x_2) \in \Omega_1$, it follows that
\begin{align}
    1 \leq \sqrt{|y_1 - d|^2 + |y_2 - x_2|^2}.
\end{align}
Now, using \eqref{eq : inequality for the derivatives of f11 and f22}, we obtain
\begin{align}
    I_2 &\leq C_{12}(f_{11}) \int_{\Omega_1}  |w_1(y_1,y_2)|e^{-a_1(|y_1 - d| + |y_2 - x_2|)} dx_2 dy_1 dy_2,\label{step_C12f11}
\end{align}
and  since $\Omega_1 \subset \om \times (-d,d)$, it follows that
\begin{align}
    I_2 &\leq C_{12}(f_{11}) \int_{\om} |w_1(y_1,y_2)| \int_{-d}^{d} e^{-a_1(|y_1 - d| + |y_2 - x_2|)} dx_2 dy_1 dy_2 \\
    \nonumber
    &\leq e^{-a_1 d} C_{12}(f_{11}) \int_{\om}  |w_1(y_1,y_2)|e^{ay_1} \int_{-\infty}^{\infty} e^{-a_1|y_2 - x_2|} dx_2 dy_1 dy_2.
\end{align}
Now, observe that
\begin{align}
     \int_{-\infty}^{\infty} e^{-a_1 |y_2 - x_2|} dx_2 &= 2 \int_{0}^{\infty} e^{-a_1 z} dz = \frac{2}{a_1}.\label{exp_int_value}
\end{align}
Recall that $w_1 = v_1^N u$ for $u \in L^2_{D_4}(\mathbb{R}^2)$ such that $\|u\|_{2} = 1$. Hence, we can apply Cauchy-Schwarz inequality to obtain
\begin{align}
    \int_{\om} |w_1(y_1,y_2)| e^{ay_1} dy_1 = \int_{\om} |u(y_1,y_2)||v_1^N(y_1,y_2)|e^{ay_1} dy_1 
    &\leq \left(\int_{\om} v_1^N(y)^2 e^{2ay_1} dy_1 \right)^{\frac{1}{2}}\label{cs_on_w_1}
\end{align}
Then, we use Parseval's identity to get
\begin{align}
    \frac{2}{a_1} e^{-a_1 d} C_{12}(f_{11}) \left(\int_{\om} v_1^N(y)^2 e^{2ay_1} dy\right)^{\frac{1}{2}} 
    &= \frac{2}{a_1} e^{-a_1 d} C_{12}(f_{11}) \sqrt{|\om|}\sqrt{(V_1^N, E_1 * V_1^N)_2}.\label{int_away_from_0}
\end{align}
Combining \eqref{exp_int_value},\eqref{cs_on_w_1} and \eqref{int_away_from_0}, we obtain
\begin{align}
    \nonumber I_2
    &\leq \frac{2\sqrt{|\om|}}{a_1} e^{-a_1 d} C_{12}(f_{11}) \sqrt{(V_1^N, E_1 * V_1^N)_2}.
\end{align}
Next, we consider $I_1$. Observe that for $(y_1,y_2,x_2) \in \mathcal{D}_0^3$, we have
\begin{align}
    0 \leq \sqrt{|y_1 - d|^2 + |y_2 - x_2|^2} \leq \sqrt{2}.
\end{align}
Hence, we can use \eqref{eq : inequality for the derivatives of f11 and f22} to obtain
    \begin{align}
    I_1 &\leq C_{11}(f_{11}) \int_{\mathcal{D}_0^3}  |w_1(y_1,y_2)| \frac{e^{-a_1(|y_1 - d| + |y_2 - x_2|)}}{|y_1-d| + |y_2 - x_2|} dx_2 dy_1 dy_2.\label{step_C11f11}
    \end{align}
Then, it yields
\begin{align}
    \nonumber I_1 &= C_{11}(f_{11})     \int_{\mathcal{D}_0^2}  |w_1(y_1,y_2)| \int_{\mathcal{D}_0} \frac{e^{-a_1(|y_1-d| + |y_2 - x_2|)}}{|y_1-d| + |y_2 - x_2|} dx_2 dy_1 dy_2 \\ \nonumber
    &\leq C_{11}(f_{11}) \int_{\mathcal{D}_0^2}  |w_1(y_1,y_2)|e^{-a_1|y_1 -d|}   \int_{-\infty}^{\infty} \frac{e^{-a_1|y_2 - x_2|}}{|y_1-d| + |y_2 - x_2|} dx_2 dy_1 dy_2 \\
    &= e^{-a_1d} C_{11}(f_{11}) \int_{\mathcal{D}_0^2}  |w_1(y_1,y_2)|e^{a_1y_1}   I_3(y_1,y_2) dy_1 dy_2\label{Z_u_2_with_I_1}
\end{align}
where we used that $\mathcal{D}_0 \subset \mathbb{R}$, $y_1 \in \mathcal{D}_0$ implies that $|y_1 - d| = d - y_1$, and 
\begin{equation}\label{I_1_definition}
    I_3(y_1,y_2) \bydef \int_{-\infty}^{\infty} \frac{e^{-a_1|y_2 - x_2|}}{(d-y_1) + |y_2 - x_2|} dx_2.
\end{equation} Now, given $y_1,y_2 \in \mathcal{D}_0$, we compute $I_3(y_1,y_2)$ as follows
\begin{align}
    \nonumber I_3(y_1,y_2)  = \int_{-\infty}^{\infty} \frac{e^{-a_1|y_2-x_2|}}{(d-y_1)+|y_2-x_2|} dx_2 
    &=  \int_{-\infty}^{\infty} \frac{e^{-a_1|t|}}{(d-y_1) + |t|} dt \\ \nonumber
    &= 2 \left( \int_{0}^{1} \frac{e^{-a_1y}}{(d-y_1) + t} dt + \int_{1}^{\infty} \frac{e^{-a_1t}}{(d-y_1) + t} dx\right)\\ \nonumber
    &\leq 2  \left( \int_{0}^{1} \frac{1}{(d-y_1) + t} dt + \int_{1}^{\infty} e^{-a_1t} dt\right) \\ 
    &= 2 \left( \ln(d+1-y_1) - \ln(d-y_1) +  \frac{e^{-a_1}}{a_1} \right).\label{I_1_val_end}
    \end{align}
Returning to \eqref{Z_u_2_with_I_1}, we use \eqref{I_1_val_end} to obtain
\begin{equation}\label{I_2_definition}
        I_1  \leq C_{11}(f_{11})e^{-a_1d} \left(I_4 + I_5\right)
\end{equation}
where
\begin{equation}\label{I_4_definition}
    I_4 \bydef 2\int_{\mathcal{D}_0^2} |w_1(y_1,y_2)|e^{a_1y_1} (\ln(d+1-y_1) - \ln(d-y_1)) dy_1 dy_2
\end{equation}
and
\begin{equation}\label{I_5_definition}
    I_5 \bydef \frac{2e^{-a_1}}{a_1} \int_{\mathcal{D}_0^2} |w_1(y_1,y_2)|e^{a_1y_1} dy_1 dy_2.
\end{equation}
To compute $I_5$, we use that $\mathcal{D}_0^2 \subset \om$ and again Cauchy-Schwarz on $w_1$ (cf. \eqref{cs_on_w_1}) to obtain
\begin{align}
    \nonumber I_5 \leq \frac{2e^{-a_1}}{a_1}\int_{\om} |w_1(y_1,y_2)| e^{a_1y_1} dy_1 dy_2
    &\leq \frac{2e^{-a_1}}{a_1}  \left(\int_{\om} v_1^N(y)^2 e^{2a_1y_1} dy_1 dy_2 \right)^{\frac{1}{2}} \\
    &= \frac{2e^{-a_1}}{a_1}  \sqrt{|\om|}\sqrt{(V_1^N,E_{1} * V_1^N)_2}.\label{first_int_near_singularity}
\end{align}
Next, we compute $I_4$. We begin again with Cauchy-Schwarz to obtain
\begin{align}
    \nonumber I_4 & \leq 2 \left(\int_{\mathcal{D}_0^2}^{d} v_1^N(y)^2 e^{2a_1y_1} (\ln(d+1-y_1) - \ln(d-y_1))^2 dy_1 dy_2 \right)^{\frac{1}{2}} \\ \nonumber
    &\leq 2 \left(\int_{\mathcal{D}_0} \int_{\mathcal{D}_0} \left(\sup_{y_1 \in \mathcal{D}_0} v_1^N(y)^2\right) e^{2a_1y_1} (\ln(d+1-y_1) - \ln(d-y_1))^2 dy_1 dy_2 \right)^{\frac{1}{2}} \\ \nonumber
    &= 2\left(\sup_{y_1 \in \mathcal{D}_0}\int_{\mathcal{D}_0} v_1^N(y)^2dy_2\right)^{\frac{1}{2}} \left( \int_{\mathcal{D}_0}e^{2a_1y_1} (\ln(d+1-y_1) - \ln(d-y_1))^2 dy_1 \right)^{\frac{1}{2}}. 
\end{align}
Let us now concentrate on the term $\left( \int_{\mathcal{D}_0}e^{2a_1y_1} (\ln(d+1-y_1) - \ln(d-y_1))^2 dy_1 \right)^{\frac{1}{2}}$.
\begin{align} 
\hspace{-1.0cm}\int_{\mathcal{D}_0}e^{2a_1y_1} (\ln(d+1-y_1) - \ln(d-y_1))^2 dy_1  \leq e^{2a_1 d} \int_{\mathcal{D}_0}(\ln(d+1-y_1) - \ln(d-y_1))^2 dy_1.\label{before_log_steps}
\end{align}
Now, observe that $\ln(d + 1 - y_1) - \ln(d - y_1) \geq 0$. Then, since $\ln(d + 1 - y_1) \leq \ln(2)$, we obtain 
\begin{align}
    0\leq \ln(d + 1 - y_1) - \ln(d - y_1) \leq \ln(2) - \ln(d - y_1) .\label{log_steps}
\end{align}
Since \eqref{log_steps} is positive, we can square it and obtain 
\begin{align}
    0\leq \left(\ln(d + 1 - y_1) - \ln(d - y_1)\right)^2 \leq \left(\ln(2) - \ln(d - y_1)\right)^2 .\label{log_steps_2}
\end{align}
Using \eqref{log_steps_2}, we return to \eqref{before_log_steps} and see that
\begin{align} 
 \nonumber e^{2a_1 d}\int_{\mathcal{D}_0} (\ln(d+1-y_1) - \ln(d-y_1))^2 dy_1  
    &\leq  e^{2a_1d}\int_{\mathcal{D}_0}(\ln(2) - \ln(d-y_1))^2 dy_1  \\ \nonumber
    &\hspace{-1.6cm}=   e^{2a_1d}\int_{d-1}^{d}  \ln(2)^2 -2 \ln(2)\ln(d-y_1) + \ln(d-y_1)^2 dy_1
      \\ \nonumber
    &\hspace{-1.6cm}= e^{2a_1d}  \left( \ln(2)^2 +2 \ln(2) + 2\right).
\end{align}
Let us now examine the term $\left(\sup_{y_1 \in \mathcal{D}_0}\int_{\mathcal{D}_0} v_1^N(y)^2dy_2\right)^{\frac{1}{2}}$. By the second fundamental theorem of Calculus, we obtain
{\small\begin{align}
    |v_1^N(y)^2| = \left| \frac{1}{2}\int_{y_1}^d  v_1^N(t,y_2)\partial_{x_1} v_1^N(t,y_2) dt - v^N_1(d,y_2)^2\right| \leq \frac{1}{2}\int_{\mathcal{D}_0} |v_1^N(t,y_2)\partial_{x_1} v_1^N(t,y_2)| dt + v^N_1(d,y_2)^2
\end{align}}
where we used that $y_1 \in \mathcal{D}_0$. Then, using Cauchy Schwarz, we obtain
\begin{align}\label{eq : estimation on the boundary}
  \int_{\mathcal{D}_0} |v_1^N(t,y_2)\partial_{x_1} v_1^N(t,y_2)| dt\leq \left( \int_{\mathcal{D}_0}| v_1^N(t,y_2)|^2dt\right)^{\frac{1}{2}} \left( \int_{\mathcal{D}_0}|\partial_{x_1} v_1^N(t,y_2)|^2dt\right)^{\frac{1}{2}}.
\end{align} Consequently, using Cauchy-Schwaz again, we get
\begin{align}
    \int_{\mathcal{D}_0} &\left( \int_{\mathcal{D}_0}| v_1^N(t,y_2)|^2dt\right)^{\frac{1}{2}} \left( \int_{\mathcal{D}_0}|\partial_{x_1} v_1^N(t,y_2)|^2dt\right)^{\frac{1}{2}} dy_2\\
    &\leq \left(  \int_{\mathcal{D}_0^2}| v_1^N(t,y_2)|^2dtdy_2\right)^{\frac{1}{2}} \left(  \int_{\mathcal{D}_0^2}|\partial_{x_1} v_1^N(t,y_2)|^2dtdy_2\right)^{\frac{1}{2}}\\
    &= \|\mathbb{1}_{\mathcal{D}_0^2} v_1^N\|_2\|\mathbb{1}_{\mathcal{D}_0^2}\partial_{x_1} v_1^N\|_2,
\end{align}
where we used that $\mathcal{D}_0^2 = (d-1,d)^2$ on the last step.
Finally, using \eqref{eq : estimation on the boundary}, this implies that
\begin{align}
    \sup_{y_1 \in \mathcal{D}_0} \int_{\mathcal{D}_0}\left|v_1^N(y)\right|^2 dy_2 \leq   \frac{1}{2}\|\mathbb{1}_{\mathcal{D}_0^2}\partial_{x_1} v_1^N\|_2\|\mathbb{1}_{\mathcal{D}_0^2} v_1^N\|_2 + \|\mathbb{1}_{\mathcal{D}_0} v_1^N(d,\cdot)\|_2^2.\label{I_4_val_end}
\end{align}
Hence, we obtain
\begin{align}
    I_4 &\leq 2e^{a_1d} ( \ln(2)^2 + 2\ln(2) +2)^{\frac{1}{2}}\left(\sup_{y_1 \in \mathcal{D}_0} \int_{\mathcal{D}_0}\left|v_1^N(y)\right|^2 dy_2\right)^{\frac{1}{2}} \\
    &\leq e^{a_1d} \frac{\mathcal{C}_{2,1}}{\sqrt{|\om|}C_{11}(f_{11})} \left( \frac{1}{2}\|\mathbb{1}_{\mathcal{D}_0^2}\partial_{x_1} v_1^N\|_2\|\mathbb{1}_{\mathcal{D}_0^2} v_1^N\|_2 + \|\mathbb{1}_{\mathcal{D}_0} v_1^N(d,\cdot)\|_2^2\right)^{\frac{1}{2}} \\
    &= e^{a_1d} \frac{\mathcal{C}_{2,1}}{C_{11}(f_{11})} C(v_1^N).\label{second_int_near_singularity}
\end{align}
Using \eqref{first_int_near_singularity} and \eqref{second_int_near_singularity}, we have an expression for \eqref{Z_u_2_with_I_1}.
\begin{align}
I_1 &\leq 2\sqrt{|\om|} e^{-a_1 d}C_{11}(f_{11})\frac{e^{-a_1}}{a_1} \sqrt{(V_1^N,E_{1} * V_1^N)_2} +\mathcal{C}_{2,1}C(v_1^N)\label{integral_x2_z12}
\end{align}
Combining \eqref{int_away_from_0} and \eqref{integral_x2_z12}, we have
\begin{align}
    \nonumber \int_{-d}^d |\partial_{x_1} \mathbb{L}_{11}^{-1} w_1(d,x_2)| dx_2 &= I_1 + I_2  \\
    &\hspace{-1.5cm}\leq 2\sqrt{|\om|} e^{-a_1 d}\frac{C_{11}(f_{11})e^{-a_1} + C_{12}(f_{11})}{a_1} \sqrt{(V_1^N, E_1 * V_1^N)_2} + \mathcal{C}_{2,1} C(v_1^N) \\ \nonumber
    &\hspace{-1.5cm}= \mathcal{C}_{1,1} \sqrt{(V_1^N, E_1 * V_1^N)_2} + \mathcal{C}_{2,1} C(v_1^N).
\end{align}
This concludes the proof of Lemma \ref{lem : first_int_Z_u_2}. 
\end{proof}
\newpage
\nocite{ferranti2021interval}
\nocite{julia_cadiot_Kawahara}
\nocite{julia_cadiot_sh}
\nocite{julia_cadiot_blanco_GS}
\nocite{julia_interval}
\bibliographystyle{abbrv}
\bibliography{biblio}
\end{document}